\documentclass[11pt]{article}
\usepackage[latin1]{inputenc}
\usepackage[english]{babel}
\usepackage{amsthm}
\usepackage{amsmath,amsfonts,amscd,amssymb}
\numberwithin{equation}{section}
\usepackage{makeidx}

\theoremstyle{plain}
  \newtheorem{theorem}{Theorem}[section]
  
  \newtheorem{corollary}[theorem]{Corollary}
  \newtheorem{lemma}[theorem]{Lemma}
  \newtheorem{proposition}[theorem]{Proposition}
\theoremstyle{definition}
  \newtheorem{definition}[theorem]{Definition}
  \newtheorem{example}[theorem]{Example}
\theoremstyle{remark}
  \newtheorem{remark}[theorem]{Remark}
  \newtheorem*{notation}{Notation}

\newcommand{\doesnotdivide}{\not\hspace{2.5pt}\mid}

\let\noi=\noindent

\makeindex
\title{Key polynomials for simple extensions of valued fields}
\author{F. J. Herrera Govantes\\
Departamento de \'Algebra\\
Facultad de Matem\'aticas\\
Avda. Reina Mercedes, s/n\\
Universidad de Sevilla\\
41012 Sevilla, Spain \\ email: jherrera@algebra.us.es \and
W. Mahboub\\
Institut de Math\'ematiques de Toulouse\\
UMR 5219 du CNRS,\\
Universit\'e Paul Sabatier\\
118, route de Narbonne\\
31062 Toulouse cedex 9, France\\email:
wael.mahboub@math.univ-toulouse.fr \and
M. A. Olalla Acosta\footnote{Partially suported by PID2020-114613GB-I00 (MCIN/AEI/10.13039/501100011033) and P20-01056 (Junta de Andaluc\'{\i}a and ERDF)}\\
Departamento de \'Algebra\\
Facultad de Matem\'aticas\\
Avda. Reina Mercedes, s/n\\
Universidad de Sevilla\\
41012 Sevilla, Spain\\
email: miguelolalla@aus.es \and
M. Spivakovsky$^*$\\
Institut de Math\'ematiques de Toulouse\\
UMR 5219 du CNRS,\\
Universit\'e Paul Sabatier\\
118, route de Narbonne\\
31062 Toulouse cedex 9, France and\\
UMI CNRS 2001 LaSol, UNAM\\email:
mark.spivakovsky@math.univ-toulouse.fr }

\begin{document}

\newcommand{\g}{\Gamma}
\newcommand{\h}{\Phi}
\newcommand{\ch}{{\operatorname{char}}}
\newcommand{\ini}{{\operatorname{in}}}
\newcommand{\ord}{{\operatorname{ord}}}
\newcommand{\gr}{{\operatorname{gr}}}
\newcommand{\rk}{{\operatorname{rank}}}
\newcommand{\md}{{\operatorname{mod}}}

\renewcommand{\do}{,\ldots,}


\maketitle

\textit{Dedicated to the memory of Roberto Callejas Bedregal.}
\bigskip

\begin{abstract}
In this paper we present a refined version of MacLane's theory of key polynomials, similar to those considered by M. Vaqui\'e and reminiscent of approximate roots of Abhyankar and Moh.

Given a simple transcendental extension of valued fields, we associate to it a countable well ordered set of polynomials called key polynomials. We define limit key polynomials and give explicit formulae for them. We give an explicit bound on the order type of the set of key polynomials.
\end{abstract}
\medskip

\section{Introduction}
\label{In}

Let $\iota:(K,\nu_K)\hookrightarrow(K(x),\nu)$ be a simple transcendental extension of valued fields such that $\text{rank}\ \nu_K=1$.  In the case when $\nu$ is discrete (that is, the value group of $\nu$ is $\mathbb Z$), S. Mac Lane described all such extensions in the nineteen thirties. His description was based on the theory of {\bf key polynomials} associated to $\iota$ that he introduced for that purpose (\cite{57}--\cite{59}).

This seminal work of Mac Lane is connected to the resolution of fundamental arithmetic--geometric problems like the computation of prime ideal decomposition in number fields or the design of polynomial factorization algorithms over henselian fields. On the geometric side, limit key polynomials are linked to the defect of the extensions of valued fields, which is a well-known obstacle for the solution of the local uniformization problem in positive characteristic.

Let $(R_{\nu_K},M_{\nu_K},k_{\nu_K})$ denote the valuation ring of $\nu_K$. The purpose of this paper is to present a refined version of MacLane's theory of key polynomials \cite{57}, \cite{58}, similar to those considered by M. Vaqui\'e \cite{62}--\cite{72}, and reminiscent of related objects studied by M. Lejeune-Jalabert \cite{LJ}, Abhyankar and Moh (approximate roots \cite{5}, \cite{6}) and T.C. Kuo \cite{50}, \cite{51}. Related questions were studied by Ron Brown \cite{RB1}--\cite{RB2}, Alexandru--Popescu--Zaharescu \cite{APZ2}--\cite{APZ3}, S. K. Khanduja \cite{KH3}, \cite{KH}, F.-V. Kuhlmann \cite{FV1} and Moyls \cite{Mo}. The study of key polynomials continues to be a vibrant subject and since the first version of the present paper appeared on the arxiv in 2014 several other works on the topic were published by various authors: \cite{Nart0}, \cite{CMT}, \cite{MN}, \cite{DMS}, \cite{BN}, \cite{Nart1}, \cite{Nart2}, \cite{N1}, \cite{NS4}. 

Precisely, we associate to $\iota$ a countable well ordered set
$$
\mathbf{Q}=\{Q_i\}_{i\in\Lambda}\subset K[x],
$$
where $\Lambda$ is an index set whose order type will be explicitly bounded later in the paper; the $Q_i$ are called {\bf key polynomials}. Key polynomials $Q_i$ that have no immediate predecessor are called {\bf limit key polynomials}

We consider the main achievements of this paper compared to the earlier works on the subject to be the following.

(1) Explicit formulae for each key polynomial $Q_i$, particularly for limit key polynomials (Proposition \ref{Proposition13.1}), in terms of the key polynomials preceding $Q_i$. If there exists an infinite sequence of linear key polynomials then the first limit key polynomial (if it exists) can always be chosen to be a $p$-polynomial in Kaplansky's terminology ($p$-polynomials may be viewed as a generalization of Artin--Schreier polynomials).

(2) An upper bound on the order type of the set of key polynomials. Namely,  we show that the order type of the set of key polynomials is bounded by $\omega\times\omega$, where $\omega$ stands for the first infinite ordinal. If $\ch\ k_{\nu_K}=0$, the order type is bounded by $\omega+1$. If $\ch\ k_{\nu_K}=0$ and $\text{rank}\ \nu=1$, the set of key polynomials has order type at most
$\omega$.

The results comparing the value of a polynomial $f\in K[x]$ with the values of its Hasse derivative $\partial_bf$ (the definition of
$\partial_b$ is recalled later in the Introduction) proved in \S\ref{Keypolanddiffop}  led to an axiomatic characterization of key polynomials  in \cite{DMS}. As well, this approach was used in \cite{NS4} to obtain another proof of the existence of a complete set of key polynomials.

The main application of the theory of key polynomials (particularly, of point (1) above) that we have in mind is proving the local uniformization theorem for quasi-excellent noetherian schemes in positive and mixed characteristic. It has been shown recently that to prove the local uniformization theorem in the positive equicharacteristic case, assuming local uniformization in lower dimensions, it is sufficient to monomialize the first limit key polynomial of a certain explicitly defined simple field extension $K\hookrightarrow K(x)$ (see \cite{JCSS0}, Chapter IV, \cite{JCSS}, Theorem 6.5 and S. D. Cutkosky--H. Mourtada \cite{CM}). In Chapter V of his Ph.D. thesis (\cite{JCSS0}, Institut de Math\'ematiques de Toulouse, 2013), J.-C. San Saturnino proved a similar reduction for local uniformization in the case of mixed characteristic, but under somewhat restrictive additional hypotheses.

Chapter 3 of the Ph.D. thesis of W. Mahboub (Institut de Math\'ematiques de Toulouse, 2013) develops the theory of key polynomials for valuations of arbitrary rank. Here we limit ourselves to the case $\text{rank}\ \nu_K=1$.

The particular importance of the case $\text{rank}\ \nu_K=1$ is witnessed by a theorem of Novacoski--Spivakovsky that says that local uniformization along rank one valuations implies local uniformization in its full generality \cite{NS}--\cite{NS1}.

Let $\g_0$ (resp. $\g$) denote the value group of $\nu_K$ (resp. $\nu$). Let $\tilde\Gamma_0:=\g_{0}\otimes_{\mathbb Z}\mathbb
Q$ (the group $\tilde\g_0$ is called the {\bf divisible hull} of $\g_0$). Fix an embedding $\tilde\g_0\hookrightarrow\mathbb R$ once and for all. In this sense, we will talk about the supremum of a certain subset of $\tilde\g_0$ (the supremum can be either a real number or infinity) or about a certain sequence of elements of $\tilde\g_0$ tending to infinity.

For an ordered abelian group $\Delta$, the notation $\Delta_+$ will stand for the semigroup formed by all the non-negative elements of
$\Delta$.

The well ordered set $\mathbf{Q}=\{Q_i\}_{i\in\Lambda}$ of key polynomials of $\nu$ will be defined recursively in $i$.
\medskip

\noindent{\bf Notation.} We will use the notation $\mathbb{N}$ for the set of strictly positive integers and $\mathbb{N}_0$ for the set of non-negative integers.

For an element $\ell\in\Lambda$, we will denote by $\ell+1$ the immediate successor of $\ell$ in $\Lambda$. The immediate predecessor of $\ell$, when it exists, will be denoted by $\ell-1$. For a strictly positive integer $t$, $\ell+t$ will denote the immediate successor of
$\ell+(t-1)$. For an element $\ell\in\Lambda$, the initial segment $\{Q_i\}_{i<\ell}$ of the set of key polynomials will be denoted by
$\mathbf{Q}_\ell$. For the rest of this paper, we let $p=\ch\ k_{\nu_K}$ if $\ch\ k_{\nu_K}>0$ and $p=1$ if $\ch\
k_{\nu_K}=0$. For an element $\beta\in\g\cup\tilde\g_0$, let
$$
\begin{array}{rcl}
\mathbf{P}_\beta&=\left\{y\in K(x)\ \left|\ \nu(y)\ge\beta\right.\right\}\cup\{0\}\\
\mathbf{P}_{\beta+}&=\left\{y\in K(x)\ \left|\ \nu(y)>\beta\right.\right\}\cup\{0\}.
\end{array}
$$
Put
\begin{equation}
G_{\nu}=\bigoplus\limits_{\beta\in\g}\frac{\mathbf{P}_\beta}{\mathbf{P}_{\beta+}}\label{eq:Gnu}
\end{equation}
and
\begin{equation}
\tilde G_\nu=\bigoplus\limits_{\beta\in\tilde\g_0}\frac{\mathbf{P}_\beta}{\mathbf{P}_{\beta+}}.\label{eq:tildeGnu}
\end{equation}
We regard $G_{\nu}$ and $\tilde G_\nu$ as $k_{\nu}$-algebras. Note that even though $\tilde\g_0$ need not be contained in $\g$, we have $\tilde G_\nu\subset G_\nu$ since for each $\beta\in\tilde\g_0\setminus\g$ the corresponding summand in (\ref{eq:tildeGnu}) is 0. For $y\in K(x)^*$, let $\ini_{\nu}y$ denote the natural image of $y$ in $\frac{\mathbf{P}_\beta}{\mathbf{P}_{\beta+}}\subset
G_{\nu}$, where $\beta=\nu(y)$.
\smallskip

Let $\Delta x$ be an independent variable. For $f\in K[x]$ and $j\in\mathbb N$ let $\partial_jf$ denote the\linebreak$j$-th \textbf{formal (or Hasse) derivative} of $f$ with respect to $x$. The polynomials $\partial_jf$ are, by definition, the coefficients appearing in the Taylor expansion of $f$: $f(x+\Delta x)=\sum\limits_j\partial_jf(x)\Delta x^j$. In papers on local uniformization the Hasse derivatives $\partial_j$ are often denoted by $\frac1{j!}\frac{\partial^j}{\partial x^j}$; this notation is regarded as one indivisible symbol; its parts such as $\frac1{j!}$ do not make sense on their own.
\medskip

Details about Hasse derivatives can be found in \cite{41}, Chapter 24.10, starting on p. 701, as well as in \cite{46}.

For an ordinal $\ell\in\Lambda$ we will use the following multi-index notation: $\bar\gamma_{\ell +1}=(\gamma_i)_{i\le\ell }$, where the $\gamma_i$ are non-negative integers, all but finitely many of which are equal to 0, and
\begin{equation}
\mathbf{Q}_{\ell +1}^{{\bar\gamma}_{\ell +1}}=\prod\limits_{i\le\ell }Q_i^{\gamma_i}.\label{tag1-7}
\end{equation}
An $\ell${\bf-standard monomial} in $\mathbf{Q}_{\ell +1}$ is a product of the form
\begin{equation}
c_{\bar\gamma_{\ell +1}}\mathbf{Q}_{\ell +1}^{\bar\gamma_{\ell +1}},\label{tag1-8}
\end{equation}
where $c_{{\bar\gamma}_{\ell +1}}\in K$ and the multiindex ${\bar\gamma}_{\ell +1}$ satisfies certain additional conditions to ensure a form of uniqueness (see Definition \ref{de9.9}). An $\ell$-{\bf standard expansion} is a finite sum of $\ell$-standard monomials satisfying a mild additional condition. A $Q_\ell$-{\bf expansion} is an expression of the form
\begin{equation}
\sum\limits_{j=0}^{s_\ell }c_{j,\ell}Q_\ell ^j,\label{eq:weakstexp}
\end{equation}
where $c_{j,\ell}\in K[x]$ and $\deg_xc_{j,\ell}<\deg_xQ_\ell$ for all $j$.
\medskip

\noi{\bf Note.} Starting with an $\ell$-standard expansion and grouping together all the terms involving $Q_\ell^j$ for each exponent $j$ produces a $Q_\ell$-expansion.
\medskip

Iterating Euclidean division of $f$ by $Q_\ell$, it is easy to see that every $f\in K[x]$ admits a unique $Q_\ell$-expansion
\begin{equation}
f=\sum\limits_{j=0}^{s_\ell }c_{j,\ell}Q_\ell ^j.\label{tag1-9}
\end{equation}
In \S\ref{SetsofKeyPols} we will show that for all $\ell\in\Lambda$ and all $f\in K[x]$ the element $f$ admits an $\ell$-standard expansion.

Let $\beta_i=\nu(Q_i)$.

A product of the form $a\prod\limits_{j=1}^sQ_{i_j}^{\gamma_j}$, where $a\in K$, $i_j\in\Lambda$ and $\gamma_j\in\mathbb{N}$ is said to be a {\bf standard monomial} if it is $\ell$-standard for some $\ell\in\Lambda$.

A set of key polynomials is said to be {\bf complete} if for every $\beta\in\g$ the additive group $\mathbf P_{\beta}\cap K[x]$ is generated by standard monomials, contained in $\mathbf P_{\beta}\cap K[x]$. It is said to be $\tilde\g_0$-{\bf complete} if the above condition holds for all $\beta\in\tilde\g_0$, in other words, if for all $\beta\in\tilde\g_0$ every polynomial $f\in K[x]$ with
$\nu(f)=\beta$ belongs to the additive group generated by standard monomials $a\prod\limits_{j=1}^sQ_{i_j}^{\gamma_j}$ such that $\sum\limits_{j=1}^s\gamma_j\beta_{i_j}+\nu_K(a)\ge\beta$.

\begin{remark} If $\mathbf Q=\{Q_i\}_{i\in\Lambda}$ is a complete set of key polynomials, the data $\{Q_i,\beta_i\}$ completely determines the ideals $\mathbf P_\beta$ for all $\beta\in\g$, hence also all the ideals $\mathbf P_{\beta+}$, since $\mathbf
P_{\beta+}=\bigcup\limits_{\tilde\beta>\beta}\mathbf P_{\tilde\beta}$. For an element $y\in K(x)$ we have $\nu(y)=\beta$ if and only if $y\in\mathbf P_\beta\setminus\mathbf P_{\beta+}$. Thus the valuation $\nu$ is completely determined by the data
$\{Q_i,\beta_i\}$.
\end{remark}
\medskip

We define the $\ell$-\textbf{truncation} $\nu_\ell$ of $\nu$ by $\nu_\ell(f)=\min\limits_{0\le j\le s_\ell }\{\nu(c_{j,\ell})+j\beta_\ell\}$ for each $f\in K[x]$ with $Q_\ell$-expansion (\ref{tag1-9}). By the ultrametric triangle law, we have
\begin{equation}
\nu(f)\ge\nu_\ell (f)\label{tag1-10}
\end{equation}
for all $f\in K[x]$. Then the statement that $\mathbf{Q}$ is a complete set of key polynomials can be expressed as follows: for all $f\in K[x]$ there exists $\ell\in\Lambda$ such that equality holds in (\ref{tag1-10}); see Remark \ref{remark_complete} for details.

The paper is organized as follows. \S\ref{S2} is devoted to generalities on algebras, graded by ordered semigroups. There we define the notion of the saturation $G^*$ of a graded algebra $G$ (Definition \ref{de4-1}). We consider an extension $G\subset G'$ of graded algebras and a homogeneous element $x\in G'$. We study the condition that $x$ be algebraic over $G$. We note that $x$ is algebraic over $G$ if and only if it is integral over $G^*$. We show that if $x$ is algebraic over $G$ then the algebra $G^*[x]$ is saturated (Lemma \ref{el4.4}). Finally, we prove the simple but useful characterization of the strict inequality
$\nu\left(\sum\limits_{i=1} ^sy_i\right)>\min\limits_{1\le i\le s}\left\{\nu(y_i)\right\}$ in terms of the elements $\ini_\nu y_i$.

In \S\ref{SetsofKeyPols} we define the notion of a set of key polynomials (not necessarily complete) and study its properties. By definition, letting $Q_0=x$, the one element set $\{x\}=\{Q_0\}$ is a set of key polynomials. We remark that if a set of key polynomials is complete then the images of the key polynomials in $G_{\nu}$ generate the field of fractions of $G_{\nu}$ over the field of fractions of $G_{\nu_K}$. The element $\ini_\nu Q_\ell$ is algebraic over $G_{\nu_K}[\ini_\nu\mathbf{Q}_\ell]$, where
$\mathbf{Q}_\ell=\{Q_i\}_{i<\ell}$ (in particular, $\beta_i\in\tilde\g_0$) whenever $i$ is not a maximal element of $\Lambda$ (Propositions \ref{notingamma} and \ref{transcendentalcomplete}).

In \S\ref{numchardeltai} we associate to each pair $h\in K[x]$, $i\in\Lambda$, a positive integer numerical character
$\delta_i(h)\le\frac{\deg_xh}{\deg_xQ_i}$ to be used in later sections and study its properties. Among other things we prove that
\begin{equation}
\delta_i(h)\ge\delta_{i'}(h)\qquad\text{whenever }i\le i';\label{eq:deltanonincreasingintro}
\end{equation}
in particular, $\delta_i(h)$ stabilizes for $i$ sufficiently large. We also show that the equality in (\ref{eq:deltanonincreasingintro})
imposes strong restrictions on $h$. The numerical character $\delta_i(h)$ helps analyze infinite ascending sequences
of key polynomials in \S\S\ref{Infsequenceskeypolynom}--\ref{Section7} and is crucial for applications to Local Uniformization.

The main step of our recursive construction of sets of key polynomials is carried out in \S\ref{Augmenting}. Namely, we start with a set
$\{Q_i\}_{i\in\Lambda}$ of key polynomials that is not complete and enlarge it to a strictly greater set $\{Q_i\}_{i\in\Lambda_+}$ of key polynomials. The well ordered set $\Lambda_+$ is a union of $\Lambda$ with a sequence that may be finite or infinite, depending on the situation. Roughly speaking, $Q_{\ell +1}$ is defined to be a lifting to $K[x]$ of the monic minimal polynomial, satisfied by
$\ini_{\nu}Q_\ell $ over the graded algebra $G_{\nu_K}\left[\ini_{\nu}\mathbf{Q}_\ell \right]$. This gives rise to explicit formulae describing each \textit{non-limit} key polynomial in terms of the preceding key polynomials.

In \S\ref{Section3} we iterate the procedure of \S\ref{Augmenting} until the resulting set of key polynomials is complete. Namely, we start our recursive construction of the $Q_i$ by putting $Q_0:=x$. We assume, inductively, that a set $\mathbf{Q}=\{Q_\ell\}_{\ell\in\Lambda}$ of key polynomials is already defined. If the set $\mathbf Q$ of key polynomials is complete, the algorithm stops here. In particular, this occurs whenever our algorithm produces a key polynomial $Q_i$ whose value does not lie in $\tilde\g_0$ or, more generally, such that $\ini_\nu Q_i$ is transcendental over $G_{\nu_K}[\ini_{\nu}\mathbf Q_i]$ (Propositions \ref{notingamma} and \ref{transcendentalcomplete}). If $\mathbf Q$ is not complete, we replace it by the set $\{Q_i\}_{i\in\Lambda_+}$ constructed in \S\ref{Augmenting} and repeat the procedure. We remark (Remark \ref{omegatimesomega}) that the well ordered set $\bar\Lambda$ resulting from this construction has order type at most $\omega\times\omega$. The set $\bar\Lambda$ contains a maximal element
$\ell$ if and only if it contains an element $\ell$ such that $\ini_\nu Q_\ell$ is transcendental over $G_{\nu_K}[\ini_\nu\mathbf{Q}_\ell]$, where $\mathbf{Q}_\ell=\{Q_i\}_{i<\ell}$ (Propositions \ref{notingamma} and \ref{transcendentalcomplete}).

\S\ref{Keypolanddiffop} is auxiliary, to be used in \S\S\ref{Infsequenceskeypolynom}--\ref{Section7}. We study the effect of Hasse derivatives $\partial_j$ on key polynomials and standard expansions. Let $b_i$ denote the smallest element $b$ of $\mathbb{N}$ which maximizes the quantity $\frac{\beta_i-\nu\left(\partial_bQ_i\right)}b$. We show that $b_i$ is of the form
\begin{equation}
b_i=p^{e_i}\quad\text{ for some }e_i\in\mathbb{N}_0\label{tag1-11}
\end{equation}
(Corollary \ref{Corollary10.6}). The non-negative integers $e_i$, $i\in\Lambda$, are important numerical characters of the extension
$\iota:(K,\nu_K)\hookrightarrow(K(x),\nu)$ of valued fields. Most importantly, given an $\ell$-standard monomial
$c_{\bar\gamma_{\ell +1}}\mathbf{Q}_{\ell +1}^{\bar\gamma_{\ell +1}}$, we prove the equality
\begin{equation}\label{comparing_derivatives}
\nu\left(\partial_{p^e}c_{\bar\gamma_{\ell +1}}\mathbf{Q}_{\ell +1}^{\bar\gamma_{\ell +1}}\right)=
\nu_\ell \left(\partial_{p^e}c_{\bar\gamma_{\ell +1}}\mathbf{Q}_{\ell +1}^{\bar\gamma_{\ell +1}}\right),
\end{equation}
and derive an explicit formula for the quantity
$\nu\left(\partial_{p^e}c_{\bar\gamma_{\ell +1}}\mathbf{Q}_{\ell +1}^{\bar\gamma_{\ell +1}}\right)$, for integers $e\ge e_i$, and under certain additional conditions. Also, for every $\ell$-standard expansion $f$ and every integer $e\ge e_i$, we derive a formula for
$\nu_\ell\left(\partial_{p^e}f\right)$ (Proposition \ref{Proposition10.1}).

The importance of this type of explicit formulae can be explained as follows. The importance of differential operators for resolution of singularities is well known. One difficulty with dealing with differential operators up to now has been the fact that they obey no simple transformation law under blowing up. Since key polynomials become coordinates after blowing up, formulae (\ref{comparing_derivatives}) can be viewed as comparison results for derivatives of the defining equations of a singularity before and after blowing up.

The main subject of study of \S\S\ref{Infsequenceskeypolynom}--\ref{Section7} are infinite sequences
$\{Q_{\ell+t}\}_{t\in\mathbb N}$ of key polynomials of a fixed degree and the corresponding limit key polynomials $Q_{\ell+\omega}$.

In \S\ref{Infsequenceskeypolynom} we let $\delta$ denote the stable value of $\delta_{\ell+t}(Q_{\ell+\omega})$ for a sufficiently large positive integer $t$ (such a stable value exists by (\ref{eq:deltanonincreasingintro})). We use the results of \S\ref{Keypolanddiffop} to show that $\delta$ must be of the form $\delta=p^e$ for some $e\in\mathbb{N}_0$ (Propositon \ref{Proposition12.5}).

Next, we assume that $\ch\ k_\nu=\ch\ K$, that the sequence $\left\{\nu\left(Q_{t}\right)\right\}_{t\in\mathbb N}$ is unbounded in
$\tilde\g_0$ and that $\deg_xQ_t=1$ for all $t\in\mathbb N$. We show that $Q_{\omega}\in K\left[x^\delta\right]$ (Remark \ref{Remark12.6}).

The third set of main results of \S\ref{Infsequenceskeypolynom} starts with Proposition \ref{Proposition12.8} which asserts that if
$$
\ch\ k_\nu=0,
$$
then for an infinite sequence $\{Q_t\}_{t\in\mathbb N_0}$ the sequence $\{\beta_t\}_{t\in\mathbb N_0}$ of their values is unbounded in $\tilde\g_0$. In particular, there are no limit key polynomials $Q_i$ such that $\beta_i\in\tilde\g_0$.  This explains why
$\bar\Lambda\le\omega+1$ whenever $\ch\ k_\nu=0$.

The main goal of \S\ref{Section7} is to derive explicit formulae for limit key polynomials in terms of the preceding key polynomials.
We assume that $\ch\ k_\nu=p>0$ and consider a limit ordinal $\ell+\omega\in\bar\Lambda$. We assume that the sequence
$\left\{\nu\left(Q_{\ell+t}\right)\right\}_{t\in\mathbb N}$ is bounded in $\tilde\g_0$. We prove that $Q_{\ell+\omega}$ can be chosen in such a way that for some $t\in\mathbb N$ the $Q_{\ell+t}$-standard expansion of $Q_{\ell+\omega}$ is weakly affine. By definition, this means that
\begin{equation}
Q_{\ell+\omega}=Q_i^{p^{e_{\ell+\omega}}}+\sum\limits_{j=0}^{e_{\ell+\omega}-1}c_{p^j,i}Q_i^{p^j}+c_{0i},\label{tag1-12}
\end{equation}
where $i=\ell+t$ and $c_{0,i}$ and $c_{p^j,i}$ are $Q_i$-free $i$-standard expansions (See Definition \ref{defQlfree} for the notion of ``$Q_i$-free'').

The results of this paper are related to those contained in the paper \cite{36} (see also \cite{72}).  However, there are some important differences, which we now explain. We chose to rewrite the whole theory from scratch for the following reasons.
\begin{enumerate}
\item In \cite{36} we work with an algebraic extension $\iota$ while for local uniformization we need to consider purely transcendental extensions. We note that the case of algebraic extensions can easily be reduced to that of transcendental ones using composition of valuations. Indeed, let $\iota_-:(K,\nu_K)\hookrightarrow(K(x),\nu)$ be a simple algebraic extension of valued fields. Write $K(x)=\frac{K[X]}{(f)}$, where $f$ is the minimal polynomial of $x$ over $K$. Let $\nu_f$ denote the $(f)$-adic valuation of $K[X]$ and put $\nu^*:=\nu_f\circ\nu$ (the composition of $\nu_f$ with $\nu$). Complete sets $\{Q_i\}_{i\in\Lambda}$ of key polynomials of the transcendental extension $\iota:(K,\nu_K)\rightarrow (K(X),\nu^*)$ constructed in the present paper are very closely related to complete sets $\{Q^-_i\}_{i\in\Lambda_-}$ of key polynomials of the algebraic extension $(K,\nu_K)\rightarrow(K(x),\nu)$, constructed in \cite{36}. Namely, we have $\Lambda=\Lambda_-\cup\{\Lambda_-\}$ (extension by one element), $Q^-_i$ is the image of $Q_i$ under the natural map $K[X]\rightarrow K(x)$ and $Q_{\Lambda_-}=f$. In other words, a complete set $\{Q_i\}$ of key polynomials for $\iota$ can be obtained from that of $\iota_-$ by lifting each key polynomial $Q_i^-$ to $K[X]$ and then adding one final key polynomial $f$. In this sense the theory presented here can be viewed as a generalization of \cite{36}.
\item Our main interest in \cite{36} was to classify all the possible extensions $\nu$ of a given $\nu_K$; in the present paper we content
ourselves with a fixed $\nu$.
\item The crucial formulae for $\nu_\ell (\partial_{p^b}f)$ were not
made explicit in \cite{36}.
\item We take this opportunity to correct numerous mistakes which, unfortunately, made their way into the paper \cite{36}: an inaccuracy in the definition of complete set of key polynomials, the failure to take into account the case of mixed characteristic, a mistake in the definition of the numerical characters $e_i$ and many others which made the paper \cite{36} unreadable.
\end{enumerate}
\medskip

\noindent\textbf{Acknowledgements.} We thank Anna Blaszczok, Julie Decaup, Franz-Viktor Kuhlmann, G\'erard Leloup and all the anonymous referees of the earlier versions of the paper for many useful comments and suggestions and for pointing out errors.

\section{Algebras graded by ordered semigroups}
\label{S2}

Graded algebras associated to valuations will play a crucial role in this paper. In this section, we give some basic definitions and prove several easy results about graded algebras. Throughout this paper, a ``graded algebra'' will mean ``an algebra without zero divisors, graded by an ordered semigroup''. As usual, for a graded algebra $G$, $\ord$ will denote the natural valuation of $G$, given by the grading.

Let $G=\bigoplus\limits_{\alpha\in\Gamma}G_{\alpha}$ be a graded algebra, where $\Gamma$ is an ordered abelian semigroup.

\begin{definition} An element $x\in G$ is said to be \textbf{homogeneous} if there exists $\alpha\in\Gamma$ such that $x\in G_\alpha$.
\end{definition}

For a homogeneous element $x\in G_\alpha\subset G$ we will write $\ord\ x=\alpha$.

Now let $X$ be an independent variable and consider the ring $G[X]$. Fix a polynomial $f=\sum\limits_{i=0}^{d}a_iX^i\in G[X]$ such that $a_i$ is a homogeneous element of $G$ for all $i\in\{0,\dots,d\}$. Fix an element $\beta\in\Gamma$.

\begin{definition} We say that $f$ is \textbf{quasi-homogeneous with} $w(X)=\beta$ if for all $i,j\in\{0,\dots,d\}$ we have $i\beta+\ord\ a_i=j\beta+\ord\ a_j$. In this situation we will also say that $\beta$ is the \textbf{weight assigned to }$X$.
\end{definition}

\begin{definition}{\label{de4-1}} Let $G$ be a graded algebra without zero divisors. The {\bf saturation} of $G$, denoted by $G^*$, is the graded algebra
$$
G^*=\left\{\left.\frac gh\ \right| \ g,h\in G,\
h\text{ homogeneous},\ h\ne0\right\}.
$$
$G$ is said to be {\bf saturated} if $G=G^*$.
\end{definition}

An element $\frac gh\in G^*$ is homogeneous in $G^*$ if and only if $g$ is homogeneous in $G$. If $f_1=\frac{g_1}{h_1}$ and $f_2=\frac{g_2}{h_2}$ are two non-zero elements of $G^*$, where $h_1,h_2,g_1,g_2$ are non-zero elements of $G$ with $h_1,h_2,g_2$ homogeneous,  then $\frac{f_1}{f_2}=\frac{g_1h_2}{g_2h_1}\in G^*$. Thus $G^*=(G^*)^*$ for any graded algebra $G$, so that $G^*$ is always saturated.

\begin{example}{\label{Example4-2}} The main example of graded algebras we are interested in this paper are graded algebras associated to valuations, centered in prime ideals of integral domains. Namely, let $R$ be a domain, $K$ its field of fractions and
$\nu:K^*\rightarrow \g$ a valuation of $K$, centered at a prime ideal $P$ of $R$ (this means, by definition, that $R\subset R_\nu$ and $P=M_\nu\cap R$, where $(R_\nu,M_\nu)$ denotes the valuation ring of $\nu$). Let $\h=\nu(R\setminus\{0\})$. For each $\beta\in\h$, consider the ideals
\begin{equation}
\begin{aligned}
I_\beta:&=\{x\in R\ |\ \nu(x)\ge\beta\}\cup\{0\}\quad\text{and}\\
I_{\beta+}:&=\{x\in R\ |\ \nu(x)>\beta\}\cup\{0\}.
\end{aligned}\label{tag3.1}
\end{equation}
$I_\beta$ is called the $\nu$-{\bf ideal} of $R$ of value $\beta$.

If $\beta_1>\beta_2>\dots$ is an infinite descending sequence of elements of $\h$ then $I_{\beta_1}\subsetneqq
I_{\beta_2}\subsetneqq\dots$ is an infinite ascending chain of ideals of $R$. Thus if $R$ is noetherian then the ordered set $\nu(R)$ contains no infinite descending sequences, that is, $\nu(R)$ is well ordered.

If $I$ is an ideal in a noetherian ring $R$ and $\nu$ a valuation of $R$, $\nu(I)$ will denote
$$
\min\{\nu(x)\ |\ x\in I\}.
$$
We can now define the graded algebra, associated to the valuation $\nu$. Let $R$, $\nu$ and $\h$ be as above. For $\beta\in\h$,
let $I_\beta$ and $I_{\beta+}$ be as in (\ref{tag3.1}). We define
$$
\gr_\nu R=\bigoplus_{\beta\in\h}\frac{I_\beta}{I_{\beta+}}.
$$
The algebra $\gr_\nu R$ is an integral domain. For any element $x\in R$ with $\nu(x)=\beta$, we may consider the natural image of $x$ in$\frac{P_\beta}{P_{\beta+}}\subset\gr_\nu R$. This image is a homogeneous element of $\gr_\nu R$ of degree $\beta$, which we will denote by $\ini_\nu x$. The grading induces an obvious valuation on $\gr_\nu R$ with values in $\h$; this valuation will be denoted by $\ord$.

Next, suppose that $(R,M,k)$ is a local domain and $\nu$ is a valuation with value group $\g$, centered at $R$. Let $K$ denote the field of fractions of $R$. Let $(R_\nu,M_\nu,k_\nu)$ denote the valuation ring of $\nu$. For $\beta\in\g$, consider the following
$R_\nu$-submodules of $K$:
\begin{equation}
\begin{aligned}
\mathbf{I}_\beta&=\{x\in K\ |\ \nu(x)\ge\beta\},\\
\mathbf{I}_{\beta+}&=\{x\in K\ |\ \nu(x)>\beta\}.
\end{aligned}\label{tag3.2}
\end{equation}
We define
$$
G_\nu=\bigoplus_{\beta\in\g}\frac{\mathbf{I}_\beta}{\mathbf{I}_{\beta+}}.
$$
Again, given $x\in K$, we may speak of the natural image of $x$ in $G_\nu$, also denoted by $\ini_\nu x$ (since $\gr_\nu R$ is naturally a graded subalgebra of $G_\nu$, there is no danger of confusion). Then $\ord$ is a valuation of the common field of fractions
of $\gr_\nu R$ and $G_\nu$, with values in $\g$.

We have $G_\nu=(\gr_\nu R)^*$; in particular, $G_\nu$ is saturated.
\end{example}

\begin{remark}{\label{Remark 2.3}} Let $G,G'$ be two graded algebras without zero divisors, with $G\subset G'$. Let $x$ be a homogeneous element of $G'$, satisfying an algebraic dependence relation
\begin{equation}
a_0x^n+a_1x^{n-1}+\dots+a_n=0\label{tag2.3}
\end{equation}
over $G$ (here $a_j\in G$ for $0\le j\le n$). Without loss of generality, we may assume that the integer $n$ is the smallest possible.
\smallskip

\noindent\textbf{Claim.} Without loss of generality, we may further assume that (\ref{tag2.3}) is homogeneous (that is, all the $a_j$ are homogeneous and the quantity $j\ \ord\ x+\ord\ a_j$ is constant for $0\le j\le n$ such that $a_j\ne0$).
\smallskip

\noindent\textit{Proof of Claim.} Let $\mu:=\min\limits_{0\le j\le n}\{j\ \ord\ x+\ord\ a_j\}$. Then each $a_j$ can be written as a finite sum of homogeneous elements of $G$, all of orders greater than or equal to $\mu-j\ \ord\ x$. For $j\in\{0,\dots,n\}$ write
$a_j=a_{0j}+\tilde a_j$, where $a_{0j}=0$ or $\ord\ a_{0j}=\mu-j\ \ord\ x$, and $\tilde a_j$ is a sum of homogeneous elements of $G$ of orders strictly greater than $\mu-j\ \ord\ x$ (note that there exist at least two different values of $j$ for which $a_{0j}\ne0$). Now, $x$ satisfies the equation
\begin{equation}
a_{00}x^n+a_{01}x^{n-1}+\dots+a_{0n}=0.\label{tag2.4}
\end{equation}
This proves the Claim.
From now on we will always take the coefficients $a_j$ to be homogeneous without mentioning it explicitly.
\smallskip

Dividing (\ref{tag2.3}) by $a_0$, we see that $x$ satisfies a {\it monic} homogeneous relation over $G^*$ of degree $n$ and no algebraic relation of degree less than $n$. In other words, $x$ is {\it algebraic} over $G$ if and only if it is {\it integral} over $G^*$; the conditions of being ``algebraic over $G^*$'' and ``integral over $G^*$'' are one and the same thing (as usual, ``integral'' means ``satisfying a monic polynomial relation'').
\end{remark}

Let $G\subset G'$, let $x$ be as above and let $G[x]$ denote the graded subalgebra of $G'$, generated by $x$ over $G$. By Remark \ref{Remark 2.3}, we may assume that $x$ satisfies a homogeneous integral relation
\begin{equation}
x^n+a_1x^{n-1}+\dots+a_n=0\label{tag4.2}
\end{equation}
over $G^*$ and no algebraic relations over $G^*$ of degree smaller than $n$.

\begin{lemma}{\label{el4.4}} Every element of $(G[x])^*$ can be written uniquely as a polynomial in $x$ with coefficients in $G^*$, of
degree strictly less than $n$.
\end{lemma}

\begin{proof} Let $y$ be a homogeneous element of $G[x]$. Since $x$ is integral over $G^*$, so is $y$ (\cite{AM}, p. 59, Proposition 5.1, implications (i)$\iff$(ii)$\iff$(iii)). Let
\begin{equation}
y^m+b_1y^{m-1}+\dots+b_m=0\label{tag4.3}
\end{equation}
with $b_j\in G^*$, be an integral dependence relation of $y$ over $G^*$, with $b_j$ homogeneous elements of $G^*$, $b_m\ne0$, where $j\ \ord\ y+\ord\ b_j$ is constant for all $j$ such that $b_j\ne0$. By (\ref{tag4.3}),
$$
\frac1y=-\frac1{b_m}(y^{m-1}+b_1y^{m-2}+\dots+b_{m-1}).
$$
Thus, for any $z\in G[x]$, we have
\begin{equation}
\frac zy\in G^*[x].\label{tag4.4}
\end{equation}
Since $y$ was an arbitrary homogeneous element of $G[x]$, we have proved that
$$
(G[x])^*=G^*[x].
$$
Now, for every element $y\in G^*[x]$ we can add a multiple of (\ref{tag4.2}) to $y$ so as to express $y$ as a polynomial in $x$ of degree strictly less than $n$. Moreover, this expression is unique because $x$ does not satisfy any algebraic relation over $G^*$ of degree smaller than $n$.
\end{proof}

\begin{notation} If $\Delta\subset\Delta'$ are ordered semigroups and $\beta$ is an element of $\Delta'$, then $\Delta:\beta$ will denote the positive integer defined by
$$
\Delta:\beta=\min\{n\in\mathbb{N}\ |\ n\beta\in\Delta\}.
$$
If the set on the right hand side is empty, we put $\Delta:\beta=\infty$.\\
Note that $\beta\in\Delta$ if and only if $\Delta:\beta=1$.
\end{notation}

\begin{lemma}{\label{el4.5}} Let $G$, $G'$ be as in Remark \ref{Remark 2.3} and $x$ a homogeneous element of $G'$. Assume that the degree 0 part of $G$ (that is, the subring of $G$ consisting of all the elements of degree 0) contains a field $k$ and that $G$ is generated as a $k$-algebra by homogeneous elements $w_1\do w_r$. Let
\[
\beta_j=\ord\ w_j,\quad1\le j\le r,
\]
and let $\Delta$ denote the group $\Delta=\{\ord\ y\ |\ y\in G^*\}=\left\{\left.\sum\limits_{j=1}^rn_j\beta_j\ \right|\
n_j\in\mathbb{Z}\right\}$. Assume that the following two conditions hold:
\begin{enumerate}
\item[(1)] $\Delta:(\ord\ x)<\infty$
\item[(2)] Let $\tilde n:=\Delta:(\ord\ x)$. Let $n_1\do n_r\in\mathbb{Z}$
be such that
\begin{equation}
\tilde n\ \ord\ x=\sum\limits_{j=1}^rn_j\beta_j.\label{tag4.5}
\end{equation}
Let $y=\prod\limits_{j=1}^rw_j^{n_j}$. Assume that the
element
\begin{equation}
z:=\frac{x^{\tilde n}}y\in(G')^*\label{tag2.8}
\end{equation}
is algebraic over $k$.
\end{enumerate}
Then $x$ is integral over $G^*$. An integral dependence relation of $x$ over $G^*$ can be described as follows. Let $z$ be as in (\ref{tag2.8}). Let $Z$ be an independent variable and let
\begin{equation}
f(Z)=Z^d+\sum\limits_{i=0}^{d-1}c_iZ^i\label{tag4.7}
\end{equation}
denote the minimal polynomial of $z$ over $k$. Then $x$ is a root of the polynomial
\begin{equation}
X^{d\tilde n}+\sum\limits_{i=0}^{d-1}c_iy^{d-i}X^{i\tilde n}=0.\label{tag2.10}
\end{equation}
Conversely, suppose $x$ is integral over $G^*$. Then (1) holds. Suppose, furthermore, that $\beta_1\do\beta_r$ are
$\mathbb{Z}$-linearly independent. Then (2) also holds. In this case, (\ref{tag2.10}) is the minimal polynomial of $x$ over $G^*$. In particular, the degree $n$ of the minimal polynomial of $x$ over $G^*$ is given by
\begin{equation}
n=d\tilde n.\label{tag4.9}
\end{equation}
\end{lemma}

\begin{proof} If (1) and (2) hold, $x$ is integral over $G^*$ because it is a root of the polynomial (\ref{tag2.10}) (this is verified immediately by substituting (\ref{tag2.8}) for $Z$ in (\ref{tag4.7}) and multiplying through by $y^d$). In particular, if $n$ denotes the degree of $x$ over $G^*$, the equation
\begin{equation}
x^{d\tilde n}+\sum\limits_{i=0}^{d-1}c_iy^{d-i}x^{i\tilde n}=0.\label{tag2.10bis}
\end{equation}
shows that
\begin{equation}
n\le d\tilde n.\label{tag4.10}
\end{equation}
Conversely, suppose $x$ is integral over $G^*$. Then $x$ satisfies a homogeneous integral relation of the form (\ref{tag4.2}) with $a_n\ne0$. Since (\ref{tag4.2}) is homogeneous, we have the equality
\begin{equation}
i\ \ord\ x+\ord\ a_{n-i}=\ord\ a_n\text{ for all }i\text{ such that}\ 0\le i\le n\text{ and }a_{n-i}\ne0.\label{eq:fullhomogeneity}
\end{equation}
Hence
\begin{equation}
n\ \ord\ x=\ord\ a_n.\label{tag4.11}
\end{equation}
Now, $a_n\in G^*$ so that
\begin{equation}
\ord\ a_n\in\Delta .\label{tag4.12}
\end{equation}
Putting together (\ref{tag4.11}) and (\ref{tag4.12}), we obtain (1) of the Lemma.

Now, assume that $\beta_1\do\beta_r$ are $\mathbb{Z}$-linearly independent. We wish to prove (2). Since $\beta_1\do\beta_r$ are $\mathbb{Z}$-linearly independent, all the monomials $w_1^{\gamma_1}\dots w_r^{\gamma_r}$, $\gamma_j\in\mathbb{Z}$, have different values with respect to ord. Since (\ref{tag4.2}) is homogeneous with respect to $\ord$, each $a_i$ must be a {\it
monomial} in the $w_j$ with (not necessarily positive) integer exponents. Also by the $\mathbb{Z}$-linear independence of
$\beta_1\do\beta_r$, the coefficients $n_1\do n_r$ in (\ref{tag4.5}) are uniquely determined. Moreover, any relation of the form
\begin{equation}
i\ \ord\ x-\sum\limits_{j=1}^rn'_j\beta_j=0,\quad i\in\mathbb{N},\ n'_1\do n'_r\in\mathbb{Z}\label{tag4.13}
\end{equation}
is a positive integer multiple of the relation
\begin{equation}
\tilde n\ \ord\ x-\sum\limits_{j=1}^rn_j\beta_j=0.\label{tag4.14}
\end{equation}
By (\ref{eq:fullhomogeneity}), if a term $a_{n-i}x^i$ appears in (\ref{tag4.2}), we have $i\ \ord\ x=\ord\ a_n-\ord\ a_{n-i}\in\Delta$. This proves that $x^i$ may appear in (\ref{tag4.2}) only if $\tilde n\ |\ i$; in particular, $\tilde n\ |\ n$. Let $d':=\frac n{\tilde n}$. Let $0\le i<d'$. To find each nonzero coefficient $a_{n-i\tilde n}$ in (\ref{tag4.2}), note
that
$$
n\ \ord\ x=d'\ \tilde n\ \ord\ x=i\ \tilde n\ \ord\ x+\ord\ a_{n-i\tilde n},
$$
so that
\begin{equation}
(d'-i)\ \tilde n\ \ord\ x=\ord\ a_{n-i\tilde n}.\label{tag4.15}
\end{equation}
Since $a_{n-i\tilde n}$ is a monomial in $w_1\do w_r$, (\ref{tag4.15}) gives rise to a $\mathbb{Z}$-linear dependence relation of the form (\ref{tag4.13}), which therefore must be equal to (\ref{tag4.14}) multiplied by $d'-i$. This determines the monomial $a_{n-i\tilde n}$ uniquely up to multiplication by an element of $k$: we must have $a_{n-i\tilde n}=c_iy^{d'-i}$, where $c_i\in k$. Then $z=\frac{x^{\tilde n}}y$ satisfies the algebraic dependence relation
\begin{equation}
z^{d'}+\sum\limits_{i=0}^{d'-1}c_iz^i=0.\label{tag4.16}
\end{equation}
This proves (2) of the Lemma. Now, we have shown that, under the hypothesis of linear independence of the $\beta_j$, if $x$ has degree $n$ over $G^*$ then $\tilde n\ |\ n$ and $z$ is a root of a polynomial of degree $d'=\frac n{\tilde n}$. Letting $d$ denote the degree of $z$ over $k$, as above, we obtain
\begin{equation}
d'=\frac n{\tilde n}\ge d.\label{tag4.17}
\end{equation}
Combining (\ref{tag4.17}) with (\ref{tag4.10}), we obtain (\ref{tag4.9}); in particular, (\ref{tag2.10bis}) is the smallest degree algebraic relation satisfied by $x$ over $G$. This completes the proof of Lemma \ref{el4.5}.
\end{proof}

\begin{corollary}{\label{Corollary4.6}} Let $G$, $w_1\do w_r$ and $\beta_1\do\beta_r$ be as in Lemma \ref{el4.5}. If
$\beta_1\do\beta_r$ are $\mathbb{Z}$-linearly independent in $\Delta$ then $w_1\do w_r$ are algebraically independent over $k$.
\end{corollary}

\begin{proof} Induction on $r$. For $r=1$ there is nothing to prove. For the induction step, assume that the Corollary is true for $r=i$. If $w_{i+1}$ were algebraic over $k[w_1\do w_i]$, we would have
\begin{equation}
(\beta_1\do\beta_i):\beta_{i+1}<\infty\label{tag4.18}
\end{equation}
by Lemma \ref{el4.5}, applied to the graded algebra $k[w_1\do w_i]$ and the element $w_{i+1}$. (\ref{tag4.18}) contradicts the linear independence of $\beta_1\do\beta_r$, and we are done. Alternatively, the Corollary can be proved by observing that
by linear independence of $\beta_1\do\beta_r$, all the monomials in $w_1\do w_r$ have different degrees, thus any polynomial in $w_1\do w_r$ over $k$ contains a unique monomial of smallest degree. Hence it cannot vanish by the ultrametric triangle law.
\end{proof}

\begin{definition}{\label{de4.7}} Let $G$ be a graded algebra and $x_\Lambda:=\{x_\lambda\}_{\lambda\in\Lambda}$ a collection of homogeneous elements of $G$. Let $k$ be a field, contained in the degree 0 part of $G$. Let $k[x_\Lambda]$ denote the
$k$-subalgebra of $G$, generated by $x_\Lambda$. We say that $x_\Lambda$ {\bf rationally generate} $G$ over $k$ if $G^*=k[x_\Lambda]^*$.
\end{definition}

The following result is an immediate consequence of definitions:
\begin{proposition}{\label{Proposition4.8}} Let $G_\nu$ be the graded algebra associated to a valuation
$\nu:K\rightarrow\Gamma$, as above. Consider a sum of the form $y=\sum\limits_{i=1}^sy_i$, with $y_i\in K$. Let
$\beta=\min\limits_{1\le i\le s}\nu(y_i)$ and
$$
S=\left\{\left.i\in\{1,\dots,n\}\ \right|\ \nu(y_i)=\beta\right\}.
$$
The following two conditions are equivalent:
\begin{enumerate}
\item $\nu(y)=\beta$
\item $\sum\limits_{i\in S}\ini_\nu y_i\ne0$.
\end{enumerate}
\end{proposition}
 Let $G$ be a saturated graded algebra, graded by a group $\g$, $L$ the field of fractions of $G$, $T$ an independent variable and
 $\beta\in\tilde\g_+:=(\g\otimes_{\mathbb Z}\mathbb Q)_+$. We regard $G[T]$ as an algebra graded by $\g+\beta\mathbb{Z}$, where $\beta$ is the weight assigned to $T$.
\begin{definition} Take an element $a\in G$. Write $a=\sum\limits_{j=0} ^ta_j$, where each $a_j$ is a homogeneous element of $G$ and $\ord\ a_0<\ord\ a_1<\dots<\ord\ a_t$. The element $a_0$ is called the {\bf initial form} of $a$ and will be denoted by $\bar a$.
\end{definition}
\begin{definition} Take a polynomial $g=\sum\limits_{j=0}^da_jT^j\in G[T]$. Let
$$
S(g):=\left\{j\in\{0\do d\}\ \left|\ \ord\left(a_jT^j\right)=ord\ g\right.\right\}.
$$
The {\bf homogeneization} of $g$ is the polynomial $\bar g:=\sum\limits_{j\in S(g)} a_jT^j$.
\end{definition}
\begin{remark}\label{Gauss} A polynomial $f$ in $G[T]$ is said to be {\bf irreducible} if it cannot be factored as a product of two polynomials, both of which have degrees strictly smaller than $\deg_Tf$. Every quasi-homogeneous polynomial $f$ admits a unique factorization of the form
\begin{equation}
f=a\prod\limits_{j=1}^tg_j^{\gamma_j},\label{eq:Gauss}
\end{equation}
where $a$ is the leading coefficient of $f$, the $\gamma_j$ are strictly positive integers and the $g_j$ are irreducible monic
quasi-homogeneous polynomials of strictly positive degrees (we allow the possibility $t=0$, in which case our claim holds trivially). Indeed, the factorization (\ref{eq:Gauss}) exists in $L[T]$. Clearing denominators in (\ref{eq:Gauss}), we can write
\begin{equation}
bf=a\prod\limits_{j=1}^t\tilde g_j^{\gamma_j},\label{eq:Gaussbis}
\end{equation}
where $b\in G$ and $\tilde g_j\in G[T]$ for all $j$. Replace $b$ by $\bar b$ and each of the $\tilde g_j$ by its homogeneization
$\bar{\tilde g}_j$. Since $a$ is homogeneous and $f$ quasi-homogeneous, this operation does not affect the truth of the equality (\ref{eq:Gaussbis}). In other words, without loss of generality we may assume that $b$ is homogeneous and all the $\tilde g_j$
quasi-homogeneous. Replacing $\tilde g_1$ by $\frac{\tilde g_1}b$, we may assume that $b=1$. Dividing each $\tilde g_j$ by its leading coefficient and modifying $a$ accordingly, we arrive at the situation where $\tilde g_j\in G[T]$ for all $j$. This proves the existence of the factorization (\ref{eq:Gauss}). The uniqueness of the factorization follows from its uniqueness in $L[T]$.
\end{remark}

\section{Sets of Key Polynomials}\label{SetsofKeyPols}

Let $K\rightarrow K(x)$ be a simple transcendental field extension, $\nu$ a valuation of $K(x)$ and $\nu_K$ the restriction of $\nu$ to $K$. We will assume that $\text{rank}\ \nu_K=1$ and
\begin{equation}
\nu(x)>0.\label{tag9.1}
\end{equation}

Let $\Lambda$ be a countable well ordered set.

Let $\{\alpha_{i}\}_{i\in\Lambda}$ be a set of strictly positive integers with
\begin{equation}
\alpha_0=1.\label{eq:alpha0=1}
\end{equation}
For an ordinal $\ell\in\Lambda$, we use the notation $\boldsymbol{\alpha}_{\ell +1}:=\{\alpha_i\}_{0\le i\le\ell }$ and
$\bar\gamma_{\ell+1}=\{\gamma_i\}_{0\le i\le\ell }$, where all but finitely many $\gamma_i$ are equal to 0.

\begin{definition}\label{inessential}
Let $i\in\Lambda$. We say that $i$ is {\bf inessential} if $i+\omega\in\Lambda$ and $\alpha_{i+t}=1$ for all $t\in\mathbb{N}_0$.
\end{definition}

\noindent{\bf Notation.} For $i\in\Lambda$ with $i$ not the maximal element of $\Lambda$, let
\begin{align*}
i+&=i+\omega\qquad\,\text{if }i\text{ is inessential}\\
&=i+1\qquad\text{ otherwise}.
\end{align*}
\begin{remark} In the sequel $\Lambda$ will be an index set for a set of key polynomials. By a recursive construction we will increase
$\Lambda$ and the set of key polynomials indexed by it. It is important to note that, given two totally ordered sets
$\Lambda\subset\Lambda'$, with $\Lambda$ an initial segment of $\Lambda'$ it may happen that an index $i\in\Lambda$ is essential in
$\Lambda$ but inessential in $\Lambda'$. By the same token, the meaning of $i+$ may depend on whether we view $i$ as an element of $\Lambda$ or of $\Lambda'$.
\end{remark}
Our next goal is to define the notion of {\bf a set of key polynomials}. We start with some preliminary definitions and notation, before giving the main definition (Definition \ref{DefKeyPolynom}).
\begin{definition}\label{prekeypolynom} A set $\{Q_{i}\}_{i\in \Lambda}\subset K[x]$ of monic polynomials with
\begin{equation}
Q_0=x,\label{eq:Q0=x}
\end{equation}
is said to be a {\bf set of pre-key polynomials} if for all non-maximal $i\geq1$ and $i_0<i$ such that $i_0+=i$ we have
$$
\deg_x Q_{i}=\alpha_{i}\cdot\deg_x Q_{i_0}.
$$
\end{definition}
\medskip

We will use the following notation: for $\ell\in\Lambda$,
$\mathbf{Q}_{\ell+1}^{{\bar\gamma}_{\ell +1}}=\prod\limits_{i\le\ell}Q_i^{\gamma_i}$. We let $\beta_i=\nu(Q_i)$ for each $i\in\Lambda$.

\begin{definition}{\label{de9.9}} Let $\ell\in\Lambda$. A multiindex ${\bar\gamma}_{\ell +1}$ is said to be {\bf standard with respect to} $\boldsymbol{\alpha}_{\ell +1}$ if
\begin{equation}
0\le\gamma_i<\alpha_{i+}\text{ for }i\le\ell ,\label{tag9.6}
\end{equation}
and if $i$ is inessential then the set $\left\{j<i+\ \left|\ j+=i+\text{ and }\gamma_j\ne0\right.\right\}$ has cardinality at most one. An
$\ell${\bf-standard monomial in the elements} $\mathbf{Q}_{\ell +1}$ (resp. an $\ell${\bf-standard monomial in the elements}
$\ini_{\nu}\mathbf{Q}_{\ell +1}$) is a product of the form $c_{{\bar\gamma}_{\ell +1}}\mathbf{Q}_{\ell +1}^{{\bar\gamma}_{\ell +1}}$, (resp. $c_{{\bar\gamma}_{\ell +1}}\ini_{\nu}\mathbf{Q}_{\ell +1}^{{\bar\gamma}_{\ell +1}}$) where
$c_{{\bar\gamma}_{\ell +1}}\in K$ (resp. $c_{{\bar\gamma}_{\ell +1}}$ is a homogeneous element of $G_{\nu_K}$) and the multiindex ${\bar\gamma}_{\ell +1}$ is standard with respect to $\boldsymbol{\alpha}_{\ell +1}$.
\end{definition}

Keep the notation of Definitions \ref{prekeypolynom}--\ref{de9.9} and let $\{Q_{i}\}_{i\in \Lambda}\subset K[x]$ be a set of pre-key polynomials.

\begin{definition}\label{defQlfree} An $\ell$-standard monomial
$c_{{\bar\gamma}_{\ell +1}}\mathbf{Q}_{\ell +1}^{{\bar\gamma}_{\ell +1}}$ is said to be $Q_\ell$-{\bf free} if it does not involve $Q_\ell$, that is, if $\gamma_\ell=0$.
\end{definition}

\begin{definition}{\label{de9.11}} A $Q_\ell$-{\bf free} $\ell$-{\bf standard expansion} is a finite sum of $Q_\ell$-free $\ell$-standard monomials whose $\nu$-value equals the minimum of the $\nu$-values of the monomials. An $\ell$-{\bf standard expansion} of an element $g\in K[x]$  is an expression of the form $g=\sum\limits_{j=0}^sc_jQ_\ell ^j$, where each $c_j$ is a $Q_\ell$-free
$\ell$-standard expansion. 
\begin{remark} Starting with an $\ell$-standard expansion and grouping together all the terms involving $Q_\ell^j$ for each exponent $j$ produces a $Q_\ell$-expansion.
\end{remark}

For an element $y\in G_{\nu}$, an expression of the form $y=\sum\limits_{{\bar\gamma}}\bar
c_{{\bar\gamma}}\ini_{\nu}\mathbf{Q}_{\ell +1}^{{\bar\gamma}}$, where each $\bar c_{{\bar\gamma}}$ is a homogeneous element of $G_{\nu_K}$ and each $\mathbf{Q}_{\ell +1}^{{\bar\gamma}}$ is an $\ell$-standard monomial, will be called an $\ell$-{\bf standard expansion of }$y$.
\end{definition}

\begin{remark} We note that a $Q_\ell$-free $\ell$-standard expansion is not just an $i$-standard expansion for some $i<\ell$. Namely, it is required, in addition, that the exponent of the last appearing key polynomial $Q_i$ be strictly less than $\alpha_{i+}$.
\end{remark}

\begin{definition}{\label{de9.14}} Let $\sum\limits_{{\bar\gamma}}\bar
c_{{\bar\gamma}}\ini_{\nu}\mathbf{Q}_{\ell +1}^{{\bar\gamma}}$ be an $\ell$-standard expansion, where the $\bar c_{{\bar\gamma}}$ are homogeneous elements of $G_{\nu_K}$. A {\bf lifting} of $\sum\limits_{{\bar\gamma}}\bar c_{{\bar\gamma}}\ini_{\nu}\mathbf{Q}_{\ell +1}^{{\bar\gamma}}$ to $K[x]$ is an $\ell$-standard expansion $\sum\limits_{{\bar\gamma}}c_{{\bar\gamma}}\mathbf{Q}_{\ell +1}^{{\bar\gamma}}$, where $c_{{\bar\gamma}}$ is a representative of $\bar
c_{{\bar\gamma}}$ in $K$.
\end{definition}

\begin{definition}{\label{de9.15}} Assume that $\ch\ k_\nu=p>0$. An $\ell$-standard expansion $\sum\limits_jc_jQ_\ell^j$ is said to be {\bf weakly affine} if $c_j=0$ whenever $j>0$ and $j$ is not of the form $p^e$ for some $e\in\mathbb{N}_0$.
\end{definition}
Before plunging into the technical definition of key polynomials we say a few informal words to motivate it. Roughly speaking, the 0-th key polynomial $Q_0$ is equal to $x$ and the key polynomials $Q_i$ with $i>0$ are elements of $K[x]$ with ``unexpected'' or ``jumping'' values. More precisely, for $f=\sum\limits_j d_jx^j\in K[x]$ define $\nu_0(f)=\min\limits_j\{\nu(d_jx^j)\}$. We have
\begin{equation}\label{eq:ulttriangle0}
\nu(f)\ge\nu_0(f)
\end{equation}
by the ultrametric triangle inequality. The first key polyomial $Q_1$ measures the fact that the inequality (\ref{eq:ulttriangle0}) may be strict; in fact, $Q_1$ is the smallest degree polynomial for which (\ref{eq:ulttriangle0}) is strict. Once $Q_1$ is defined, we define the $Q_1$-expansion of $f$ for each $f$ and the valuation $\nu_1$ satisfying
\begin{equation}
\nu_0(f)\le\nu_1(f)\le\nu(f).\label{eq:ulttriangle1}
\end{equation}
If the second inequality in (\ref{eq:ulttriangle1}) is strict for some $f$, the key polynomial $Q_2$ is the smallest degree polynomial for which (\ref{eq:ulttriangle1}) is strict. We iterate this procedure to construct a (possibly infinite) sequence $Q_0,Q_1,Q_2,\dots$ and the corresponding sequence $\nu_0,\nu_1,\dots$ of truncations of $\nu$, satisfying $\nu_{i-1}(f)\le\nu_i(f)\le\nu(f)$ for all $f$ and all strictly positive integers $i$. The passage from $Q_{i-1}$ to $Q_i$ corresponds to the successor case in the definition below. It may happen that even this infinite process does not describe the valuation $\nu$ completely, that is, there exists $f\in K[x]$ such that
\begin{equation}
\nu_i(f)<\nu(f)\qquad\text{ for all }i\in\mathbb N.\label{eq:ulttriangleinfty}
\end{equation}
A polynomial $f$ of smallest degree satisfying (\ref{eq:ulttriangleinfty}) is the first limit key polynomal $Q_\omega$. The passage from
$\{Q_i\}_{i\in\mathbb N}$ to $Q_\omega$ is the first instance of the limit case in the definition below. More generally, if $\ell$ is an ordinal such that the key polynomials $Q_{\ell+i},Q_{\ell+\omega}$ are defined, $(Q_{\ell+i})_{i\in\mathbb N}$ is a pseudo-Cauchy sequence of algebraic type and $Q_{\ell+\omega}$ is the minimal polynomial of a pseudo-limit of $(Q_{\ell+i})_{i\in\mathbb N}$. Now for the formal definition.
\begin{definition}\label{DefKeyPolynom} We say that the set $\{Q_{i}\}_{i\in\Lambda}$ of pre-key polynomials is a \textbf{set of key polynomials} for $\nu$ if it satisfies the following conditions (throughout this definition, $i$ stands for an element of $\Lambda$).

\noi(a)
\begin{equation}
\beta_i\in\tilde\g_0\quad\text{whenever }i\ne\max\ \Lambda.\label{eq:notmax<infty}
\end{equation}
\noi(b) {\bf The successor case.} For each $i\geq1$, if $i$ has an immediate predecessor $i-1$, the $(i-1)$-standard expansion of $Q_i$ has the form
\begin{equation}
Q_i=Q_{i-1}^{\alpha_i}+\sum\limits_{j=0}^{\alpha_i-1}\left(\sum\limits_{{\bar\gamma}_{i-1}}
c_{j,i,{\bar\gamma}_{i-1}}\mathbf{Q}_{i-1}^{{\bar\gamma}_{i-1}}\right)Q_{i-1}^j,\label{tag9.9}
\end{equation}
where:

\noindent (1) if $i=1$, we have $\mathbf{Q}_{i-1}=\mathbf{Q}_0=\emptyset$, $\bar\gamma_{i-1}=\bar\gamma_0=\emptyset$ by convention, $Q_0=x$ and the coefficients in parentheses in (\ref{tag9.9}) are elements of $K$ 

\noindent (2) if $i\ge2$, $\bar\gamma_{i-1}=\left(\gamma_{i'}\right)_{0\le i'<i-1}$, where all but finitely many $\gamma_{i'}$ are equal to 0

\noindent (3) each $c_{j,i,{\bar\gamma}_{i-1}}\mathbf{Q}_{i-1}^{{\bar\gamma}_{i-1}}$ is an $(i-1)$-standard monomial (which is, by definition, $Q_{i-1}$-free)

\noindent (4) the quantity $\nu_K\left(c_{j,i,{\bar\gamma}_{i-1}}\right)+\sum\limits_{q<i-1}\gamma_q\beta_q+j\beta_{i-1}$ is constant for all the monomials
$$
\left(c_{j,i,{\bar\gamma}_{i-1}}\mathbf{Q}_{i-1}^{{\bar\gamma}_{i-1}}\right)Q_{i-1}^j
$$
appearing on the right hand side of (\ref{tag9.9})

\noindent (5) the equation
\begin{equation}
\ini_{\nu}Q_{i-1}^{\alpha_i}+\sum\limits_{j=0}^{\alpha_i-1}\left(\sum\limits_{{\bar\gamma}_{i-1}}\ini_{\nu_K}
c_{j,i,{\bar\gamma}_{i-1}}\ini_{\nu}\mathbf{Q}_{i-1}^{{\bar\gamma}_{i-1}}\right)\ini_{\nu}Q_{i-1}^j=0\label{tag9.10}
\end{equation}
is the minimal algebraic relation satisfied by $\ini_{\nu}Q_{i-1}$ over $G_{\nu_K}[\ini_{\nu}\mathbf{Q}_{i-1}]^*$.
\smallskip

\noi(c) {\bf The limit case.} If $i$ does not have an immediate predecessor then there exists $i_0$ such that $i=i_0+$ and for every such $i_0$ there exists an $i_0$-standard expansion
\begin{equation}
Q_i=\sum\limits_{j=0}^{\alpha_i}c_{j,i_0}Q_{i_0}^{j}.\label{tag9.11}
\end{equation}
Every $i_0$-standard expansion (\ref{tag9.11}) satisfies
\begin{equation}
\nu\left(Q_i\right)>\min\limits_{0\le j\le\alpha_i}\left\{\nu\left(c_{j,i_0}Q_{i_0}^{j}\right)\right\}.\label{eq:almosthomogeneous}
\end{equation}
Moreover, the polynomial $Q_i$ is of the smallest degree among those satisfying (\ref{eq:almosthomogeneous}) for all $i_0$ with $i_0+=i$.
\medskip

If
\begin{equation}
\sup\left\{\left.\beta_{i'}\ \right|\ i'<i\right\}<\infty\label{eq:sup<infty}
\end{equation}
(in particular, whenever $i$ is not the maximal element of $\Lambda$), then
\begin{equation}
c_{\alpha_i,i_0}=1\label{eq:limitmonic}
\end{equation}
and
\begin{equation}
\min\limits_{0\le j\le\alpha_i}\left\{\nu\left(c_{j,i_0}Q_{i_0}^{j}\right)\right\}=\alpha_i\beta_{i_0}=
\nu\left(c_{0,i_0}\right).\label{eq:nui0ofQi}
\end{equation}
\medskip

\noi(d) If for a certain degree $n$, the set $\{i\in\Lambda\ \mid\ \deg_x(Q_i)=n\}$ is infinite, then the set
$$
\{\nu(Q_i) \mid\ i\in\Lambda,\ \deg_x(Q_i)=n\}
$$
is cofinal in $\{\nu(f) \mid\ f\in K[x],\ f\ \text{is\ monic},\ \deg_x(f)=n\}$.
\medskip

An element $Q_i$ of the set $\{Q_{i}\}_{i\in\Lambda}$ is called a \textbf{key polynomial}.
\end{definition}
\begin{remark} It follows from Definition \ref{DefKeyPolynom} that the elements $\beta_i$ are strictly increasing with $i$.
\end{remark}
\begin{proposition}{\label{Proposition9.12}} Let $i$ be an ordinal and $t$ a positive integer. Assume that $i+t\in\Lambda $, so that the key polynomials $\mathbf{Q}_{i+t+1}$ are defined, and that $\alpha_i=\dots=\alpha_{i+t}=1$. Then every $(i+t)$-standard expansion does not involve any $Q_q$ with $i\le q<i+t$. In particular, a $Q_i$-free $i$-standard expansion is the same thing as a $Q_{i+t}$-free $(i+t)$-standard expansion.
\end{proposition}

\begin{proof} (\ref{tag9.6}) implies that for $i\le q< i+t$, $Q_q$ cannot appear in an $(i+t)$-standard expansion with a positive exponent.
\end{proof}

We will frequently use this fact in the sequel without mentioning it explicitly.
\medskip

For $i\in\Lambda$ and $\alpha\in\mathbb N$, let $G_{<\alpha}$ denote the $G_{\nu_K}$-subalgebra of $G_\nu$, generated by elements of the form $\ini_\nu f$, $\deg_xf<\alpha$.

For $i\in\Lambda$, put $\bar\alpha_i:=\deg_xQ_i$.

For the rest of this section, we assume that we have a set of key polynomials $\{Q_{i}\}_{i\in\Lambda}$ for $\nu$ as above and derive some properties of this set.
\begin{proposition}{\label{Proposition9.17}} For each $i\in\Lambda$ we have:
\begin{enumerate}
\item[(1)] If $\sup\left\{\left.\beta_q\ \right|\ q<i\right\}<\infty$ (in particular, whenever $i\ne\max\Lambda$) then
\begin{equation}
\bar\alpha_i=\prod\limits_{j\le i}\alpha_j.\label{tag9.18}
\end{equation}
\item[(2)] Take an element $z\in K[x]$. Assume that $z$ admits a $Q_{i}$-free $i$-standard expansion. Then
\begin{equation}
\deg_xz<\bar\alpha_i.\label{tag9.19}
\end{equation}
\end{enumerate}
\end{proposition}
\begin{proof} We use transfinite induction on $i$. For the base of the induction, consider the case $i=0$. (\ref{tag9.18}) says that
\begin{equation}
\bar\alpha_0=\deg_xx=\alpha_0=1;\label{eq:degx=1}
\end{equation}
this follows immediately from (\ref{eq:alpha0=1}) and (\ref{eq:Q0=x}). By (\ref{eq:degx=1}) and (\ref{tag9.6}), every monomial appearing in the 0-standard expansion $z$ is of degree strictly less than $1=\bar\alpha_0$, that is, is an element of $K$. This proves (\ref{tag9.19}) in the case $i=0$.

Assume given an ordinal $i>0$. Assume that (\ref{tag9.18}) and (\ref{tag9.19}) are known for some ordinal $i_0$ such that $i_0+=i$. If $i=i_0+\omega$ then $i_0$ is inessential, so $\alpha_{i_0+t}=1$ for all $t\in\mathbb N_0$ by Definition \ref{inessential}. Thus in all the cases we have
\begin{equation}
\bar\alpha_{i_0}=\prod\limits_{j<i}\alpha_j.\label{eq:indhypi0}
\end{equation}
By assumption, (\ref{eq:sup<infty}) holds, hence so does (\ref{eq:limitmonic}) by condition (c) of Definition \ref{DefKeyPolynom}.
Thus by (\ref{tag9.9}), (\ref{tag9.11})  and (\ref{tag9.19}) applied to $i_0$ instead of $i$, the term $Q_{i_0}^{\alpha_i}$ in the $i_0$-standard expansion of $Q_i$ has strictly greater degree than all the other terms. Hence
\begin{equation}
\bar\alpha_i=\deg_xQ_{i_0}^{\alpha_i}=\alpha_i\bar\alpha_{i_0}.\label{baralphai=alphaibaralphai0}
\end{equation}
We obtain
$$
\bar\alpha_i=\alpha_i\bar\alpha_{i_0}=\alpha_i\prod\limits_{j<i}\alpha_j=\prod\limits_{j\le i}\alpha_j,
$$
which proves (\ref{tag9.18}).

Let $z$ be as in part (2) of the Proposition. Fix a $Q_i$-free $i$-standard expansion of $z$ and let $i_0$ denote the greatest ordinal such that $Q_{i_0}$ appears in this expansion (by definition of $Q_i$-free, we have $i_0<i$): $z=\sum\limits_{j=0}^{\alpha_i-1}d_{j,i}Q_{i_0}^j$, where each $d_{j,i}$ is a $Q_{i_0}$-free $i_0$-standard expansion. By the induction assumption, we have
$\deg_xd_{j,i}<\bar\alpha_{i_0}$ for all $j$. Combining this with (\ref{baralphai=alphaibaralphai0}), we obtain
$$
\deg_xz<\max\limits_{0\le j<\alpha_i-1}\left\{(j+1)\bar\alpha_{i_0}\right\}\le\alpha_i\bar\alpha_{i_0}=\bar\alpha_i,
$$
as desired.
\end{proof}
\begin{corollary}\label{degreebound} Let $f$ be an $i$-standard expansion of an element of $K[x]$. Then $f$ is $Q_i$-free if and only if
$\deg_xf<\bar\alpha_i$.
\end{corollary}
\begin{proof} Straightforward transfinite induction on $i$.
\end{proof}

\begin{remark}\label{inclusionlimord} Assume that $i\in\Lambda$ is a limit ordinal and take an $i_0\in\Lambda$ such that $i=i_0+$. Since $\alpha_{i_0+1}=1$ by definition, we have $Q_{i_0+1}=Q_{i_0}+z$, where $\ini_\nu z=-\ini_\nu Q_{i_0}$ and
$\deg_xz<\bar\alpha_{i_0}$. Hence
\begin{equation}
G_{\nu_K}\left[\ini_{\nu}\mathbf{Q}_i\right]\subset G_{<\bar\alpha_{i_0}}.\label{eq:G<baralphabis}
\end{equation}
Below (Proposition \ref{GradedAlgebras}) we will show that this inclusion is, in fact, an equality.
\end{remark}
\begin{remark} In \S\ref{Section7} we will show, assuming that $i$ is a limit ordinal and that the set $\{\nu(Q_{i_0})\}_{i=i_0+}$ is bounded in $\tilde\g_0$, that we can choose $i_0$ and $Q_i$ so that $Q_i$ is a weakly affine monic $i_0$-standard expansion of degree $\alpha_i=p^{e_i}$ for a certain integer $e_i$ and, moreover, that there exists a positive element $\bar\beta_i\in\mathbb R$ such that
\begin{align}
\bar\beta_i&>\beta_q\qquad\text{ for all }q<i,\label{tag9.12}\\
\beta_i&\ge \alpha_i\bar\beta_i\quad\text{ and}\label{tag9.13}\\
p^j\bar\beta_i+\nu\left(c_{p^j,i_0}\right)&=\alpha_i\bar\beta_i\quad\text{ for }0\le j\le e_i.\label{tag9.14}
\end{align}
\end{remark}

\begin{remark}\label{eq:coefficientsunique} Take an element $h\in K[x]$. Then $h$ admits a unique $Q_i$-expansion
\begin{equation}
h=\sum\limits_{j=0}^{s_i}d_{j,i}Q_i^j.\label{eq:weakstexph}
\end{equation}
This can be shown by induction on $\deg_xh$. Indeed, let $h=qQ_i+r$ be the Euclidean division of $h$ by $Q_i$. Then $d_{0,i}=r$ and for $j>0$ the coefficients $d_{j,i}$ are nothing but the coefficients of powers of $Q_i$ in the $Q_i$-expansion of $q$.
\end{remark}
\begin{proposition}\label{Qlislstandard} Assume that the set $\{\beta_q\ |\ q<i\}$ is bounded in $\tilde\g_0$ (so that the hypothesis of Proposition \ref{Proposition9.17} (1) is satisfied). Every coefficient $d_{j,i}$ in (\ref{eq:weakstexph}) admits a $Q_i$-free $i$-standard expansion. Writing each coefficient $d_{j,i}$ in this way produces an $i$-standard expansion of $h$. In particular, every element of $K[x]$ admits an $i$-standard expansion.
\end{proposition}
\begin{proof} We use transfinite induction on $i$. For $i=0$ the result is obvious. Assume that $i>0$ and that the Proposition holds for all $i'<i$. By definition of $Q_i$-expansion we have
\begin{equation}
\deg_xd_{j,i}<\bar\alpha_i\quad\text{ for all }j.\label{eq:dij<baralpha}
\end{equation}
Take an ordinal $i_0$ such that $i=i_0+$. By the induction assumption applied to $i_0$ the $Q_{i_0}$-expansion of
\begin{equation}
d_{j,i}=\sum\limits_qd_{q,j,i}Q_{i_0}^q\label{eq:dqji}
\end{equation}
can be made into an $i_0$-standard expansion by writing out a $Q_{i_0}$-free $i_0$-standard expansion of each coefficient
$d_{q,j,i}$. Moreover, by Proposition \ref{Proposition9.17} (1) and (\ref{eq:dij<baralpha}), all the monomials $Q_{i_0}^j$ appearing in (\ref{eq:dqji}) satisfy $j<\alpha_i$. This makes (\ref{eq:dqji}) into a $Q_i$-free $i$-standard expansion of $d_{j,i}$, so (\ref{eq:weakstexph}) can be made into an $i$-standard expansion of $h$, as desired.
\end{proof}
\begin{remark}\label{lstandardnotunique} The $i$-standard expansions (\ref{eq:dqji}) of the $d_{j,i}$ need not, in general, be unique. For example, if $i$ is a limit ordinal, $d_{j,i}$ admits an $i_0$-standard expansion (\ref{eq:dqji}) for each
\begin{equation}
i_0<i\quad\text{ such that }i=i_0+,\label{eq:i0<ii=i_0+}
\end{equation}
but there are countably many choices of $i_0$ satisfying (\ref{eq:i0<ii=i_0+}). If $i\in\mathbb{N}_0$, then the $i$-standard expansion of $h$ is unique. This follows from Remark \ref{eq:coefficientsunique} by induction on $i$. Part of the point of Remarks \ref{eq:coefficientsunique} and \ref{lstandardnotunique} and Proposition \ref{Qlislstandard} is that an $i$-standard expansion is a $Q_i$-expansion {\it together with an additional set of data}, namely, a $Q_i$-free $i$-standard expansion of the coefficients
$d_{j,i}$.
\end{remark}
For each $i\in\Lambda$ we define a map $\nu_i:K(x)^*\rightarrow\Gamma$ as follows. Given a $Q_i$-expansion
\begin{equation}
f=\sum\limits_{j=0}^{s_i}d_{j,i}Q_i^j,\label{eq:istandexpofh}
\end{equation}
put
\begin{equation}
\nu_i(f)=\min\limits_{0\le j\le s_i}\{j\beta_i+\nu(d_{j,i})\}.\label{tag9.15}
\end{equation}
By Remark \ref{eq:coefficientsunique}, the elements $d_{j,i}\in K[x]$ are uniquely determined by $f$, so $\nu_i$ is well defined. We extend $\nu_i$ to all of $K(x)$ by additivity.
\begin{remark}\label{nui<nu} We have
\begin{equation}
\nu_i(f)\le\nu(f)\label{tag9.16}
\end{equation}
by the ultrametric triangle law.
\end{remark}
\begin{remark} It is clear from the definition that the map $\nu_i$ satisfies the ultrametric triangle law.  Below, we will show that it is, in fact, a valuation.
\end{remark}
Next, we give criteria for when the inequality (\ref{tag9.16}) is strict and when it is an equality.
\begin{remark}\label{degreestrictlyless} If $\deg_xf<\deg_xQ_i$ then we have $s_i=0$ in the $i$-standard expansion (\ref{tag9.15}), so $\nu_i(f)=\nu(d_{0,i})=\nu(f)$.
\end{remark}

\begin{remark}\label{nui0qi} Consider ordinals $i_0<i$ and $t\in\mathbb N_0$ such that $i=i_0+$, $i+t\in\Lambda$ and for each ordinal $i'$ such that $i<i'\le i+t$ we have
\begin{equation}
\alpha_{i'}=1.\label{alphai'=1}
\end{equation}

(1) If (\ref{eq:sup<infty}) holds, we have $\nu_{i_0}(Q_i)=\alpha_i\beta_{i_0}$ by (4) of Definition \ref{DefKeyPolynom} (b) and (\ref{eq:nui0ofQi}) (recall that (\ref{eq:sup<infty}) can only fail if $i=\max\ \Lambda$).

(2) By (\ref{tag9.9}), (\ref{alphai'=1}) and induction on $t$,
\begin{equation}
Q_{i+t}=Q_i+\sum\limits_{q=0}^{t-1}z_{i+q},\quad\text{ where }\deg\ z_{i+q}<\bar\alpha_i.\label{eq:iexpansionofQi+t}
\end{equation}
Hence every $i_0$-standard expansion of $Q_{i+t}$ contains the term $Q_{i_0}^{\alpha_i}$. We obtain $\nu_{i_0}(Q_{i+t})\le\alpha_i\beta_{i_0}$.
\medskip

Below (Corollary \ref{nui0Qi+t}) we will see that under a mild additional hypothesis on $i_0$ the last inequality is, in fact, an equality.
\end{remark}
\begin{remark}\label{nuQi=nuellQi} Consider ordinals $i,\ell\in\Lambda$, $i<\ell$. Then $\nu_\ell(Q_i)=\nu_i(Q_i)=\nu(Q_i)$. Indeed, if $\bar\alpha_i<\bar\alpha_\ell$ our statement holds by definition of $\nu_\ell$.

Assume that $\bar\alpha_i=\bar\alpha_\ell$. Then
\begin{equation}
Q_\ell=Q_i+z\label{Qell=Qi+z}
\end{equation}
where $\deg_xz<\bar\alpha_i$ and $\nu(z)=\beta_i<\beta_\ell$. Now, $Q_i=Q_\ell-z$ is an $\ell$-expansion of $Q_i$, so
$\nu_\ell(Q_i)=\min\{\beta_i,\nu(z)\}=\nu(z)=\beta_i$, as desired.
\end{remark}
\begin{proposition}{\label{Proposition9.18}} For a pair of ordinals $i_0<i $ such that $i=i_0+$ we have
\begin{equation}
\nu_{i_0}(Q_{i})<\beta_i.\label{eq:nu0Qi<nuQi}
\end{equation}
\end{proposition}
\begin{proof} Let
\begin{equation}
Q_i=\sum\limits_{j=0}^{\alpha_i}c_{j,i_0}Q_{i_0}^j\label{eq:i0standardQi}
\end{equation}
be the $i_0$-standard expansion of $Q_i$. If $i_0=i-1$ is the immediate predecessor of $i$ then by (\ref{tag9.9})--(\ref{tag9.10}) the quantity $\nu(Q_i)=\nu_i(Q_i)$ is strictly greater than the common value of $\nu\left(c_{j,i_0}Q_{i_0}^j\right)$ (cf. (4) of Definition \ref{DefKeyPolynom}), where $j$ ranges over the elements of $\{0\do\alpha_i\}$ for which $c_{j,i_0}\ne0$. This common value is, by definition, $\nu_{i_0}(Q_i)$. If $i$ does not admit an immediate predecessor, the desired strict inequality is nothing but (\ref{eq:almosthomogeneous}).
\end{proof}

\begin{corollary}\label{notcomplete} If $i\in\Lambda$ is not the maximal element of $\Lambda$, then there exists $f\in K[x]$ such that
$\nu_i(f)<\nu(f)$.
\end{corollary}
\begin{proposition}\label{monom<betai} Fix an ordinal $i\in\Lambda$ and let $\theta=\mathbf{Q}_i^{{\bar\gamma}_i}$ be a
$Q_i$-free $i$-standard monomial. Then
\begin{equation}
\nu\left(\theta\right)<\beta_i.
\end{equation}
\end{proposition}
\begin{proof} If (\ref{eq:sup<infty}) does not hold, we have $\beta_i>\tilde\g_0$ and the result follows immediately. Therefore we may assume that (\ref{eq:sup<infty}) holds. We proceed by transfinite induction on $i$. If $i=0$ then $\theta=1$ and $\beta_0=\nu(x)>0$, so the result holds. Assume that $i>0$ and the Proposition holds for all the ordinals $i'<i$. Let $i_0$ be the smallest ordinal such that $i_0+=i$ and $\gamma_{i'}=0$ for all $i'$ such that $i_0<i'<i$. Write $\theta=\mathbf{Q}_{i_0}^{{\bar\gamma}_{i_0}}Q_{i_0}^{\gamma_{i_0}}$ where
\begin{equation}
\gamma_{i_0}<\alpha_{i}\label{eq:alphai0<gammai0}
\end{equation}
and $\mathbf{Q}_{i_0}^{{\bar\gamma}_{i_0}}$ is a $Q_{i_0}$-free $i_0$-standard monomial. By the induction assumption,
$\nu\left(\mathbf{Q}_{i_0}^{{\bar\gamma}_{i_0}}\right)<\beta_{i_0}$. Combining this with (\ref{eq:alphai0<gammai0}), Remark \ref{nui0qi} and Proposition \ref{Proposition9.18}, we obtain
$$
\nu\left(\theta\right)<\beta_{i_0}+\gamma_{i_0}\beta_{i_0}\le\alpha_i\beta_{i_0}<\beta_i,
$$
as desired.
\end{proof}

\begin{proposition}\label{QiinRnux} Take an ordinal $i\in\Lambda$ and assume that (\ref{eq:sup<infty}) holds. Then $Q_i\in
R_{\nu_K}[x]$.
\end{proposition}
\begin{proof} We use transfinite induction on $i$. For $i=0$ we have $Q_0=x$ and the result is clear. Assume that $i>0$ and the Proposition holds for all the ordinals $i'<i$. Take an ordinal $i_0$ such that $i=i_0+$ and let
\begin{equation}
Q_i=Q_{i_0}^{\alpha_i}+\sum\limits_{j=0}^{\alpha_i-1}\left(\sum\limits_{{\bar\gamma}_{i_0}}
c_{j,i,{\bar\gamma}_{i_0}}\mathbf{Q}_{i_0}^{{\bar\gamma}_{i_0}}\right)Q_{i_0}^j,\label{i0standardofQi}
\end{equation}
be an $i_0$-standard expansion of $Q_i$. By the induction assumption, it is enough to prove that
\begin{equation}
\nu\left(c_{j,i,{\bar\gamma}_{i_0}}\right)\ge0\label{eq:coefficients>0}
\end{equation}
for all the choices of $j,i,{\bar\gamma}_{i_0}$. Fix $j,i,{\bar\gamma}_{i_0}$ appearing in one of the terms in (\ref{i0standardofQi}). We have
\begin{equation}
\nu\left(\left(\sum\limits_{{\bar\gamma}_{i_0}}c_{j,i,{\bar\gamma}_{i_0}}\mathbf{Q}_{i_0}^{{\bar\gamma}_{i_0}}\right)
Q_{i_0}^j\right)\ge\nu\left(Q_{i_0}^{\alpha_i}\right)\label{eq:monomiallowerbound}
\end{equation}
by Definition \ref{DefKeyPolynom} (b) (4) and (\ref{eq:nui0ofQi}) and
$\nu\left(\mathbf{Q}_{i_0}^{{\bar\gamma}_{i_0}}\right)<\beta_{i_0}$ by Proposition \ref{monom<betai} with $i$ replaced by $i_0$. Combining the last inequality with (\ref{eq:monomiallowerbound}), we obtain
$\nu\left(c_{j,i,{\bar\gamma}_{i_0}}\right)\ge\alpha_i\beta_{i_0}-(j+1)\beta_{i_0}\ge0$, as desired.
\end{proof}
Let $\beta^{*} $ be a non-negative element of $\g$. Keep the notation of (\ref{eq:istandexpofh}). We denote
\[
S_i(\beta^{*}, f):=\left\{j\in\{0,\dots,s_i \}\ \left|\
j\beta^{*}+\nu(d_{j,i})=\min\limits_{0\le k\le s_i}\left\{k\beta^{*}+\nu(d_{k,i})\right\}\right.\right\}.
\]
Let $\bar{Q}_i$ be a new variable.
\smallskip

\noindent{\bf Notation:}\begin{align}
S_i (f):&=S_i (\beta_i,f).\label{tag9.24}\\
\delta_i(f):&=\max\ S_i(f)\\
\ini_if:&=\sum\limits_{j\in S_i (f)}\ini_{\nu}d_{j,i}\bar Q_i ^j;\label{tag9.25}
\end{align}
the polynomial $\ini_i f$ is quasi-homogeneous in $G_{\nu_K}\left[\ini_{\nu}\mathbf{Q}_i ,\bar Q_i \right]$, where the weight assigned to $\bar Q_i$ is $\beta_i$.
\begin{remark}\label{delta0nu=nui} If $\delta_i(f)=0$ then $\nu\left(d_{0,i}\right)<\nu\left(d_{j,i}Q^j_i\right)$ for $j>0$, so
$\nu(f)=\nu(d_{0,i})=\nu_i(f)$, where the first equality holds by the ultrametric triangle law.
\end{remark}
\begin{proposition}\label{GradedAlgebras} (1) For each ordinal $i\in\Lambda$ we have
\begin{equation}
G_{<\bar\alpha_i}=G_{\nu_K}\left[\ini_{\nu}\mathbf{Q}_i\right].\label{eq:G<baralpha}
\end{equation}

(2) The elements $\ini_\nu f$, $\deg_xf<\bar\alpha_i$, generate $G_{<\bar\alpha_i}$ as a $G_{\nu_K}$-module.
\end{proposition}
\begin{proof} We use transfinite induction on $i$. For $i=0$ we have $\bar\alpha_0=1$, both sides of (\ref{eq:G<baralpha}) coincide with $G_{\nu_K}$ and all the statements are clear. Assume that $i>0$ and that the result is known for all the ordinals strictly smaller than $i$. To prove that
\begin{equation}
G_{\nu_K}\left[\ini_{\nu}\mathbf{Q}_i\right]\subset G_{<\bar\alpha_i},\label{eq:G<baralphathird}
\end{equation}
we distinguish two cases. If $i$ is a limit ordinal, (\ref{eq:G<baralphathird}) follows immediately from (\ref{eq:G<baralphabis}).

Suppose that $i$ is not a limit ordinal. We have
$$
G_{\nu_K}\left[\ini_{\nu}\mathbf{Q}_i\right]=G_{\nu_K}\left[\ini_{\nu}\mathbf{Q}_{i-1}\right]\left[\ini_{\nu}Q_{i-1}\right]\cong G_{<\bar\alpha_{i-1}}\left[\ini_{\nu}Q_{i-1}\right],
$$
where the last isomorphism is given by the induction assumption. Since $\alpha_{i-1}$ is the degree of the minimal polynomial of
$\ini_{\nu}Q_{i-1}$ over $G_{\nu_K}\left[\ini_{\nu}\mathbf{Q}_{i-1}\right]$ (Definition \ref{DefKeyPolynom} (b) (5)), every element of
$G_{<\bar\alpha_{i-1}}\left[\ini_{\nu}Q_{i-1}\right]$ admits an $(i-1)$ standard expansion with all the exponents of $Q_{i-1}$ strictly smaller than $\alpha_{i-1}$ (by Lemma \ref{el4.4}, with $\ini_\nu Q_{i-1}$ playing the role of $x$). Such an element is represented by a polynomial $f\in K[x]$ admitting a $Q_i$-free $i$-standard expansion, that is, a polynomial such that $\deg_xf<\bar\alpha_i$ (Corollary \ref{degreebound}). This proves the inclusion (\ref{eq:G<baralphathird}).

To prove the opposite inclusion, take $f\in K[x]$ with
\begin{equation}
\deg_xf<\bar\alpha_i.\label{eq:degeta<baralpha}
\end{equation}
By Corollary \ref{degreebound} $f$ admits a $Q_i$-free $i$-standard expansion $f=\sum\limits_{j=1}^t\eta_j$, where
$\eta_1\do\eta_t$ are $Q_i$-free $i$-standard monomials. Let $\beta=\min\limits_j\{\nu(\eta_j)\}$ and $\{\theta_1\do\theta_s\}=\left\{\eta_j\ \left|\ j\in\{1\do t\},\nu(\eta_j)=\beta\right.\right\}$.

If $i$ is not a limit ordinal, $\alpha_i$ is the  degree of the minimal polynomial of $\ini_{\nu}Q_{i-1}$ over
$G_{\nu_K}\left[\ini_{\nu}\mathbf{Q}_{i-1}\right]$ (Definition \ref{DefKeyPolynom} (b) (5)). By (\ref{eq:degeta<baralpha}) we have
$\sum\limits_j\ini_\nu\theta_j\ne0$ in $G_\nu$ (otherwise we would have an algebraic relation satisfied by $\ini_{\nu}Q_{i-1}$ over $G_{\nu_K}\left[\ini_{\nu}\mathbf{Q}_{i-1}\right]$ of degree strictly less than $\alpha_i$), hence $\nu(f)=\nu(\theta_j)$ (for all $j$) and $\sum\limits_j\ini_\nu\theta_j=\ini_\nu f$ (see Proposition \ref{Proposition4.8}). If $i$ is a limit ordinal, by Definition \ref{DefKeyPolynom} (c) there exists an ordinal $i_0$ satisfying $i_0+=i$ such that, letting
$f=\sum\limits_{j=0}^{s_i}c_{j,i_0}Q_{i_0}^{j}$ be an $i_0$-standard expansion of $f$ (where $s_i<\alpha_i$), we have
\begin{equation}
\nu\left(f\right)=\min\limits_{0\le j\le s_i}\left\{\nu\left(c_{j,i_0}Q_{i_0}^{j}\right)\right\},\label{eq:almosthomogeneousbis}
\end{equation}
so
\begin{equation}
\ini_\nu f=\sum\limits_{j\in S_{i_0}(f)}\ini_\nu\left(c_{j,i_0}Q_{i_0}^{j}\right)\label{eq:almosthomogeneousthird}
\end{equation}
by Proposition \ref{Proposition4.8}. In both cases we have proved that the left hand side of (\ref{eq:G<baralpha}) is contained in the right hand side.

To prove (2) of the Proposition, first assume that $i$ is not a limit ordinal. By the induction assumption elements of the form $\ini_xg$, $g\in K[x]$, $\deg_xg<\bar\alpha_{i-1}$, generate $G_{<\bar\alpha_{i-1}}$ as a $G_{\nu_K}$-module. Elements $\ini_\nu f$,
$\deg_xf<\bar\alpha_i$ are precisely those that admit homogeneous $i$-standard expansions where $Q_{i-1}$ appears with exponents strictly smaller than $\alpha_i$. By part (1) of this Proposition such standard expansions generate $G_{<\bar\alpha_i}$ as a $G_{\nu_K}$-algebra, hence also as a $G_{\nu_K}$-module by Lemma \ref{el4.4} (since $\alpha_i$ is the degree of the minimal polynomial of $\ini_\nu Q_{i-1}$ over $G_{\nu_K}\left[\ini_{\nu}\mathbf{Q}_{i-1}\right]$).

It remains to consider the case when $i$ is a limit ordinal. Take an ordinal $i_0$ such that $i=i_0+$. We have
\begin{equation}
G_{\nu_K}\left[\ini_{\nu}\mathbf{Q}_i\right]\subset G_{<\bar\alpha_{i_0}}\subset G_{<\bar\alpha_i}\label{eq:iequalsi0},
\end{equation}
by Remark \ref{inclusionlimord} hence both inclusions are equalities by part (1) of this Proposition. By the induction assumption elements of the form $\ini_xg$, $g\in K[x]$, $\deg_xg<\bar\alpha_{i-1}$, generate $G_{<\bar\alpha_{i_0}}=G_{<\bar\alpha_i}$ as a $G_{\nu_K}$-module, which completes the proof of (2).
\end{proof}
\medskip

\begin{proposition}\label{Proposition9.26b} Take an element $h$ of $K[x]$ and an ordinal $i\in\Lambda\setminus\{0\}$. Assume that
\begin{equation}
\deg_xh<\bar\alpha_i.\label{eq:degh<alphai}
\end{equation}
Then there exists an ordinal $i_0$ such that $i\in\{i_0+1,i_0+\}$ and $\nu(h)=\nu_{i_0}(h)$.
\end{proposition}
\begin{proof} For an ordinal $i_0<i$ we will denote by
\begin{equation}
h=\sum\limits_{j=0}^sc_{j,i_0}Q_{i_0}^j,\label{tag9.32bis}
\end{equation}
the $i_0$-standard expansion of $h$. Let $S_{i_0}(h)$ be as defined in (\ref{tag9.24}). First, assume that $i$ is not a limit ordinal; put $i_0=i-1$. By (\ref{eq:degh<alphai}) we have $s<\alpha_i$ in (\ref{tag9.32bis}). Since the degree of $\ini_{\nu}Q_{i-1}$ over $G_{\nu_K}[\ini_{\nu}\mathbf{Q}_{i-1}]^*$ is $\alpha_i$ by (\ref{tag9.10}) and $\sum\limits_{j\in S_{i-1}(h)}\ini_{\nu}c_{j,i-1}\bar Q_{i-1}^j$ is a polynomial of degree strictly less than $\alpha_i$ in $\bar Q_{i-1}$, we see that
$\sum\limits_{j\in S_{i-1}(h)}\ini_{\nu}c_{j,i-1}\ini_{\nu}Q_{i-1}^j\ne0$ in $G_{\nu}$. The result now follows from Proposition \ref{Proposition4.8}.

If $i$ is a limit ordinal, the results follows immediately from the fact that by Definition \ref{DefKeyPolynom} $Q_i$ is of the smallest degree among those polynomials that satisfy (\ref{eq:almosthomogeneous}).
\end{proof}
\begin{definition}{\label{de9.1}} Let $\{Q_{i}\}_{i\in\Lambda}$ be a set of key polynomials. We say that $\{Q_{i}\}_{i\in\Lambda}$ is {\bf complete} for $\nu$ (or that  $\{Q_{i}\}_{i\in\Lambda}$ is a {\bf complete} set of key polynomials  for $\nu$) if for each
$\beta\in\g$ the additive group $\mathbf{P}_\beta\cap K[x]$ is generated by standard monomials of the form
$a\prod\limits_{j=1}^sQ_{i_j}^{\gamma_j}$, $a\in K$, such that $\sum\limits_{j=1}^s\gamma_j\nu(Q_{i_j})+\nu_K(a)\ge\beta$. The collection $\mathbf{Q}=\{Q_i\}_{i\in\Lambda}$ is said to be $\tilde{\g}_0$-{\bf complete} if for all $\beta\in\tilde\g_0$ every polynomial $f\in K[x]$ with $\nu(f)=\beta$ belongs to the additive group generated by standard monomials of the form
$a\prod\limits_{j=1}^sQ_{i_j}^{\gamma_j}$, $a\in K$, such that $\sum\limits_{j=1}^s\gamma_j\nu(Q_{i_j})+\nu_K(a)\ge\beta$.
\end{definition}

Note, in particular, that if $\mathbf{Q}$ is a complete set of key polynomials then their images $\ini_{\nu}Q_i\in G_{\nu}$ rationally generate $G_{\nu}$ over $G_{\nu_K}$; if $\mathbf{Q}$ is a $\tilde{\g}_0$-complete set of key polynomials then their images
$\ini_{\nu}Q_i\in G_{\nu}$ rationally generate $\tilde G_\nu$ over $G_{\nu_K}$ (see page 3 for the definition of $\tilde G_\nu$).
\medskip

\begin{remark}\label{remark_complete} The set $\mathbf Q$ is a complete set of key polynomials if and only if for each polynomial $f\in K[x]$ there exists an ordinal $\ell$ such that $\nu(f)=\nu_\ell(f)$. Indeed, assume that such an ordinal exists for all $f$. Take any
$\beta\in\g$ and let $f\in \mathbf{P}_\beta\cap K[x]$. Put $\beta'=\nu(f)$ and let $\ell$ be such that $\beta'=\nu(f)=\nu_{\ell}(f)$. Write $f=\sum\limits_{j=0}^{s_\ell}d_{j,\ell}Q_{\ell}^{j}$, where each $d_{j,\ell}Q_{\ell}^{j}\in \mathbf{P}_{\beta'}\cap K[x]\subset \mathbf{P}_\beta\cap K[x]$.

Conversely, take $f\in K[x]$. Let $\beta=\nu(f)$. Write $f$ as a finite sum
\begin{equation}
f=\sum\limits_\gamma d_\gamma\mathbf Q^\gamma\quad d_\gamma\in K,\label{eq:sumofstandardmonom}
\end{equation}
of standard monomials with $\nu\left(d_\gamma\mathbf Q^\gamma\right)\ge\beta$ for all $\gamma$ such that $d_\gamma\ne0$. Let
$\ell$ denote the greatest ordinal such that $Q_\ell$ appears in one of the monomials $d_\gamma\mathbf Q^\gamma$. Then (\ref{eq:sumofstandardmonom}) can be rewritten as $f=\sum\limits_{j=0}^sd_jQ_\ell^j$, where each $d_j$ is a $Q_\ell$-free
$\ell$-standard expansion; in particular, $\deg_xd_j<\bar\alpha_\ell$. We obtain
$$
\beta=\nu(f)\geq\nu_{\ell}(f)=\min\limits_j\left\{\nu\left(d_j\mathbf Q^j_\ell\right)\right\}\ge
\min\limits_\gamma\left\{\nu\left(d_\gamma\mathbf Q^\gamma\right)\right\}\geq\beta.
$$
This also shows that if $\mathbf{Q}$ is a complete set of key polynomials then $\mathbf{Q}$ is not strictly contained as an initial segment in any other set of key polynomials; in other words, the construction of key polynomials cannot be continued beyond
$\mathbf{Q}$.
\end{remark}
\begin{proposition}\label{notingamma}
If $\Lambda$ has a maximal element $i$ and
\begin{equation}
\beta_i\notin\tilde\g_0,\label{eq:betanoting0}
\end{equation}
then $\ini_\nu Q_i$ is transcendental over $G_{\nu_K}[\ini_{\nu}\mathbf Q_i]$.
\end{proposition}
\begin{proof} Take a non-zero polynomial $F(\bar Q_i)=\sum\limits_{j=0}^s\bar c_j\bar Q_i^j\in G_{\nu_K}[\ini_{\nu}\mathbf Q_i][\bar Q_i]$. By (\ref{eq:betanoting0}) the terms in the expression $F(\ini_\nu Q_i)=\sum\limits_{j=0}^s\bar c_j\ini_\nu Q_i^j$ have different orders for different $j$ (where, by definition, $\ord\ \bar Q_i=\beta_i$); in particular, there exists a unique $j_{min}\in\{0\do s\}$ that minimizes $\ord\ \bar c_j\ini_\nu Q_i^j$. By the ultrametric triangle inequality, we have $F(\ini_\nu Q_i)\ne0$, as desired.
\end{proof}
\begin{proposition}\label{transcendentalcomplete} Take an ordinal $i\in\Lambda$. Assume that $\ini_\nu Q_i$ is transcendental over\linebreak $G_{\nu_K}[\ini_{\nu}\mathbf Q_i]$. Then the set $\{Q_\ell\}_{\ell\in\Lambda}=\{Q_\ell\}_{\ell\leq i}$ of key polynomials is complete for $\nu$.
\end{proposition}
\begin{proof} Take an element $h\in K[x]$. Let
\begin{equation}
h=\sum\limits_{j=0}^sc_{j,i}Q_i^j\label{eq:istandardexpansionofh}
\end{equation}
be the $Q_i$-expansion of $h$. We want to prove that
$$
\nu_i(h)=\nu(h).
$$
Replacing $h$ by $\sum\limits_{j\in S_i(h)}c_{j,i}Q_i^j$ does not change the problem. Thus we may assume that all the terms
$c_{j,i}Q_i^j$ appearing in (\ref{eq:istandardexpansionofh}) have value $\nu_i(h)$. Since $\ini_\nu Q_i$ is transcendental over $G_{\nu_K}[\ini_{\nu}\mathbf Q_i]$, we have $\sum\limits_j\ini_\nu(c_{j,i}Q_i^j)\ne0$. The Proposition now follows from Proposition \ref{Proposition4.8}.
\end{proof}
\begin{proposition}\label{Proposition9.26a} Take an element $h$ of $K[x]$ and an ordinal $i\in\Lambda$. Assume that at least one of the following conditions holds:
\begin{enumerate}
\item[(1)]
\begin{equation}
\nu(h)<\beta_i\label{tag9.31}
\end{equation}
and
\begin{equation}
h\in R_{\nu_K}[x]\label{eq:hinrnu0x}
\end{equation}
or
\item[(2)]
\begin{equation}
\sup\left\{\beta_{i'}\ \left|\ i'<i\right.\right\}=\infty.\label{eq:sup=infty}
\end{equation}
\end{enumerate}
Then $\nu(h)=\nu_i(h)$.
\end{proposition}

\begin{proof} Let
\begin{equation}
h=\sum\limits_{j=0}^sc_jQ_i^j,\label{tag9.32}
\end{equation}
be an $i$-standard expansion of $h$. First, assume that $\left\{\left.\beta_{i'}\ \right|\ i'<i\right\}$ is unbounded in $\tilde\g_0$; in particular,
\begin{equation}
\beta_i>\tilde\g_0.\label{eq:betai>gamma0}
\end{equation}
Then the Proposition holds by Propositions \ref{notingamma} and \ref{transcendentalcomplete}.

Next, suppose that (\ref{tag9.31})--(\ref{eq:hinrnu0x}) hold and $\sup\left\{\beta_{i'}\ \left|\ i'<i\right.\right\}<\infty$. Since (\ref{tag9.32}) is obtained from $h$ by iterating Euclidean division by $Q_i$ which is monic, and in view of Proposition \ref{QiinRnux}, we have
\begin{equation}
\nu(c_j)\ge0\quad\text{ for all }j\in\{0\do s\}.\label{tag9.33}
\end{equation}
By definition of standard expansion, each $c_j$ in (\ref{tag9.32}) is a $Q_i $-free $i$-standard expansion. Then $\nu_i(c_j)=\nu(c_j)$ for $0\leq j\leq s$. By (\ref{tag9.31}) and (\ref{tag9.33}),
\begin{equation}
\nu\left(c_jQ_i^j\right)=\nu(c_j)+j\beta_i>\nu(h)\qquad\text{ for }j>0.\label{tag9.34}
\end{equation}
so
\begin{equation}
\nu(c_0)=\nu(h)<\nu(c_j)+j\beta_i=\nu_i(c_j)+j\beta_i\quad\text{ for all }j>0.\label{eq:c0wins}
\end{equation}
In other words, in the sum (\ref{tag9.32}) the $\nu_i$-value (resp. the $\nu$-value) $\nu_i(c_0)=\nu(c_0)$ of $c_0$ is strictly smaller than the $\nu_i$-values (resp. the $\nu$-values) $\nu_i(c_jQ_i^j)=\nu(c_jQ_i^j)$ of all the other terms. It is well known and follows easily from the ultrametric triangle law that in this situation we have $\nu_i(h)=\nu_i(c_0)=\nu(c_0)=\nu(h)$, as desired.
\end{proof}

For $b\in\mathbb{N}_0$, let $\partial_b$ denote the $b$-th Hasse derivative with respect to $x$.

\begin{proposition} For all $b\in\mathbb{N}$ and all $i\in\Lambda$, there exists an ordinal $i_0<i$ such that 
$$
i\in\{i_0+1,i_0+\}
$$
and
\begin{equation}
\nu_{i_0}(\partial_bQ_i)=\nu(\partial_bQ_i).\label{tag9.17}
\end{equation}
\end{proposition}
\begin{proof} If $i$ is not a limit ordinal, put $i_0=i-1$. If $i$ is a limit ordinal, take $i_0$ such that
\begin{equation}\label{eq:i=i0+}
i=i_0+.
\end{equation}
The result now follows from the fact that $\deg_x\partial_bQ_i<\deg_xQ_i$, where in the case when $i$ is a limit ordinal we must take $i_0$ sufficiently large subject to (\ref{eq:i=i0+}).
\end{proof}
\begin{proposition}\label{nuiisavaluation} (1) The map $\nu_i$ is a valuation.

(2) For every ordinal $q<i$ and every $f\in K[x]$ we have $\nu_q(f)\le\nu_i(f)\le\nu(f)$.
\end{proposition}
\begin{proof} It is obvious that $\nu_i$ satisfies the ultrametric triangle law. To prove (1) of the Proposition, we must prove the equality
\begin{equation}
\nu_i(fg)=\nu_i(f)+\nu_i(g).\label{eq:pseudoval}
\end{equation}
For an ordinal $i'\in\Lambda$ let (1)$_{i'}$ and (2)$_{i'}$ denote, respectively, (1) and (2) of the Proposition with $i$ replaced by $i'$. We proceed by transfinite induction on $i$. (1)$_0$ is easy to prove and (2)$_0$ is vacuously true.
 
Assume that (1)$_{i'}$ and (2)$_{i'}$ hold for all ordinals $i'<i$. To prove (1)$_i$ and (2)$_i$, we start with some preliminary lemmas. For future reference we state some of the lemmas in slightly greater generality than needed for the proof of Proposition \ref{nuiisavaluation}.

Consider an ordinal $i_0<i$ and $t\in\mathbb N_0$ such that $i=i_0+$, $i+t\in\Lambda$ and for each ordinal $i'$ such that $i<i'\le i+t$ we have
\begin{equation}
\alpha_{i'}=1.\label{alphai'=1bis}
\end{equation}
\begin{lemma}{\label{el11.4}} Consider two terms of the form $dQ_{i+t}^j$ and $d'Q_{i+t}^{j'}$, where $j,j'\in\mathbb{N}_0$ and $d$ and $d'$ are $Q_{i+t}$-free $(i+t)$-standard expansions. Assume that
\begin{equation}
\nu_{i_0}\left(dQ_{i+t}^j\right)\le\nu_{i_0}\left(d'Q_{i+t}^{j'}\right),\label{tag11.10}
\end{equation}
\begin{equation}
\nu\left(dQ_{i+t}^j\right)\ge\nu\left(d'Q_{i+t}^{j'}\right),\label{tag11.11}
\end{equation}
\begin{equation}
\nu_{i_0}(d)=\nu(d)\label{eq:nu0d=nud}
\end{equation}
and
\begin{equation}
\nu_{i_0}(d')=\nu(d')\label{eq:nu0d'=nud'}
\end{equation}
Then $j\ge j'$. If at least one of the inequalities (\ref{tag11.10}), (\ref{tag11.11}) is strict then $j>j'$.
\end{lemma}
\begin{remark} If $i=i_0+1$ then assumptions (\ref{eq:nu0d=nud}) and (\ref{eq:nu0d'=nud'}) hold automatically by Corollary \ref{degreebound} and the definition of $\nu_{i_0}$. In the general case, Proposition \ref{Proposition9.26b} and (2)$_q$ for $q<i$  imply that there exists $i_0$ such that $i_0+=i$ and (\ref{eq:nu0d=nud}) and (\ref{eq:nu0d'=nud'}) hold.
\end{remark}
\begin{proof}[Proof of Lemma \ref{el11.4}] We have
\begin{equation}
\nu_{i_0}(Q_{i+t})\le\alpha_i\beta_{i_0}<\beta_{i+t}\label{eq:alphai+1betai<betai+1}
\end{equation}
by Remark \ref{nui0qi} and Proposition \ref{Proposition9.18}.  Since $\nu_{i_0}$ is a valuation by (1)$_{i_0}$, the inequalities (\ref{tag11.10})--(\ref{tag11.11}) can be rewritten as
\begin{equation}
\nu(d)+j\nu_{i_0}(Q_{i+t})\le\nu(d')+j'\nu_{i_0}(Q_{i+t})\label{djnui<d'j'nui}
\end{equation}
and
\begin{equation}
\nu(d)+j\beta_{i+t}\ge\nu(d')+j'\beta_{i+t}\label{djnui>d'j'nui}
\end{equation}
Subtract (\ref{djnui<d'j'nui}) from (\ref{djnui>d'j'nui}) and use (\ref{eq:alphai+1betai<betai+1}). The result follows.
\end{proof}
\begin{corollary}\label{nui0Qi+t} Keep the notation of the last lemma. Assume that
\begin{equation}
\nu_{i_0}\left(\sum\limits_{q=0}^{t-1}z_{i+q}\right)=\nu\left(\sum\limits_{q=0}^{t-1}z_{i+q}\right)\label{eq:nuz=nui0z}
\end{equation}
in the notation of (\ref{eq:iexpansionofQi+t}) (note that by Proposition \ref{Proposition9.26b} such an $i_0$ always exists). We have
\begin{equation}
\nu_{i_0}(Q_{i+t})=\alpha_i\beta_{i_0}\label{eq:alphai+1betai}
\end{equation}
\end{corollary}
\begin{proof} By Proposition \ref{Qlislstandard} we may assume that (\ref{eq:iexpansionofQi+t}) is an $i_0$-standard expansion of $Q_{i+t}$. The leading term of this $i_0$-standard expansion is $Q_{i_0}^{\alpha_i}$. We have
$\alpha_i\beta_{i_0}=\nu_{i_0}\left(Q_{i_0}^{\alpha_i}\right)$. Now Lemma \ref{el11.4}, applied with $t=j=0$ and $j'=1$, implies that $\nu_{i_0}\left(Q_{i_0}^{\alpha_i}\right)<\nu_{i_0}\left(\sum\limits_{q=0}^{t-1}z_{i+q}\right)$. This completes the proof.
\end{proof}

Keep the above notation. Let $f$ be a non-zero element of $K[x]$. Let
\begin{equation}
f=\sum\limits_{j=0}^{s_{i_0}}d_{j,i_0}Q_{i_0}^j\label{tag9.23}
\end{equation}
be an $i_0$-standard expansion of $f$, where each $d_{j,i_0}$ is a $Q_{i_0} $-free $i_0$-standard expansion.
\smallskip

Consider an $(i+t)$-standard expansion of $f$:
\begin{equation}
f=\sum\limits_{j=0}^{s_{i+t}}d_{j,i+t}Q_{i+t}^j,\label{tag11.9}
\end{equation}
where the $d_{j,i+t}$ are $Q_{i+t}$-free standard expansions. Let $\delta=\delta_{i_0}(f)$. Let
$$
\mu=\min\limits_{0\le j\le s_{i+t}}\nu_{i_0}\left(d_{j,i+t}Q_{i+t}^j\right)
$$
and
$$
S_{i_0,i+t}=\left\{j\in\{0,\dots,s_{i+t}\}\ \left|\ \nu_{i_0}\left(d_{j,i+t}Q_{i+t}^j\right)=\mu\right.\right\}.
$$
Assume that (\ref{eq:nuz=nui0z}) holds and that
$$
\nu_{i_0}\left(d_{j,i+t}\right)=\nu\left(d_{j,i+t}\right)
$$
for all $j\in S_{i_0,i+t}$ (note that by Proposition \ref{Proposition9.26b} and (2)$_q$, $q<i$, this holds for all sufficiently large ordinals $i_0$ such that $i_0+=i$).
\begin{lemma}{\label{el11.3}}
\begin{enumerate}
\item[(1)] We have
$$
\nu_{i_0}(f)=\mu=\min\limits_{0\le j\le s_{i+t}}\{\nu(d_{j,i+t})+j\alpha_i\beta_{i_0}\}.
$$
\item[(2)] Let $j_0=\max\ S_{i_0,i+t}$. Then $\delta=\alpha_ij_0+\delta_{i_0}(d_{j_0,i+t})$.
\end{enumerate}
\end{lemma}

\begin{proof} (1) The second equality in (1) of the Lemma is given by Corollary \ref{nui0Qi+t} and (1)$_{i_0}$. Let $\theta=\min\ S_{i_0,i+t}$.

Let $\bar f=\sum\limits_{j\in S_{i_0,i+t}}d_{j,i+t}Q_{i+t}^j$. Then $\nu_{i_0}\left(f-\bar f\right)>\mu$, so to prove that
$\nu_{i_0}\left(f\right)=\mu$ it is sufficient to prove that
\begin{equation}
\nu_{i_0}\left(\bar f\right)=\mu.\label{eq:nui0barf=mu}
\end{equation}
The equality (\ref{eq:nui0barf=mu}) holds if and only if
\begin{equation}
\nu_{i_0}\left(Q_{i+t}^{-\theta}\bar f\right)=\mu-\theta\alpha_i\beta_{i_0}.\label{eq:nui0barf=mubis}
\end{equation}
Therefore, without loss of generality, we may assume that $\theta=0$.

Now, all the terms $d_{j,i+t}Q_{i+t}^j$, $j\in S_{i_0,i+t}$, appearing in the $(i+t)$-standard expansion of $\bar f$, have the same
$\nu_{i_0}$-values, so by Lemma \ref{el11.4} they all have different $\nu$-values with
$$
\nu\left(d_{0,i+t}\right)<\nu\left(d_{j,i+t}Q_{i+t}^j\right)
$$
for $j\in S_{i_0,i+t}\setminus\{0\}$. By the ultrametric triangle law we have
$$
\nu\left(\bar f\right)=\nu\left(d_{0,i+t}\right)=\nu_{i_0}\left(d_{0,i+t}\right)=\mu\le\nu_{i_0}\left(\bar f\right),
$$
hence the inequalitiy in this formula is an equality and (\ref{eq:nui0barf=mu}) is proved. This proves (1) of the Lemma.

(2) In view of (1), a term of the form $aQ_{i_0}^\delta$ with $\deg_xa<\deg_xQ_{i_0}$ and
$$
\nu\left(aQ_{i_0}^\delta\right)=\nu_{i_0}(f)
$$
appears in the $Q_{i_0}$-expansion of $\sum\limits_{j\in S_{i_0,i+t}}d_{j,i+t}Q_{i+t,j}$ and $\delta$ is the greatest integer with this property. The $Q_{i_0}$-expansion of $d_{j_0,i+t}Q_{i+t}^{j_0}$ contains a term of the form $aQ_{i_0}^{\alpha_ij_0+\delta_{i_0}(d_{j_0,i+t})}$ with $\deg_xa<\deg_xQ_{i_0}$ and $\nu\left(aQ_{i_0}^{\alpha_ij_0+\delta_{i_0}(d_{j_0,i+t})}\right)=\nu_{i_0}(f)$; $\alpha_ij_0+\delta_{i_0}(d_{j_0,i+t})$ is the greatest power of $Q_{i_0}$ appearing there. The term
$aQ_{i_0}^{\alpha_ij_0+\delta_{i_0}(d_{j_0,i+t})}$ cannot be canceled by contributions from any of $d_{j,i+t}Q_{i+t}^j$, $j\in S_{i_0,i+t}$, $j<j_0$, for reasons of degree. This completes the proof of (2).
\end{proof}

\begin{lemma}\label{nuimonom>nui0f} Keep the notation and assumptions of Lemma \ref{el11.3}. For every $j\in\{1\do s_{i+t}\}$ we have
\begin{equation}
\nu_{i+t}\left(d_{j,i+t}Q_{i+t}^j\right)-\nu_{i_0}(f)\ge\beta_{i+t}-\alpha_i\beta_{i_0}>0.\label{eq:strictinequalityi}
\end{equation}
\end{lemma}
\begin{proof} We have
\begin{equation}
\nu_{i+t}\left(d_{j,i+t}Q_{i+t}^j\right)=\nu\left(d_{j,i+t}Q_{i+t}^j\right)=\nu\left(d_{j,i+t}\right)+j\beta_{i+t}.\label{eq:nui+tmonomial}
\end{equation}
By the ultrametric triangle law applied to the $Q_{i_0}$-expansion of $d_{i+t}$ we have
\begin{equation}
0\le\nu(d_{j,i+t})-\nu_{i_0}(d_{j,i+t})\label{eq:nudi+t}.
\end{equation}
Finally, by  Lemma \ref{el11.3} we have
\begin{equation}
\nu_{i_0}(f)\le\nu_{i_0}\left(d_{j,i+t}Q_{i+t}^j\right)=\nu_{i_0}\left(d_{j,i+t}\right)+j\alpha_i\beta_{i_0}.\label{eq:nui0f}
\end{equation}
Adding up (\ref{eq:nudi+t}) and (\ref{eq:nui0f}), subtracting the result from (\ref{eq:nui+tmonomial}) and using the fact that $j\ge1$, we obtain (\ref{eq:strictinequalityi}).
\end{proof}
\begin{corollary} (1) Conclusion (2)$_i$ holds.

(2) We have $\nu_i(f)=\nu_{i_0}(f)$ if and only if the $Q_i$-expansion of $f$ has non-zero coefficient $d_{0,i}$ satisfying
\begin{equation}
\nu_i\left(d_{0,i}\right)=\nu_{i_0}\left(d_{0,i}\right)=\nu_{i_0}(f).\label{eq:equalityconstantterm}
\end{equation}
\end{corollary}
\begin{proof} (1) follows from (\ref{eq:strictinequalityi}) (with $t=0$) by transfinite induction on $i-q$.

(2) By (\ref{eq:strictinequalityi})  (with $t=0$), for every term $d_{j,i}Q_i^j$ of the $Q_i$-expansion of $f$ with $j>0$ we have
$\nu\left(d_{j,i}Q_i^j\right)>\nu\left(d_{j,i}Q_i^j\right)$. This proves (2). 
\end{proof}
Note that that the equivalent conditions of (2) of the Corollary are also equivalent to saying that the $Q_i$-expansion of $\ini_{i_0}f$ has non-zero coefficient $d_{0,i}$ satisfying (\ref{eq:equalityconstantterm}).
\begin{remark} The first equality in (\ref{eq:equalityconstantterm}) holds automatically when $i$ is not a limit ordinal.
\end{remark}
\medskip

Take a finite collection of polynomials $f_1\do f_s\in K[x]$ such that $\deg_xf_j<\bar\alpha_i$ for all $j\in\{1\do s\}$. Let
\begin{equation}
\prod\limits_{j=1}^sf_j=qQ_{i+t}+r\label{eq:Euclideanfg}
\end{equation}
be the Euclidean division of $\prod\limits_{j=1}^sf_j$ by $Q_{i+t}$. Assume that $i=i_0+$,
\begin{equation}
\nu_{i_0}(f_j)=\nu(f_j)\text{  for all }j\in\{1\do s\}\label{nui0f}
\end{equation}
\begin{equation}
\nu_{i_0}(q)=\nu(q)\label{nui0q}
\end{equation}
and
\begin{equation}
\nu_{i_0}(r)=\nu(r)\label{nui0r}
\end{equation}
(such an $i_0$ exists by Proposition \ref{Proposition9.26b} and (2)$_i$).
\begin{lemma}\label{qQ+r} We have $\nu\left(\prod\limits_{j=1}^sf_j\right)=\nu_{i_0}\left(\prod\limits_{j=1}^sf_j\right)=\nu(r)$ and
$$
\nu_{i+t}(qQ_{i+t})-\nu\left(\prod\limits_{j=1}^sf_j\right)\ge\beta_{i+t}-\alpha_i\beta_{i_0}>0.
$$
\end{lemma}
\begin{proof} By (1)$_{i_0}$ we have
\begin{equation}
\nu_{i_0}\left(\prod\limits_{j=1}^sf_j\right)=\sum\limits_{j=1}^s\nu_{i_0}\left(f_j\right)=
\sum\limits_{j=1}^s\nu\left(f_j\right)=\nu\left(\prod\limits_{j=1}^sf_j\right).\label{nui0additive}
\end{equation}
Now, $r$ in (\ref{eq:Euclideanfg}) is nothing but the constant term in the $(i+t)$-standard expansion of $\prod\limits_{j=1}^sf_j$, whereas $qQ_{i+t}$ is the sum of all the remaining terms in this $(i+t)$-standard expansion. Now the inequality of the Lemma is a special case of Lemma \ref{nuimonom>nui0f}. The fact that $\nu(r)=\nu\left(\prod\limits_{j=1}^sf_j\right)$ follows immediately from this by the ultrametric triangle law.
\end{proof}
We are now in the position to finish the proof of (1)$_i$. Let $g=\sum\limits_{j=0}^{u_i}c_{j,i}Q_i^j$ and
$f=\sum\limits_{j=0}^{s_i}d_{j,i}Q_i^j$ be the respective $i$-standard expansions of $g$ and $f$. By Lemma \ref{qQ+r}, for every $j\in\{0\do s_i\}$ and $j'\in\{0\do u_i\}$, we have $d_{j,i}c_{j',i}=q_{j,j'}Q_i+r_{j,j'}$ with
$$
\nu\left(d_{j,i}c_{j',i}\right)=\nu\left(r_{j,j'}\right)<\nu\left(q_{j,j'}Q_i\right).
$$
This implies that all the terms in the $i$-standard expansion of $fg$ have value at least $\nu_i(f)+\nu_i(g)$ and some terms (for example, the term involving $Q_i^{\delta_i(f)+\delta_i(g)}$) have exactly this value. This completes the proof.
\end{proof}
The valuation $\nu_i$ will be called the $i$-{\bf truncation} of $\nu$.
\begin{remark}\label{squeezebis} By (2)$_i$, if $\nu(f)=\nu_{i_0}(f)$ then $\nu_i(f)=\nu_{i_0}(f)$.
\end{remark}
\begin{proposition}\label{gradedalgebraoftruncation} (1) We have
\begin{equation}
\gr_{\nu_i}K[x]\cong G_{\nu_K}\left[\ini_{\nu}\mathbf{Q}_i ,\bar Q_i \right],\label{eq:grnui}
\end{equation}
where the symbol $\bar Q_i$ denotes an element, transcendental over $G_{\nu_K}\left[\ini_{\nu}\mathbf{Q}_i\right]$ with $\ord\ \bar Q_i=\beta_i$. The element $\bar Q_i$ is the image of $\ini_{\nu_i}Q_i$ under the isomorphism (\ref{eq:grnui}).

(2) With the identification (\ref{eq:grnui}) there is a natural degree-preserving homomorphism
$$
\phi:\gr_{\nu_i}K[x]\longrightarrow\gr_\nu K[x]
$$
of $G_{\nu_K}\left[\ini_{\nu}\mathbf{Q}_i\right]$-algebras defined by $\phi\left(\bar Q_i\right)=\ini_\nu Q_i$.
\smallskip

For parts (3) and (4) of the Proposition, assume that $i+1\in\Lambda$.
\smallskip

(3) $Ker\ \phi$ is generated by the polynomial obtained from the left hand side of (\ref{tag9.10}) by first replacing $i$ by $i+1$ and then replacing $\ini_\nu Q_i$ by $\bar Q_i$.

(4) $Im\ \phi=G_{\nu_K}\left[\ini_{\nu}\mathbf{Q}_{i+1}\right]=\{0\}\bigcup\left\{\left.\ini_\nu f\ \right|\ f\in K[x],\
\nu_i(f)=\nu(f)\right\}$.
\end{proposition}
\begin{proof} (2) The homomorphism $\phi$ is self-explanatory.

(1) Let $G_{<\bar\alpha_i}^{(i)}$ denote the $\gr_{\nu_K}K$-subalgebra of $\gr_{\nu_i}K[x]$, generated by all the $\ini_{\nu_i}f$ with $\deg_xf<\bar\alpha_i$. The homomorphism $\phi$ maps $G^{(i)}_{<\bar\alpha_i}$ isomorphically onto
$G_{<\bar\alpha_i}$. Given a non-zero element $f\in K[x]$ with $i$-standard expansion (\ref{tag9.23}), $\ini_{\nu_i}f$ can be written as
$$
\ini_{\nu_i}f=\sum\limits_{j\in S_i(f)}\ini_{\nu_i}d_{j,i}\ini_{\nu_i}Q_i^j
$$
with $\ini_{\nu_i}d_{j,i}\in G^{(i)}_{<\bar\alpha_i}$. If we had a non-trivial relation of the form
$\sum\limits_j\bar d_{j,i}\ini_{\nu_i}Q_i^j=0$ with
$$
\bar d_{j,i}\in G^{(i)}_{<\bar\alpha_i},
$$
we could take it to be homogeneous and lift it to an element $f=\sum\limits_jd_{j,i}Q_i^j\in K[x]$ such that $\nu_i(f)>\min\limits_j\left\{\nu\left(d_{j,i}Q_i^j\right)\right\}$ contradicting the definition of $\nu_i$. Thus $\ini_{\nu_i}Q_i$ is transcendental over $G_{<\bar\alpha_i}^{(i)}$.

Conversely, given an element $\bar f\in G_{<\bar\alpha_i}^{(i)}\left[\ini_{\nu_i}Q_i\right]$ we can consider its $i$-standard expansion and lift it to an element $f\in K[x]$ such that $\ini_{\nu_i}f=\bar f$. Thus $\ini_{\nu_i}K[x]\cong
G_{<\bar\alpha_i}^{(i)}\left[\ini_{\nu_i}Q_i\right]$ with $\ini_{\nu_i}Q_i$ transcendental over
$G_{<\bar\alpha_i}^{(i)}\left[\ini_{\nu_i}Q_i\right]$ and $G_{<\bar\alpha_i}^{(i)}\left[\ini_{\nu_i}Q_i\right]\cong
G_{\nu_K}\left[\ini_{\nu}\mathbf{Q}_i ,\bar Q_i \right]$ in view of Proposition \ref{GradedAlgebras} and the identification $G_{<\bar\alpha_i}^{(i)}\cong G_{<\bar\alpha_i}$ given by $\phi$.

(3) follows from the fact that $\ini_{\nu_i}Q_{i+1}$ is the minimal polynomial satisfied by $\ini_\nu Q_i$ over
$G_{\nu_K}\left[\ini_{\nu}\mathbf{Q}_i\right]$ (Definition \ref{DefKeyPolynom} (b) (5)).

(4) follows from the fact that $\phi(f)\ne0\iff\nu_i(f)=\nu(f)$ (Proposition \ref{Proposition4.8}).
\end{proof}

\section{The numerical character $\delta_i(h)$}\label{numchardeltai}

Let $\{Q_i\}_{i\in\Lambda}$ be a set of key polynomials for $\nu$. Let $i\in\Lambda$ and $h\in K[x]$. Let
\begin{equation}
h=\sum\limits_{j=0}^{s_i}d_{j,i}Q_i^j\label{tag9.23h}
\end{equation}
be an $i$-standard expansion of $h$, where each $d_{j,i}$ is a $Q_i $-free $i$-standard expansion. In this section we study the properties of the numerical character
\begin{equation}
\delta_i(h)=\max\ S_{i}(h)=\deg_{\bar Q_i}\ini_ih\label{tag11.1}
\end{equation}
(cf. (\ref{tag9.24})--(\ref{tag9.25})) that will play a crucial role in the rest of the paper. We prove that $\delta_i(h)$ does not increase with $i$. We also show that the equality $\delta_i(h)=\delta_{i+1}(h)$ imposes strong restrictions on $\ini_ih$.

\begin{definition}{\label{de9.20}} The $i$-th Newton polygon of $h$ with respect to $\nu$ is the convex hull $\Delta_i (h)$ of the set
$\bigcup\limits_{j=0}^{s_i }\left(\left(\nu(d_{j,i}),j\right)+\left(\tilde\g_{0+}\times\mathbb{Q}_+\right)\right)$ in
$\tilde\g_0\times\mathbb{Q}$.
\end{definition}

\begin{definition}{\label{de11.1}} The vertex $\left(\nu\left(d_{\delta_i(h),i}\right),\delta_i(h)\right)$ of the Newton polygon
$\Delta_i(h)$ is called the {\bf pivotal vertex} of $\Delta_i(h)$.
\end{definition}

Let
\begin{equation}
\nu_i^+(h)=\min\left\{\left.\nu\left(d_{j,i}Q_i^j\right)\ \right|\ \delta_i(h)<j\le
s_i\right\}\label{tag11.2}
\end{equation}
and
$$
S'_i(h)=\left\{j\in\{\delta_i(h)+1,\dots,s_i\}\ \left|\ \nu\left(d_{j,i}Q_i^j\right)=\nu_i^+(h)\right.\right\}.
$$
If the set on the right hand side of (\ref{tag11.2}) is empty, we adopt the convention that $\nu_i^+(h)=\infty$. We have
$\delta_i(h)>0$ whenever $\nu_i(h)<\nu(h)$ by Remark \ref{delta0nu=nui}.

Take an ordinal $i\in\Lambda$. Assume that there exists a polynomial $h$ such that
\begin{equation}
\nu_i(h)<\nu(h)\label{tag9.26}
\end{equation}
(this happens, for instance, whenever $i+1\in\Lambda$ by Corollary \ref{notcomplete}). Consider the $i$-th Newton polygon of $h$. Let $S_i(h)$ be as in (\ref{tag9.24}).

\begin{definition}{\label{de9.22}} For a polynomial $f\in K[x]$, we say that $\beta^{*} $ {\bf determines a side} of $\Delta_i(f)$ if the set $S_i(\beta^{*},f)$ has at least two elements.
\end{definition}
\medskip

\begin{proposition}{\label{Proposition9.21}} We have $\sum\limits_{j\in S_i(f)}\ini_{\nu}\left(d_{j,i}Q_i^j\right)=0$ in
$\frac{\mathbf{P}_{\nu_i(h)}}{\mathbf{P}_{\nu_i(h)+}}\subset G_{\nu}$.
\end{proposition}

\begin{proof} This follows immediately from (\ref{tag9.26}), the fact that
$$
\sum\limits_{j\in S_i(h)}d_{j,i}Q_i^j=h-\sum\limits_{j\in \{0,\dots,s_i\}\setminus S_i(h)}d_{j,i}Q_i^j
$$
and Proposition \ref{Proposition4.8}.
\end{proof}
In the following corollary, recall that a factorization of $\ini_ih$ into irreducible factors exists and is unique by Remark \ref{Gauss}.
\begin{corollary}{\label{Corollary9.24}} Viewing $\ini_ih$ as a polynomial in the variable $\bar Q_i$ with coefficients in
$G_{\nu_K}[\ini_{\nu}\mathbf{Q}_i]$, we have $\ini_ih\left(\ini_\nu Q_i\right)=0$. In particular, the element $\ini_{\nu}Q_i$ is integral over $G_{\nu_K}[\ini_{\nu}\mathbf{Q}_i]$. Its minimal polynomial over $G_{\nu_K}[\ini_{\nu}\mathbf{Q}_i]$ is one of the irreducible factors of $\ini_ih$.
\end{corollary}

\begin{corollary}{\label{Corollary9.23}} The element $\beta_i$ determines a side of $\Delta_i(h)$.
\end{corollary}

\begin{proof} Suppose not. Then the $i$-standard expansion of $h$ contains a unique term $d_{j,i}Q_i^j$ of minimal value, so
$\nu(h)=\nu(d_{j,i}Q_i^j)=\nu_i(h)$, contradicting (\ref{tag9.26}). Corollary \ref{Corollary9.23} is proved.
\end{proof}

Let
\begin{equation}
\ini_ih=\ini_\nu d_{\delta,i}\prod\limits_{j=1}^tg_{j,i}^{\gamma_{j,i}}\label{tag9.27}
\end{equation}
be the factorization of $\ini_ih$ into (monic) irreducible factors in $G_{\nu_K}\left[\ini_{\nu}\mathbf{Q}_i\right]\left[\bar
Q_i\right]$, where
\begin{equation}
\delta=\delta_i(h)\label{eq:delta=deltai}
\end{equation}
and $g_{1,i}$ is the minimal polynomial of $\ini_{\nu}Q_i$ over $G_{\nu_K}\left[\ini_{\nu}\mathbf{Q}_i\right]$.

The next Proposition shows that $\delta_i(h)$ is non-increasing with $i$ and that the equality $\delta_{i+1}(h)=\delta_i(h)$ imposes strong restrictions on $\ini_ih$.

\begin{proposition}{\label{Proposition11.2}}
Assume that $i+1\in\Lambda$ and
\begin{equation}
\nu(h)>\nu_i(h).\label{eq:nuh>nuih}
\end{equation}
\begin{enumerate}
\item[(1)] We have
\begin{equation}
\alpha_{i+1}\delta_{i+1}(h)\le\delta.\label{tag11.3}
\end{equation}
\item[(2)] If $\delta_{i+1}(h)=\delta$ then
\begin{equation}
Q_{i+1}=Q_i+z_i,\label{Qi+1=qi+zi}
\end{equation}
where $z_i$ is some $Q_i$-free $i$-standard expansion,
\begin{equation}
\ini_ih=\ini_{\nu}d_{\delta,i}\left(\bar Q_i+\ini_{\nu}z_i\right)^\delta\label{tag11.5}
\end{equation}
and $\ini_{i+1}h$ contains a monomial of the form $\ini_{\nu}d_{\delta,i}\bar Q_{i+1}^\delta$; in particular,
\begin{equation}
\ini_{\nu}d_{\delta,i}=\ini_{\nu}d_{\delta,i+1}.\label{tag11.6}
\end{equation}
\item[(3)] If $\delta_{i+1}(h)=\delta$, then for all $j>\delta$ we have
\begin{equation}
\nu\left(d_{j,i+1}Q_{i+1}^{j}\right)-\nu_{i+1}(h)\geq\nu_i^+(h)-\nu_i(h).\label{tag11.7}
\end{equation}
\end{enumerate}
\end{proposition}
\begin{remark} This Proposition is true as stated without the hypothesis (\ref{eq:nuh>nuih}), but the proof given below needs to be modified to include the case of equality in (\ref{eq:nuh>nuih}). In the sequel, the Proposition will only be used under the assumption (\ref{eq:nuh>nuih}).
\end{remark}
\begin{proof} We start with a lemma. Consider an $(i+1)$-standard expansion of $h$:
\begin{equation}
h=\sum\limits_{j=0}^{s_{i+1}}d_{j,i+1}Q_{i+1}^j,\label{tag11.9bis}
\end{equation}
where the $d_{j,i+1}$ are $Q_{i+1}$-free $(i+1)$-standard expansions.

In the notation immediately preceding Lemma \ref{el11.3} (with $f$ replaced by $h$ and $(i_0,i+t)$ by $(i,i+1)$), let $\theta_{i+1}(h)=\min\ S_{i,i+1}$.

\begin{definition}{\label{de11.5}} The vertex $(\nu(d_{\theta_{i+1}(h),i+1}),\theta_{i+1}(h))$ is called the {\bf characteristic vertex} of $\Delta_{i+1}(h)$.
\end{definition}

The polynomial $\ini_ih$ is divisible by the minimal polynomial $g_{1,i}$ of $\ini_{\nu}Q_i$ over
$G_{\nu_K}\left[\ini_{\nu}\mathbf{Q}_i\right]$; in particular, we have $\delta=\deg_{\bar Q_i}\ini_ih>0$.

For $j\in\{1\do t\}$, let $g_{j,i,0}$ denote the coefficient of $\ini_{\nu}Q_i^0$ in the polynomial $g_{j,i}$.

\begin{lemma}{\label{el11.6}} We have
\begin{equation}
\gamma_{1,i}=\theta_{i+1}(h)\label{tag11.13}
\end{equation}
(in particular, $d_{\gamma_{1,i},i+1}\ne0$) and
\begin{equation}
\ini_{\nu}d_{\theta_{i+1}(h),i+1}=\ini_{\nu}d_{\delta,i}\prod\limits_{j=2}^tg_{j,i,0}^{\gamma_{j,i}}.\label{tag11.14}
\end{equation}
\end{lemma}

\begin{proof} Write
$$
h=\sum\limits_{q\in S_{i,i+1}}d_{q,i+1}Q_{i+1}^q+\sum\limits_{q\in\{0,\dots,s\}\setminus S_{i,i+1}}d_{q,i+1}Q_{i+1}^q.
$$
By Lemma \ref{el11.3}, the terms of lowest $\nu_i$-value in the $Q_i$-expansion of $h$ are recovered from the terms of lowest
$\nu_i$-value of $\sum\limits_{q\in S_{i,i+1}}d_{q,i+1}Q_{i+1}^q$, in other words,
\begin{equation}
\ini_ih=\sum\limits_{q\in S_{i,i+1}}\ini_id_{q,i+1}\ini_iQ_{i+1}^q.\label{tag11.15}
\end{equation}
By definition of key polynomials (Definition \ref{DefKeyPolynom} (b)  (5)), we have
\begin{equation}
\ini_iQ_{i+1}=g_{1,i}.\label{tag11.16}
\end{equation}
By definition of $\theta_{i+1}(h)$, $\ini_iQ_{i+1}^{\theta_{i+1}(h)}$ is the highest power of $\ini_iQ_{i+1}$ dividing
$$
\sum\limits_{q\in S_{i,i+1}}\ini_id_{q,i+1}\ini_iQ_{i+1}^q.
$$
Since $g_{1,i}^{\gamma_{1,i}}$ is, by definition, the highest power of $g_{1,i}$ dividing $\ini_ih$, (\ref{tag11.15})--(\ref{tag11.16}) imply (\ref{tag11.13}).

Combining (\ref{tag9.27}) and (\ref{tag11.15}) and dividing both equations by $g_{1,i}^{\gamma_{1,i}}=\ini_iQ_{i+1}^{\theta_{i+1}(h)}$, we obtain
\begin{equation}
\sum\limits_{q\in S_{i,i+1}}\ini_id_{q,i+1}\ini_iQ_{i+1}^{q-\theta_{i+1}(h)}=
\ini_{\nu}d_{\delta, i}\prod\limits_{j=2}^tg_{j,i}^{\gamma_{j,i}}.\label{eq:sum=product}
\end{equation}
Equating the constant terms in (\ref{eq:sum=product}) yields (\ref{tag11.14}).
\end{proof}

We are now in the position to finish the proof of Proposition \ref{Proposition11.2}. Apply Lemma \ref{el11.4} to the monomials $d_{\theta_{i+1}(h),i+1}Q_{i+1}^{\theta_{i+1}(h)}$ and $d_{\delta_{i+1}(h),i+1}Q_{i+1}^{\delta_{i+1}(h)}$. We have
\begin{equation}
\nu\left(d_{\delta_{i+1}(h),i+1}Q_{i+1}^{\delta_{i+1}(h)}\right)\le
\nu\left(d_{\theta_{i+1}(h),i+1}Q_{i+1}^{\theta_{i+1}(h)}\right)\label{tag11.17}
\end{equation}
by definition of $\delta_{i+1}$ and
\begin{equation}
\nu_i\left(d_{\theta_{i+1}(h),i+1}Q_{i+1}^{\theta_{i+1}(h)}\right)=\nu_i(h)\le
\nu_i\left(d_{\delta_{i+1}(h),i+1}Q_{i+1}^{\delta_{i+1}(h)}\right)\label{tag11.18}
\end{equation}
by Lemma \ref{el11.3}, so the hypotheses of Lemma \ref{el11.4} are satisfied (note that (\ref{eq:nu0d=nud}) and (\ref{eq:nu0d'=nud'}) hold by definition of $\nu_i$ and $Q_i$-expansion of $h$). By Lemma \ref{el11.4}
\begin{equation}
\theta_{i+1}(h)\ge\delta_{i+1}(h).\label{tag11.19}
\end{equation}
Since
\begin{equation}
\alpha_{i+1}\theta_{i+1}(h)=\alpha_{i+1}\gamma_{1,i}\le\deg_{\bar Q_i}\ini_ih=\delta\label{tag11.20}
\end{equation}
by Lemma \ref{el11.6} and (\ref{tag9.27}), (1) of the Proposition follows.

(2) Assume that $\delta_{i+1}(h)=\delta$ (where the notation is as in (\ref{eq:delta=deltai})). Recall that $\delta>0$ by Remark \ref{delta0nu=nui}. Then the monomials $d_{\theta_{i+1}(h),i+1}Q_{i+1}^{\theta_{i+1}(h)}$ and
$d_{\delta_{i+1}(h),i+1}Q_{i+1}^{\delta_{i+1}(h)}$ coincide and
\begin{equation}
\alpha_{i+1}=1.\label{tag11.21}
\end{equation}
This proves (\ref{Qi+1=qi+zi}). Furthermore, we have equality in (\ref{tag11.20}). Then (\ref{tag9.27}) rewrites as
\begin{equation}
\ini_ih=\ini_{\nu}d_{\delta,i}g_{1,i}^\delta\prod\limits_{j=2}^tg_{j,i}^{\gamma_{j,i}}\label{tag11.12bis}
\end{equation}
Now, the left hand side of (\ref{tag11.12bis}) is a polynomial of degree $\delta$ in $\bar Q_i$, while the right hand side is divisible by the non-constant polynomial $g_{1,i}^\delta$. This implies that
\begin{equation}
t=1,\quad\prod\limits_{j=2}^tg_{j,i}^{\gamma_{j,i}}=1,\label{eq:t=1}
\end{equation}
and $\ini_ih=\ini_{\nu}d_{\delta,i} g_{1,i}^\delta$. (\ref{tag11.5}) now follows from (\ref{Qi+1=qi+zi}) and (\ref{tag11.16}).

The equality (\ref{tag11.6}) follows from (\ref{tag11.14}), (\ref{eq:t=1}) and the fact that $\theta_{i+1}(h)=\delta_{i+1}(h)=\delta$.
\medskip

(3) Assume that $\delta_{i+1}(h)=\delta$. Fix an integer $j>\delta$. The $Q_{i+1}$-expansion (\ref{tag11.9bis}) of $h$ is obtained from its $Q_i$-expansion (\ref{tag9.23h}) by making the substitution $Q_i=Q_{i+1}-z_i$ and performing repeated Euclidean divisions by $Q_{i+1}$ to turn the result into a $Q_{i+1}$-expansion. For $j'<j$, monomials of the form $d_{j',i}Q_i^{j'}$ have degrees strictly less than $j\bar\alpha_{i+1}$. Hence they contribute nothing to $d_{j,i+1}Q_{i+1}^j$; in other words, the coefficient $d_{j,i+1}$ is completely determined by $\sum\limits_{j'=j}^{s_i}d_{j,'i}Q_i^{j'}$.

Fix an integer $j'\in\{j,\dots,s_i\}$. Write
\begin{equation}
d_{j',i}Q_i^{j'}=d_{j',i}\left(Q_{i+1}-z_i\right)^{j'}=d_{j',i}\sum\limits_{k=0}^{j'}\binom{j'}k(-1)^kQ_{i+1}^kz_i^{j'-k},
\label{eq:binomii+1}
\end{equation}
where $z_i$ is a $Q_i$-free $i$-standard expansion. Again, the terms on the right hand side of (\ref{eq:binomii+1}) with $k>j$ contribute nothing to $d_{j,i+1}Q_{i+1}^j$. For $k\le j$, let $d_{j',k,i}$ denote the coefficient of $Q_{i+1}^j$ in the $(i+1)$-standard expansion of $d_{j',i}Q_{i+1}^kz_i^{j'-k}$. To prove (3), it is sufficient to prove that for all $j'\in\{j,\dots,s_i\}$ and all
$k\in\{0,\dots,j\}$  we have
\begin{equation}
\nu\left(d_{j',k,i}Q_{i+1}^j\right)-\nu_{i+1}(h)\ge\nu_i^+(h)-\nu_i(h).\label{eq:dj'ki}
\end{equation}
To prove (\ref{eq:dj'ki}), we start out by noting that (\ref{Qi+1=qi+zi}) is an $i$-standard expansion of $Q_{i+1}$. Hence
\begin{equation}
\nu_i(Q_{i+1})=\beta_i=\nu(z_i),\label{eq:nuiQi+1zi}
\end{equation}
where the last equality holds by Definition \ref{DefKeyPolynom} (4). By Lemma \ref{el11.3},
$\nu_i\left(d_{j',i}Q_{i+1}^kz_i^{j'-k}\right)$ equals the minimum of the $\nu_i$-values of the terms appearing in its $(i+1)$-standard expansion, so
$\nu_i\left(d_{j',i}Q_{i+1}^kz_i^{j'-k}\right)\le\nu_i\left(d_{j',k,i}Q_{i+1}^j\right)$. Combining this with (\ref{eq:nuiQi+1zi}) and with Proposition \ref{Proposition9.26b} applied to $d_{j',k,i}$ and the ordinal $i+1$, we obtain
\begin{equation}\label{eq:j'ki}
\nu\left(d_{j',k,i}\right)=\nu_i\left(d_{j',k,i}\right)\ge(j'-j)\beta_i+\nu\left(d_{j',i}\right).
\end{equation}
By definition of $\delta_{i+1}(h)$, since $\delta_{i+1}(h)=\delta$ and in view of (\ref{tag11.6}), we have
$$
\nu_{i+1}(h)=\nu(d_{\delta,i+1})+\delta\beta_{i+1}
$$
Combining this with (\ref{eq:j'ki}) and using (\ref{tag11.6}) again, we obtain
\begin{align}
\begin{split}
&\nu\left(d_{j',k,i}Q_{i+1}^j\right)-\nu_{i+1}(h)=\nu\left(d_{j',k,i}\right)-\nu\left(d_{\delta,i+1}\right)+(j-\delta)\beta_{i+1}\ge\\
&\ge(j'-j)\beta_i+\nu\left(d_{j',i}\right)-\nu\left(d_{\delta,i+1}\right)+(j-\delta)\beta_i=
\nu\left(d_{j',i}Q_i^{j'}\right)-\nu\left(d_{\delta,i}Q_i^\delta\right)\ge\nu_i^+(h)-\nu_i(h),
\end{split}
\end{align}
as desired. This completes the proof of the Proposition.
\end{proof}
\begin{corollary}\label{thetaiincreasing} For as long as the character $\delta_i(h)$ remains constant (that is, does not strictly decrease), the quantity $\nu_i^+(h)-\nu_i(h)$ is non-decreasing with $i$.
\end{corollary}
\begin{proof} Take the minimum over all $j$ on the left hand side of (\ref{tag11.7}).
\end{proof}
\begin{remark}{\label{Remark11.7}} One way of interpreting Lemma \ref{el11.4}, together with the inequalities
(\ref{tag11.17})--(\ref{tag11.19}) is that the characteristic vertex
$\left(\nu\left(d_{\theta_{i+1}(h),i+1}\right),\theta_{i+1}(h)\right)$ of $\Delta_{i+1}(h)$ always lies above its pivotal vertex
$\left(\nu\left(d_{\delta_{i+1}(h),i+1}\right),\delta_{i+1}(h)\right)$.
\end{remark}

\begin{corollary}[of Proposition \ref{Proposition11.2}]{\label{Corollary11.8}}
Assume that the set $\{i\in\Lambda\ |\ \alpha_{i}>1\}$ is infinite. Then $\{Q_i\}_{i\in\Lambda}$ is a complete set of key polynomials.
\end{corollary}

\begin{proof} Take an element $h\in K[x]$ and an index $i\in\Lambda$. If $\nu_i(h)=\nu(h)$, there is nothing to prove. Assume that
$\nu(h)>\nu_i(h)$. Then $\delta_i(h)>0$ by Remark \ref{delta0nu=nui}. Proposition \ref{Proposition11.2} (1) says that
\begin{equation}
\delta_{i+1}(h)<\delta_i(h)\label{tag11.26}
\end{equation}
whenever $\alpha_{i+1}>1$. Since the set $\{i\in\Lambda\ |\ \alpha_{i}>1\}$ is infinite and the strict inequality (\ref{tag11.26}) can occur for at most finitely many values of $i$, we have $\delta_i(h)=0$ for some $i\in\Lambda$. Then $\nu_i(h)=\nu(h)$ by Remark \ref{delta0nu=nui}.
\end{proof}

\section{Augmenting sets of key polynomials}\label{Augmenting}

Suppose we are given a set of key polynomials $\{Q_i\}_{i\in\Lambda}$ that is not complete for $\nu$. In this section, we will construct a set $\{Q_i\}_{i\in\Lambda_+}$ where $\Lambda_+$ is a well-ordered set of order type less than $\omega\times\omega$, and
$\Lambda\subsetneqq\Lambda_+$.

\subsection{Sets of key polynomials having a maximal element}
Suppose first that $\Lambda$ has a maximal element $\ell$.

Since $\{Q_i\}_{i\in\Lambda}$ is not complete, there exists $h\in K[x]$ such that
\begin{equation}
\nu_\ell(h)<\nu(h).\label{tag9.26bis}
\end{equation}
Let $\bar Q_\ell $ be a new variable. Take a polynomial $h$ satisfying (\ref{tag9.26bis}).

By Proposition \ref{Proposition9.21}, we have
\begin{equation}
\sum\limits_{j\in S_\ell (h)}\ini_{\nu}\left(d_{j,\ell}Q_\ell ^j\right)=0\label{eq:vanishingralg}
\end{equation}
in $\frac{\mathbf{P}_{\nu_\ell (h)}}{\mathbf{P}_{\nu_\ell (h)+}}\subset G_{\nu}$. By Corollary \ref{Corollary9.23} the element
$\beta_\ell $ determines a side of $\Delta_\ell (h)$. Let
\begin{equation}
\ini_{\nu_\ell}h=in_{\nu_\ell}d_{\delta,\ell}\prod\limits_{j=1}^tg_{j,\ell}^{\gamma_{j,\ell}},\label{tag9.27bis}
\end{equation}
where $\delta=\delta_\ell(h)$, be the factorization of $\ini_\ell h$ into monic, quasi-homogeneous irreducible factors in $G_{\nu_K}\left[\ini_{\nu}\mathbf{Q}_\ell \right]\left[\bar Q_\ell \right]$ (such a factorization exists and is unique by Remark \ref{Gauss}). By (\ref{eq:vanishingralg}), the element $\ini_{\nu}Q_\ell $ is integral over $G_{\nu_K}[\ini_{\nu}\mathbf{Q}_\ell]$. Its minimal polynomial over $G_{\nu_K}[\ini_{\nu}\mathbf{Q}_\ell ]$ is one of the irreducible factors $g_{j,\ell}$ of (\ref{tag9.27bis}).

Let $\alpha_{\ell+1}$ denote the degree of $\ini_{\nu}Q_\ell $ over $G_{\nu_K}[\ini_{\nu}\mathbf{Q}_\ell ]$. Renumbering the factors in (\ref{tag9.27bis}), if necessary, we may assume that $g_{1,\ell}$ is the minimal polynomial of $\ini_{\nu}Q_\ell $ over
$G_{\nu_K}\left[\ini_{\nu}\mathbf{Q}_\ell\right]$, so that
\begin{equation}
\alpha_{\ell +1}=\deg_{\bar Q_\ell }g_{1,\ell}.\label{tag9.28}
\end{equation}
Let
\begin{equation}
g_{1,\ell}=\bar Q_\ell^{\alpha_{\ell +1}}+\sum\limits_{j=0}^{\alpha_{\ell +1}-1}\left(\sum\limits_{{\bar\gamma}_\ell }\bar
c_{\ell +1,j,{\bar\gamma}_\ell }\ini_{\nu}\mathbf{Q}_\ell^{{\bar\gamma}_\ell }\right)\bar Q_\ell ^j\label{tag9.29}
\end{equation}
be an $\ell$-standard expansion of $g_{1,\ell}$. Let $X$ be a new variable, and consider a lifting of the right hand side of (\ref{tag9.29}) to $K[X]$, that is, a polynomial of the form\\
$X ^{\alpha_{\ell +1}}+\sum\limits_{j=0}^{\alpha_{\ell +1}-1}\left(\sum\limits_{{\bar\gamma}_\ell}
c_{\ell +1,j,{\bar\gamma}_\ell}\mathbf{Q}_\ell ^{{\bar\gamma}_\ell}\right) X ^j$, where
$\ini_\nu c_{\ell+1,j,{\bar\gamma}_\ell}=\bar c_{\ell +1,j,{\bar\gamma}_\ell}$ for all the choices of $\ell$, $j$, and
${\bar\gamma}_\ell$.

Define {\bf the $(\ell +1)$-st key polynomial} of $\nu$ to be
\begin{equation}
Q_{\ell +1}=Q_\ell ^{\alpha_{\ell +1}}+\sum\limits_{j=0}^{\alpha_{\ell +1}-1}
\left(\sum\limits_{{\bar\gamma}_\ell }
c_{\ell +1,j{\bar\gamma}_\ell }\mathbf{Q}_\ell ^{{\bar\gamma}_\ell }\right)Q_\ell ^j.\label{tag9.30}
\end{equation}

Let $\Lambda_+=\Lambda\cup \{\ell+1\}$. By definition $Q_{\ell+1}$ has the form (\ref{tag9.9}). The set $\{Q_i\}_{i\in\Lambda_+}$ satisfies Definition \ref{DefKeyPolynom}: it is a set of key polynomials.

In the special case when
\begin{equation}
\alpha_{\ell+1}=1\label{eq:alphaell+1=1}
\end{equation}
we will define several consecutive key polynomials at the same time.

Assume that $\alpha_{\ell+1}=1$. Let $T$ denote the set of all the polynomials of the form
$$
Q'=Q_\ell+w,
$$
where $\deg_xw<\bar\alpha_\ell=\deg_xQ_\ell$. To define $Q_{\ell +1}$, consider two cases:

\noindent{\bf Case 1.} The set $\nu(T)$ contains a maximal element. Let $Q'=Q_\ell+w$ be an element of $T$ for which the maximum is attained. Write $w=\sum\limits_{t=0}^sz_{\ell+t}+\tilde w$, where $s\in\mathbb N_0$, for each $t\in\{0\do s\}$ the expression $z_{\ell+t}$ is a homogeneous $Q_\ell$-free $\ell$-standard expansion, such that
\begin{equation}
\beta_\ell =\nu(z_\ell )<\nu(z_{\ell+1})<\dots<\nu(z_{\ell+s})<\nu\left(Q'\right)\label{tag9.36}
\end{equation}
and $\tilde w$ is a $Q_\ell$-free $\ell$-standard expansion all of whose terms have value greater than or equal to $\nu\left(Q'\right)$. Put $\Lambda_+=\Lambda\cup \{\ell+1,\dots,s+1\}$ and
$$
Q_i=Q_\ell +z_\ell +\dots+z_{i-1}\qquad\text{ for }\ell +1\le i\le\ell+s+1.
$$
\noindent{\bf Case 2.} The set $\nu(T)$ does not contain a maximal element. Let $\left(\tilde Q_{\ell+t}=Q_\ell+\tilde w_t\right)_{t\in\mathbb N_0}$ be a sequence of elements of $T$ such that the sequence
$\left(\nu\left(\tilde Q_{\ell+t}\right)\right)_{t\in\mathbb N_0}$ is strictly increasing and cofinal in $\nu(T)$. We will now define, recursively in $q$, an infinite sequence
\begin{equation}
(z_{\ell+q})_{q\in\mathbb N_0}\label{eq:sequencezl+t}
\end{equation}
of homogeneous $Q_\ell $-free $\ell$-standard expansions of strictly increasing values and a sequence $(w_t)_{t\ge-1}$ consisting of certain partial sums of the sequence (\ref{eq:sequencezl+t}). We adopt the convention that $z_{\ell-1}=w_{-1}=0$. Assume that for certain integers $t_0,q_0$ the finite sequences $(z_{\ell+q})_{q\le q_0}$ and $(w_t)_{t\le t_0}$ are already defined, that
$w_{t_0}=\sum\limits_{q=0}^{q_0}z_{\ell+q}$ and $\nu(z_{\ell+q})<\nu\left(\tilde Q_{\ell+t_0}\right)$ for all $q\le q_0$. Write
$\tilde w_{t_0+1}-w_{t_0}=\sum\limits_{q=q_0+1}^{q_1}z_{\ell+q}+w'_{t_0+1}$, where each $z_{\ell+q}$ is a homogeneous $Q_\ell$-free $\ell$-standard expansion, $\nu(z_{\ell+q_0+1})<\dots<\nu(z_{\ell+q_1})<\nu\left(\tilde Q_{\ell+t_0+1}\right)$ and $w'_{t_0+1}$ is a $Q_\ell$-free $\ell$-standard expansion, each of whose terms has value greater than or equal to $\nu\left(\tilde Q_{\ell+t_0+1}\right)$. Put 
$$
w_{t_0+1}:=w_{t_0}+\sum\limits_{q=q_0+1}^{q_1}z_{\ell+q}.
$$
This completes the recursive definition of the infinite sequences $(z_{\ell+q})_{q\in\mathbb N_0}$ and $(w_t)_{t\ge-1}$. Define
$\alpha_{\ell+t}=1$,
$$
Q_{\ell +t}=Q_\ell +z_\ell +z_{\ell +1}+\dots+z_{\ell +t-1}\qquad\text{ for }t\in\mathbb{N},
$$
and put $\Lambda_+=\Lambda\cup \{\ell+t\}_{t\in\mathbb{N}}$.
\smallskip

In both Cases 1 and 2, for each $t$ under consideration we have
$$
\nu(Q_{\ell+t})=\nu(z_{\ell+t})<\nu(Q_{\ell+t+1})=\nu(z_{\ell+t+1}).
$$
Thus, by the choice of the sequences $\left(\nu\left(\tilde Q_{\ell+t}\right)\right)_{t\in\mathbb N_0}$, $(z_{\ell+q})_{q\in\mathbb N_0}$ and $(w_t)_{t\ge-1}$, conditions (a)--(d) of Definition \ref{DefKeyPolynom} are satisfied for all the ordinals $i$ of the form $i=\ell+t$, $t\in\mathbb N_0$.

Therefore in all cases the set $\{Q_{i}\}_{i\in\Lambda_+}$ is a set of key polynomials for $\nu$.

\begin{remark}{\label{Remark9.25}} We claim that $Q_{\ell+1}$ is an irreducible polynomial in $x$. Indeed, consider a factorization
$Q_{\ell+1}=fg$ in $K[x]$. Passing to the natural images of $Q_{\ell+1}$, $f$ and $g$ in
$$
G_{\nu_K}\left[\ini_{\nu}\mathbf{Q}_\ell \right]\left[\bar Q_\ell \right]\cong G_{\nu_\ell},
$$
we obtain $g_{1,\ell}=\ini_\ell f\ini_\ell g$. Since $g_{1,\ell}$ is an irreducible polynomial in $\bar Q_\ell $ by definition, we have, up to interchanging $f$ and $g$,
\begin{equation}
\deg_{\bar Q_\ell}\ini_\ell f=\alpha_{\ell+1}.\label{eq:degreemaximal}
\end{equation}
Then
\begin{equation}
\deg_xf\ge\alpha_{\ell+1}\deg_xQ_\ell=\deg_xQ_{\ell+1},\label{eq:degreeequal}
\end{equation}
where the equality holds since $g_{1\ell}$ has the form (\ref{tag9.29}). We must have equality in (\ref{eq:degreeequal}) and
$\deg_xg=0$. Thus $g\in K$; this completes the proof of the irreducibility of $Q_{\ell+1}$ in $K[x]$.

Similarly, in the case when (\ref{eq:alphaell+1=1}) holds, $Q_{\ell+t}$ is irreducible for all $t\in\mathbb N$ for which $Q_{\ell+t}$ is defined. This follows from the above argument by induction on $t$.
\end{remark}

\subsection{Augmenting sets of key polynomials without maximal elements}\label{setswomaxelts}

Suppose that the set $\Lambda$ does not have a maximal element. Since the set $\{Q_i\}_{i\in\Lambda}$ of key polynomials is not complete, by Corollary \ref{Corollary11.8} the degrees of the polynomials $Q_i$, $i\in\Lambda$, are bounded in $\mathbb{N}$.
Hence there exists $\ell$ such that for each $i\in\Lambda$, $i>\ell$, we have $i=\ell+t$ and $\alpha_{\ell+t}=1$ for
$t\in\mathbb{N}$. Since this set of key polynomials is not complete for $\nu$, there exists a monic polynomial $h$ such that
\begin{equation}
\nu_{\ell+t}(h)<\nu(h)\label{tag9.43}
\end{equation}
for all $t\in\mathbb{N}$. In this case, define $Q_{\ell+\omega}$ to be a smallest degree monic polynomial $h$ satisfying (\ref{tag9.43}).

\begin{remark}\label{degreejumps} The inequality (\ref{tag9.43}) implies that
\begin{equation}
\deg_xQ_{\ell+\omega}\ge\deg_xQ_\ell=\deg_xQ_{\ell+t}\quad\text{for all }t\in\mathbb N.\label{eq:degreenondecreasing}
\end{equation}
by Proposition \ref{Proposition9.26b}. If the inequality in (\ref{eq:degreenondecreasing}) were an equality, we would have a contradiction with the condition (d) of Definition \ref{DefKeyPolynom}. Thus,
$\deg_xQ_{\ell+\omega}>\deg_xQ_\ell=\deg_xQ_{\ell+t}$, $t\in\mathbb N$.

For each $t\in\mathbb N$, the expression $Q_{\ell+t}=Q_{\ell+t+1}-z_{\ell+t}$ is an $(\ell+t+1)$-expansion of $Q_{\ell+t}$. Thus, in view of Proposition \ref{QiinRnux}, $h=Q_{\ell+t}$ satisfies the hypotheses of Proposition \ref{Proposition9.26a} (1) with $i=\ell+t+1$. We obtain $\nu_{\ell+t+1}\left(Q_{\ell+t}\right)=\nu\left(Q_{\ell+t}\right)$. To summarize, we have
\begin{equation}
\nu(Q_{\ell+t+1})=\nu_{\ell+t+1}(Q_{\ell+t+1})>\nu_{\ell+t+1}(Q_{\ell+t})=\nu\left(Q_{\ell+t}\right).\label{eq:valuestrictlygreater}
\end{equation}
\end{remark}
\medskip

Let
$$
\bar\beta=\sup\left\{\left.\nu\left(Q_{\ell+t}\right)\ \right|\ t\in\mathbb N\right\}.
$$
(here we allow the possibility $\bar\beta=\infty$, which means that the set $\left\{\left.\nu\left(Q'\right)\ \right|\ Q'\in T\right\}$ is unbounded in $\tilde\g_0$).

We put $\Lambda_+=\Lambda\cup \{Q_{\ell+\omega}\}$. By Remark \ref{eq:coefficientsunique} and Proposition \ref{Qlislstandard}, for every $t\in\mathbb N_0$ the polynomial $Q_{\ell+\omega}$ admits an $(\ell+t)$-standard expansion (\ref{tag9.11}):
\begin{equation}
Q_{\ell+\omega}=\sum\limits_{j=0}^{\alpha_{\ell+\omega}}c_{j,\ell+t}Q_{\ell+t}^{j}.\label{tag9.11bis}
\end{equation}
By (\ref{tag9.43}), the inequality (\ref{eq:almosthomogeneous}) holds for $i=\ell+\omega$ and $i_0=\ell+t$.

To complete the proof that $\{Q_{i}\}_{i\in\Lambda_+}$ is a set of key polynomials for $\nu$, it remains to prove (\ref{eq:limitmonic}) and (\ref{eq:nui0ofQi}), assuming (\ref{eq:sup<infty}) (with $i$ replaced by $\ell+\omega$ and $i_0$ replaced by $\ell+t$, $t\in\mathbb N$).

Therefore, assume that the sequence $(\beta_{\ell+t})_{t\in\mathbb{N}}$ has an upper bound (but no maximum) in $\tilde\g_0$.

To simplify the notation, for the purposes of the next Proposition we will denote $Q_{\ell+\omega}$ by $f$. We have
\begin{equation}
\deg_{Q_\ell}f:=\left[\frac{\deg_xf}{\deg_xQ_\ell}\right]\ge\delta_\ell(f).\label{eq:deggedeltal}
\end{equation}
By Proposition \ref{Proposition11.2} there exists $t_0\in\mathbb{N}$ such that
\begin{equation}
\delta_{\ell+t}(f)=\delta_{\ell+t_0}(f)\ \text{ for all }\ t\ge t_0\label{eq:alpha=1deltastabiilzes}
\end{equation}
(in fact, (\ref{eq:alpha=1deltastabiilzes}) holds already for $t_0=1$, but we have not proved that yet). Let $\delta$ denote the stable value of $\delta_{\ell+t}(f)$ for large $t$. The inequality (\ref{eq:deggedeltal}) implies that
\begin{equation}
\deg_{Q_\ell}f\ge\delta.\label{eq:deggedelta}
\end{equation}
The next Proposition says that equality holds in (\ref{eq:deggedelta}).

In what follows, the index $i$ will run over the set $\{\ell+t\}_{t\in\mathbb{N}_0}$.
\begin{proposition}\label{weakly_affine_lim} For each
\[
i\in\{\ell+t\}_{t\in\mathbb{N}_0}
\]
we have
\begin{equation}
\deg_xf=\delta\deg_xQ_i\label{eq:deg=delta}
\end{equation}
and
\begin{equation}\label{lim_kp_in2}
\bar\beta\le\frac1{\delta}\nu\left(f\right).
\end{equation}
\end{proposition}

\begin{proof} By (\ref{tag9.43}) for all $t\in\mathbb{N}_0$ we have $\nu(f)>\nu_{\ell+t}(f)=\delta\nu(Q_{\ell+t})$. This proves (\ref{lim_kp_in2}).  The main point is to prove (\ref{eq:deg=delta}). Our strategy for doing this consists in gradually modifying $f$ while preserving the numerical character $\delta_i(f)$ and the condition
\begin{equation}
\nu_i(f)<\nu(f)\quad\text{for all }i\text{ of the form }i=\ell+t,\ t\in\mathbb N\label{eq:nuif<nu(f)}
\end{equation}
though possibly changing $\deg_xf$ in the process, until we arrive at a polynomial of degree $\delta\bar\alpha_i$.

Since $\alpha_i=1$ for all $i$, $\deg_xQ_i$ is independent of $i$, so all the $i$-standard expansions of $f$ have the same degree $\alpha_{\ell+\omega}$ in $Q_i$. For each $i$, let
\begin{equation}
f=\sum\limits_{j=0}^{\alpha_{\ell+\omega}}a_{j,i}Q_i^j\label{eq:istandardexpoff}
\end{equation}
be an $i$-standard expansion of $f$.
\begin{lemma} There exists a polynomial $a^*\in K[x]$ of degree strictly less than $\deg_x Q_\ell$ such that
\begin{equation}\label{dom_coef1}
\ini_\nu a^*\ini_\nu a_{\delta,\ell}=1
\end{equation}
in $G_\nu$.
\end{lemma}
\begin{proof} Let $\ell_-$ denote the smallest ordinal such that $\bar\alpha_{\ell_-}=\bar\alpha_\ell$. By Proposition \ref{Proposition9.26b} (applied to $h=a_{\delta,\ell}$), there exists an ordinal $\ell_0<\ell_-$ such that $\ell_-=\ell_0+$ and
\begin{equation}
\nu_{\ell_0}(a_{\delta,\ell})=\nu(a_{\delta,\ell}).\label{eq:nul0adelta=nuadelta}
\end{equation}
By (\ref{eq:nul0adelta=nuadelta}) and Proposition \ref{gradedalgebraoftruncation} (4) we have $\ini_\nu a_{\delta,\ell}\in
G_{\nu_K}\left[\ini_\nu\mathbf Q_{\ell_0+1}\right]=G_{\nu_K}\left[\ini_\nu\mathbf Q_{\ell_0}\right]\left[\ini_\nu Q_{\ell_0}\right]$. By Definition \ref{DefKeyPolynom}, $\ini_\nu Q_{\ell_0}$ is integral over $G_{\nu_K}\left[\ini_\nu\mathbf Q_{\ell_0+1}\right]$ of degree
$\alpha_{\ell_0+1}$. By Lemma \ref{el4.4} and transfinite induction on $\ell_0$, the graded algebras
\begin{equation}
G_{\nu_K}\left[\ini_\nu\mathbf Q_{\ell_0}\right]\subset G_{\nu_K}\left[\ini_\nu\mathbf Q_{\ell_0}\right]\left[\ini_\nu Q_{\ell_0}\right]
\label{eq:inclusionofsaturatedalgebras}
\end{equation}
are saturated. Hence the homogeneous element
$\ini_\nu a_{\delta,\ell}\in G_{\nu_K}\left[\ini_\nu\mathbf Q_{\ell_0}\right]\left[\ini_\nu Q_{\ell_0}\right]$ has a multiplicative inverse
$\bar a^*\in G_{\nu_K}\left[\ini_\nu\mathbf Q_{\ell_0}\right]\left[\ini_\nu Q_{\ell_0}\right]$. Again by Lemma \ref{el4.4}, $\bar a^*$ admits an $\ell_0$-standard expansion whose degree in $\ini_\nu Q_{\ell_0}$ is strictly less than $\alpha_{\ell_0+1}$. A lifting of this
$\ell_0$-standard expansion to $K[x]$ is the desired element $a^*$.
\end{proof}
Note that by Proposition \ref{Proposition11.2} (2) and transfinite induction on $i$, for all $i\ge \ell$ 
\begin{equation}\label{dom_coef2}
\ini_{\nu}a_{\delta,\ell}=\ini_{\nu}a_{\delta,i}
\end{equation}
and 
\[
\ini_if=\ini_{\nu}a_{\delta,i}(\bar Q_i+\ini_{\nu}z_i)^\delta.
\]
Hence, in view of (\ref{dom_coef1}) and (\ref{dom_coef2}), we have
\begin{equation}
\ini_i(a^*f)=(\bar Q_i+\ini_{\nu}z_i)^\delta.\label{eq:iniif}
\end{equation}
Now, $\bar Q_i+\ini_{\nu}z_i$ is the minimal algebraic relation satisfied by $\ini_\nu Q_i$ over $G_{\nu_K}\left[\mathbf Q_i\right]$, so
this polynomial maps to 0  under the homomorphism $\phi$ of  Proposition \ref{gradedalgebraoftruncation} (2). Therefore
$$
\left(\bar Q_i+\ini_{\nu}z_i\right)^\delta\in Ker\ \phi.
$$
Together with (\ref{eq:iniif}) and Proposition \ref{gradedalgebraoftruncation} (3)--(4) this implies that
\begin{equation}
\nu(a^*f)>\nu_i(a^*f)\quad\text{ for all }i.
\end{equation}
We claim that
\begin{equation}
\delta_i(a^*f)=\delta\quad\text{ for all }i.\label{eq:doesnotaffectdelta}
\end{equation}
Indeed, applying Lemma \ref{qQ+r} with $s=2$ to the pairs of polynomials
\[
(f_1,f_2)=(a_{\delta,i},a^*)
\]
and
\[
(f_1,f_2)=(a_{\delta-1,i},a^*),
\]
we see that after multiplying $f$ by $a^*$ and applying Euclidean division of $a^*f$ by $Q_{i}$ to obtain the $i$-standard expansion of $a^*f$, only the remainders in the Euclidean division contribute to $\ini_{i}(a^*f)$. This means that the powers of $\bar Q_i$ that appear in $\ini_i(a^*f)$ with non-zero coefficients are exactly the same as those appearing in $\ini_if$ and (\ref{eq:doesnotaffectdelta}) follows.

Thus, replacing $f$ by $a^*f$, we may assume that $\ini_{\nu}a_{\delta,i}=1$ for all $i$.

Let
\begin{equation}
\theta(i)=\frac12\min\left\{\nu_i^+(f)-\nu_i(f),\beta_i-\beta_\ell\right\};\label{eq:thetai}
\end{equation}
we have $\theta(i)>0$. By Corollary \ref{thetaiincreasing} the quantity $\nu_i^+(f)-\nu_i(f)$ is non-decreasing with $i$ and hence so is $\theta(i)$. Taking into account the fact that $\bar\beta=\lim\limits_{i\to\infty}\beta_i$, we have, for $i$ sufficiently large,
\begin{equation}\label{dom_coef3}
\nu(a_{\delta,i})+\delta\bar\beta-\nu_i(f)=\delta(\bar\beta-\beta_i)<\theta(i).
\end{equation}
By choosing $\ell_1>\ell$ sufficiently large, we may assume that (\ref{dom_coef3}) holds for $i\ge\ell_1$.

For each $i\ge\ell_1$, write $a_{\delta,i}=1+a_i^\dag$ with
\begin{equation}
\nu\left(a_i^\dag\right)>0.\label{eq:nuadagbdedawayfrom0bis}
\end{equation}
Write
\[
f=\bar f_i+\tilde f_i,
\]
where
\[
\bar f_i=Q_i^\delta+\sum\limits_{j=0}^{\delta-1}a_{j,i}Q_i^j
\]
and
\[
\tilde f_i=a_i^\dag Q_i^\delta+\sum\limits_{j=\delta+1}^{n_i}a_{j,i}Q_i^j.
\]
We want to compare the $\ell_1$-standard expansion of $f$ with its $i$-standard expansion for $i>\ell_1$ and, in particular, to study the dependence on $i$ of $a_{\delta,i}$ and $a_i^\dag:=a_{\delta,i}-1$.

First of all, for all $j$ with $\delta<j\le n_{\ell_1}$ and all $i\ge\ell_1$ we have
\begin{equation}
\nu_i\left(a_{j,\ell_1}Q_{\ell_1}^j\right)-\delta\beta_i\ge\nu_{\ell_1}\left(a_{j,\ell_1}Q_{\ell_1}^j\right)-
\delta\beta_{\ell_1}-\delta(\beta_i-\beta_{\ell_1})\ge\nu_{\ell_1}^+(f)-\nu_{\ell_1}(f)-\delta(\bar\beta-\beta_{\ell_1})>
\theta(\ell_1),\label{eq:contributionofj>deltabounded}
\end{equation}
where:

the first inequality holds by Proposition \ref{nuiisavaluation} (2)

the second inequality holds by the definitions of $\nu_{\ell_1}^+(f)$ and of $\bar\beta$ and

the third inequality holds by (\ref{eq:thetai})--(\ref{dom_coef3}).
\medskip

For $i>\ell_1$, write $Q_i=Q_{\ell_1}+w_i$, where $w_i$ is a $Q_i$-free $i$-standard expansion (in particular,
$\deg_xw_i<\bar\alpha_i=\bar\alpha_{\ell_1}$. To compute the coefficient $a_{i,\delta}$ in the $i$-standard expansion of $f=\bar
f_{\ell_1}+\tilde f_{\ell_1}$, we must consider the expressions
\begin{equation}
Q_{\ell_1}^\delta+\sum\limits_{j=0}^{\delta-1}a_{j,\ell_1}Q_{\ell_1}^j=
\left(Q_i-w_i\right)^\delta+\sum\limits_{j=0}^{\delta-1}a_{j,\ell_1}\left(Q_i-w_i\right)^j\label{eq:barfw}
\end{equation}
and
$$
a_{\ell_1}^\dag Q_{\ell_1}^\delta+\sum\limits_{j=\delta+1}^{n_{\ell_1}}a_{j,\ell_1}Q_{\ell_1}^j=
a_{\ell_1}^\dag\left(Q_i-w_i\right)^\delta+\sum\limits_{j=\delta+1}^{n_{\ell_1}}a_{j,\ell_1}\left(Q_i-w_i\right)^j,
$$
open the parentheses and perform the appropriate Euclidean divisions by $Q_i$. We have
$$
\deg_x\sum\limits_{j=0}^{\delta-1}a_{j,\ell_1}\left(Q_i-w_i\right)^j<\delta\bar\alpha_i.
$$
As well, every term in the Newton binomial expansion of $\left(Q_i-w_i\right)^\delta$ except for $Q_i^\delta$ has degree strictly smaller than $\delta\bar\alpha_i$. Therefore no terms in (\ref{eq:barfw}) contribute anything to $a_i^\dag$. By (\ref{eq:contributionofj>deltabounded}), all the contributions to $a_i^\dag$ by terms of the form $a_{j,\ell_1}Q_{\ell_1}$, $j>\delta$, have $\nu_i$-value strictly greater than  $\theta(\ell_1)$.

It remains to study the contribution to $a_i^\dag$ from
$a_{\ell_1}^\dag\left(Q_i-w_i\right)^\delta=\sum\limits_{j=0}^\delta\binom\delta ja_{\ell_1}^\dag w_i^jQ_i^{\delta-j}$. For $j=0$ this contribution equals $a_{\ell_1}^\dag$. For each $j\in\{1\do\delta\}$ this contribution is nothing but the coefficient $d_j$ of $Q_i^j$ in the $i$-standard expansion of $a_{\ell_1}^\dag w_i^j$. We have
$$
\nu(d_j)\ge\nu_i(d_j)\ge\beta_i-\beta_\ell-j(\beta_i-\beta_{\ell_1})>\theta(\ell_1),
$$
where the second inequality follows from Lemma \ref{qQ+r}, applied with $s=j+1$, $f_1=a_{\ell_1}^\dag$,
$$
f_2=\dots=f_s=w_i,
$$
$i+t$ replaced by $i$ and $i_0$ replaced by $\ell$. Putting together all of the above information, we obtain
\begin{equation}
\nu_i\left(a_i^\dag\right)\ge\min\left\{\theta(\ell_1),\nu\left(a_{\ell_1}^\dag\right)\right\}.\label{eq:nuadagbdedawayfrom0}
\end{equation}
Take an ordinal $\ell_2$ of the form $\ell_2=\ell_1+t_1$, $t_1\in\mathbb N$, such that
$\bar\beta-\beta_{\ell_2}<\nu\left(a_{\ell_1}^\dag\right)$.

Since for all $i\ge\ell_2$ we have $\nu\left(a_i^\dag\right)>0$ and in view of the definition of $\delta$, every term in the $Q_i$-expansion of  $\tilde f_i$ has value strictly greater than $\delta\beta_i$. Hence $\nu\left(\tilde f_{\ell_2}\right)\ge\nu_i\left(\tilde
f_{\ell_2}\right)>\nu_i\left(\bar f_{\ell_2}\right)$ which implies that $\nu_i\left(\bar
f_{\ell_2}\right)=\nu_i\left(f\right)<\min\left\{\nu(f),\nu\left(\tilde f_{\ell_2}\right)\right\}\le\nu\left(f-\tilde
f_{\ell_2}\right)=\nu\left(\bar f_{\ell_2}\right)$. Since $\deg_x\bar f_{\ell_2}=\delta\deg_xQ_\ell\le\deg_xf$ and $f$ was chosen of minimal degree subject to inequality (\ref{eq:nuif<nu(f)}), we must have $\deg_xf=\deg_x\bar f_{\ell_2}=\delta\deg_xQ_\ell$.
\end{proof}
\begin{corollary} We have equalities (\ref{eq:limitmonic}) and (\ref{eq:nui0ofQi}) (with $i$ replaced by $\ell+\omega$ and $i_0$ replaced by $\ell+t$, $t\in\mathbb N$). In particular, the set $\{Q_{i}\}_{i\in\Lambda_+}$ is a set of key polynomials for $\nu$.
\end{corollary}
\begin{proof} This follows immediately from (\ref{eq:deg=delta}) and Proposition \ref{Proposition11.2} (2) (specifically, the equality (\ref{tag11.5})).
\end{proof}

Below, in Proposition \ref{Proposition12.5}, we will show that $\delta(f)$ is of the form $\delta(f)=p^{e_0}$ for some
$e_0\in\mathbb{N}_0$. Together with Remark \ref{degreejumps} this will prove that, under the assumptions of this section, we have
$\ch\ k_\nu>0$ and $e_0>0$.

\begin{proposition}{\label{Proposition9.30}} If $\nu(Q_{\ell+t})\in\tilde\g_0$ for all $t\in\mathbb{N}_0$ and the sequence
$\left(\beta_{\ell+t}\right)_{t\in\mathbb{N}_0}$ is cofinal in $\tilde{\g}_0$, then the set $\mathbf{Q}_{\ell+\omega}$ of key polynomials defined above is $\tilde{\g}_0$-complete. In other words, for every element $\beta\in\tilde\g_{0+}$ every polynomial $h\in K[x]$ with $\nu(h)=\beta$ belongs to the additive subgroup of $\mathbf{P}_\beta\cap K[x]$ generated by all the standard monomials in $\mathbf{Q}_{\ell+\omega}$, multiplied by elements of $K$, of value $\beta$ or higher.
\end{proposition}
\begin{proof} Take an element $h\in K[x]$. Without loss of generality, we may assume that, writing $h=\sum\limits_{j=0}^sd_jx^j$, we have
\begin{equation}
\nu_K(d_j)\ge0\quad\text{for all }j\label{tag9.46}
\end{equation}
(otherwise, multiply $h$ by a suitable element of $K$). In other words, we may assume that $h\in R_{\nu_K}[x]$. Since the sequence
$\{\beta_{\ell+t}\}_{t\in\mathbb{N}_0}$ is cofinal in $\tilde{\g}_0$, there exists $i$ of the form $i=\ell+t$, $t\in\mathbb{N}_0$, such that
\begin{equation}
\beta_i>\nu(h).\label{tag9.47}
\end{equation}
Then $h$ satisfies the hypotheses of Proposition \ref{Proposition9.26a}. Now, Proposition \ref{Proposition9.26a} says that
\[
\nu_i(h)=\nu(h).
\]
This means, by definition, that $h$ can be written as a sum of standard monomials in $\mathbf{Q}_{i+1}$ of value at least $\nu(h)$, hence it
belongs to the additive abelian group generated by all such monomials. This completes the proof.
\end{proof}

\section{Constructing a complete set of key polynomials}\label{Section3}

In this section we will recursively construct a complete set of key polynomials.
\medskip

First, consider the set $\{Q_i\}_{i\in\Lambda}$, with $Q_0=x$ and $\Lambda=\{0\}$. It is a set of key polynomials.
\medskip

Next, assume that a set  $\{Q_i\}_{i\in\Lambda}$ of key polynomials is constructed and that $\Lambda<\omega\times\omega$. If $\Lambda$ does not have a maximal element then there exists and ordinal $\ell$ such that $\left\{\ell+t\ \left|\ t\in\mathbb
N_0\right.\right\}$ is a subset of $\Lambda$, cofinal in $\Lambda$.

If the set $\{Q_i\}_{i\in\Lambda}$ is complete for $\nu$, put $\bar\Lambda:=\Lambda$ and stop. Otherwise, consider the set
$\{Q_i\}_{i\in\Lambda_{+}}$ of key polynomials constructed in the previous section. If the set  $\{Q_i\}_{i\in\Lambda_{+}}$ of key polynomials is complete, put $\bar\Lambda:=\Lambda_+$ and stop here. Otherwise, replace $\{Q_i\}_{i\in\Lambda}$ by
$\{Q_i\}_{i\in\Lambda_{+}}$ and repeat the procedure.
\medskip

\begin{remark}\label{omegatimesomega} If $\mathbf Q=\{Q_i\}_{i\in\Lambda}$ is a set of key polynomials, $i_0,i\in\Lambda$, $i$ is a limit ordinal and $i_0+=i$ then
\begin{equation}
\bar\alpha_i>\bar\alpha_{i_0}\label{eq:degreejumpsinlimits}
\end{equation}
by Remark \ref{degreejumps}. Therefore, if as a result of the above recursive procedure we arrive at a set of key polynomials with
$\Lambda=\omega\times\omega$, since $\Lambda$ contains infinitely many limit ordinals, such a set of key polynomials must have unbounded degrees and is therefore complete by Proposition \ref{Proposition9.17} (1) and Corollary \ref{Corollary11.8}. This proves that the aobve recursive procedure produces a complete set of key polynomials after at most $\omega\times\omega$ steps. In other words, $\bar\Lambda\le\omega\times\omega$.
\end{remark}
\medskip

We have proved the following:

\begin{theorem}{\label{th13.10}} The well ordered set $\mathbf{Q}:=\{Q_i\}_{i\in\bar\Lambda}$ constructed above is a complete set of key polynomials. In other words, for every element $\beta\in\Gamma$ the $R_\nu$-module $\mathbf{P}_\beta\cap K[x]$ is generated as an additive group by all the monomials in the $Q_i$ of value $\beta$ or higher, multiplied by elements of $K$. In particular, we have
$$
\bigoplus\limits_{\beta\in\Gamma}\frac{\mathbf{P}_\beta}{\mathbf{P}_{\beta+}}=G_{\nu_K}[\ini_{\nu}\mathbf{Q}]^*.
$$
\end{theorem}

In \S\ref{Infsequenceskeypolynom} we will fix an ordinal $\ell$ such that $\ell+\omega\in\bar\Lambda$ and will study further properties of $Q_{\ell+\omega}$. Among other things, we will show (Propositions \ref{Proposition12.5} and \ref{Proposition12.8} and Remark \ref{Remark12.6}) that:

(a) if $\ch\ k_\nu=0$ then our construction gives a complete set of key polynomials that is of order type at most
$\omega+1$.

(b) if, in addition, $\rk\ \nu=1$ then the construction produces a complete set of key polynomials that is of order type at most
$\omega$.

In the next section we study the effect of differential operators on the polynomials $Q_i$ in order to give a more precise description of the form of limit key polynomials.

\section{Key polynomials and differential operators}\label{Keypolanddiffop}

This section is devoted to proving some basic results about the effect of differential operators on key polynomials. Here and below, for a non-negative integer $b$, $\partial_b$ will denote the $b$-th Hasse (or formal) derivative, defined in the Introduction. Given an
$\ell$-standard expansion $h$, we are interested in proving lower bounds on (and, in some cases, exact formulae for) the
quantities $\nu(\partial_bh)$ and $\nu_\ell (\partial_bh)$ and in computing the elements $\ini_{\nu}\partial_bh$ and $\ini_\ell \partial_bh$. In particular, we will give sufficient conditions for the element $\partial_bh$ to be non-zero.

Assume given a complete set $\mathbf{Q}:=\{Q_i\}_{i\in\bar\Lambda}$ of key polynomials. Take an ordinal $i\in\bar\Lambda$. Let $b_i$ denote the smallest positive integer which maximizes the quantity
\begin{equation}
\frac{\beta_i-\nu(\partial_{b_i}Q_i)}{b_i}
\end{equation}
(later in this section, we will show that $b_i$ is necessarily of the form $p^{e_i}$ for some $e_i\in\mathbb{N}_0$ and, in particular, that $b_i=1$ if $\ch\ k_\nu=0$).

Let $h$ be an element of $K[x]$. We use the following convention for binomial coefficients: if $s<t$, the binomial coefficient $\binom st$ is considered to be 0. We view the binomial coefficients as elements of $K$ via the natural map $\mathbb{Z}\rightarrow K$.
\medskip

\noindent{\bf Notation:}\\
Let $p$ be as defined in the Introduction. If $p>1$, for an integer $a$ we shall denote by $\nu^{(p)}(a)$ the $p$-adic value of $a$, that is, the greatest power of $p$ that divides $a$. If $p=1$, we adopt the convention $\nu^{(p)}(a)=1$ for all non-zero $a$ and
$\nu^{(p)}(0)=\infty$.
\begin{remark}\label{diffkills} Consider $e,b,b'\in\mathbb N$ such that $b\ \not|\ p^e$ (this holds, in particular, whenever $b<p^e$) and $\left.p^e\ \right|\ b'$. Then $(x+\Delta x)^{b'}\in K\left[x^{p^e},\Delta x^{p^e}\right]$, so $\partial_bx^{b'}=0$.
\end{remark}
\begin{proposition}{\label{Proposition10.1}} Take an element $h\in K[x]\setminus\{0\}$.
\begin{enumerate}
\item[(1)] For all $b\in \mathbb{N}_0$ we have
\begin{equation}
\nu_i(h)-\nu_i\left(\partial_bh\right)\le\frac b{b_i}\left(\beta_i-\nu(\partial_{b_i}Q_i)\right).\label{tag10.1}
\end{equation}
\item[(2)] Let $h=\sum\limits_{j=0}^sd_{j,i}Q_i^j$ be an $i$-standard expansion of $h$. Assume that
$$
\left\{j\in\{0,\dots,s\}\ \left|\ \nu\left(d_{j,i}Q_i^j\right)=\nu_i(h)\right.\right\}\ne\{0\}
$$
(in particular, we have $s>0$). Let $d_{j,i}Q_i^j$ denote the term in the $i$-standard expansion of $h$ which minimizes the triple
$\left(\nu_i\left(d_{j,i}Q_i^j\right),\nu^{(p)}(j),j\right)$ in the lexicographical ordering. Let $e=\nu^{(p)}(j)$ and $b(i,h)=b_ip^e$. Then equality holds in (\ref{tag10.1}) for $b=b(i,h)$.
\end{enumerate}
\end{proposition}

\begin{remark}{\label{Remark10.2}} For all $b\in\mathbb{N}$ we have $\deg_x\partial_{b}Q_i<\bar\alpha_i$. By Proposition \ref{Proposition9.26b} there exists an ordinal $i_0$ such that $i_0+=i$ and
\begin{equation}
\nu(\partial_{b}Q_i)=\nu_{i_0}(\partial_{b}Q_i).\label{tag10.2}
\end{equation}
In particular $\nu(\partial_{b_i}Q_i)=\nu_{i_0}(\partial_{b_i}Q_i)$. Thus replacing $\nu(\partial_{b_i}Q_i)$ by
$\nu_{i_0}(\partial_{b_i}Q_i)$ in (\ref{tag10.1}) gives rise to an equivalent inequality. Also, $\nu_i\left(\partial_bh\right)\le\nu\left(\partial_bh\right)$, so replacing $\nu_i\left(\partial_bh\right)$ by $\nu\left(\partial_bh\right)$ in (\ref{tag10.1}) gives rise to a true, but an a priori weaker inequality.
\end{remark}

\begin{proof}[Proof of Proposition \ref{Proposition10.1}] We prove Proposition \ref{Proposition10.1} by transfinite induction. For $i=0$ we have $Q_i=x$, $b_i=1$ and $b=p^e$. Write $h=\sum\limits_{j=0}^sd_{j,0}x^j$. For $b=0$ and for $b>s$, (\ref{tag10.1}) is trivially true. Assume that $0<b\leq s$. There exists $j$, $0\leq j\leq s$, such that
$$
\nu_0(\partial_b h)=\nu\left(\binom{ j}{b}d_{j,0}x^{j-b}\right)\geq\nu(d_{j,0})+(j-b)\beta_0\geq \nu_0(h)-b\beta_0,
$$
which gives (\ref{tag10.1}).

To prove (2), chose $d_{j,0}x^j$ as in part (2), then $\nu_0(h)=\nu_0(d_{j,0}x^j)=\nu(d_{j,0})+j\beta_0$. By definition of $e$ and $b$, we have $\nu_0\left(\binom jb\right)=\nu\left(\binom jb\right)=0$. Now,
\begin{equation}
\begin{split}
\nu_0\left(\partial_{p^e} d_{j,0}x^j\right)=\nu_0\left(\binom{j}{p^e} d_{j,0}x^{j-p^e}\right)&=
\nu_0\left(\binom{j}{p^e}\right)+\nu(d_{j,0})+(j-p^e)\beta_0=\\
=&\nu(d_{j,0})+j\beta_0-p^e\beta_0=\nu_0(h)-p^e\beta_0.
\end{split}
\end{equation}
Assume that $i>1$ and that the result is known for all the ordinals strictly smaller than $i$.

\begin{lemma}{\label{el10.4}} Consider a pair of ordinals $i',i''$ such that $i'<i''\le i$. Then
\begin{equation}
\frac{\beta_{i'}-\nu\left(\partial_{b_{i'}}Q_{i'}\right)}{b_{i'}}<\frac{\beta_{i''}-\nu\left(\partial_{b_{i''}}Q_{i''}\right)}{b_{i''}}
\label{tag10.3}
\end{equation}
\end{lemma}

\begin{proof} {\bf Basic Case.} First, assume that $i''=i'+$ and that for every $\tilde b\in\mathbb N$ we have
\begin{equation}
\nu(\partial_{\tilde b}Q_{i''})=\nu_{i'}(\partial_{\tilde b}Q_{i''}).\label{eq:nuofpartialequalsnui'}
\end{equation}
By definition of $b_{i''}$, it is sufficient to prove that there exists a strictly positive integer $\tilde b$ such that (\ref{tag10.3}) holds with $b_{i''}$ replaced by $\tilde b$.

We take $\tilde b:=b\left(i',Q_{i''}\right)$. We have:
$$
\begin{aligned}
\beta_{i''}-\nu\left(\partial_{\tilde b}Q_{i''}\right)&>\nu_{i'}\left(Q_{i''}\right)-\nu\left(\partial_{\tilde b}Q_{i''}\right)=
\nu_{i'}\left(Q_{i''}\right)-\nu_{i'}\left(\partial_{\tilde b}Q_{i''}\right)=\\
&=\frac{\tilde b}{b_{i'}}\left(\beta_{i'}-\nu\left(\partial_{b_{i'}}Q_{i'}\right)\right).
\end{aligned}
$$
Here the first inequality is given by Proposition \ref{Proposition9.18}, the first equality by (\ref{eq:nuofpartialequalsnui'}) and the second equality by Proposition \ref{Proposition10.1} (2) applied to $i'<i$, which we are allowed to use by the induction assumption of Proposition \ref{Proposition10.1}. This completes the proof of the Basic Case.
\medskip

In the general case, we argue by contradiction. Assume that (\ref{tag10.3}) does not hold. Take the smallest ordinal $\ell''\le i''$ such that there exists an ordinal $\ell'$ satisfying $i'\le\ell'<\ell''$ and
\begin{equation}
\frac{\beta_{\ell'}-\nu\left(\partial_{b_{\ell'}}Q_{\ell'}\right)}{b_{\ell'}}\ge
\frac{\beta_{\ell''}-\nu\left(\partial_{b_{\ell''}}Q_{\ell''}\right)}{b_{\ell''}}.
\label{tag10.3contradiction}
\end{equation}
Increasing $\ell'$, if necessary, we may assume that
\begin{equation}
\ell''=\ell'+.\label{eq:l''=l'+}
\end{equation}
By Propositions \ref{Proposition9.26b} and \ref{nuiisavaluation} (2), by further increasing $\ell'$ we may assume, in addition, that
\begin{equation}
\nu\left(\partial_{\tilde b}Q_{\ell''}\right)=\nu_{\ell'}\left(\partial_{\tilde b}Q_{\ell''}\right)\quad\text{ for all }\ \tilde b\in\mathbb N.\label{eq:nuofpartialequalsnul'}
\end{equation}
Together with (\ref{tag10.3contradiction}) and (\ref{eq:l''=l'+}) this contradicts the Basic Case. The proof of the Lemma is complete.
\end{proof}
To prove Proposition \ref{Proposition10.1} (1), it is sufficient to prove it for each $i$-standard monomial appearing in the $i$-standard expansion of $h$. Indeed, let
\begin{equation}
h=\sum\limits_{\bar\gamma_{i+1}\in W}d_{\bar\gamma_{i+1}}\mathbf{Q}_{i+1}^{{\bar\gamma}_{i+1}}\label{eq:istandardexpofhW}
\end{equation}
be an $i$-standard expansion of $h$, where $W$ is a certain finite index set. Assume that the result is true for each $i$-standard monomial $\mathbf{Q}_{i+1}^{{\bar\gamma}_{i+1}}$ appearing in (\ref{eq:istandardexpofhW}). This means that for each
$\bar\gamma_{i+1}\in W$ we have
$$
\nu_i\left(\partial_{b}\mathbf{Q}_{i+1}^{{\bar\gamma}_{i+1}}\right)\ge\nu_i\left(\mathbf{Q}_{i+1}^{{\bar\gamma}_{i+1}}\right)-\frac b{b_i}\left(\beta_i-\nu\left(\partial_{b_i}Q_i\right)\right).
$$
Thus
\begin{equation}
\begin{split}
\nu_i(\partial_bh)\ge
\min\limits_{\bar\gamma_{i+1}}\nu_i\left(d_{\bar\gamma_{i+1}}\partial_b\mathbf{Q}_{i+1}^{{\bar\gamma}_{i+1}}\right)&\ge
\nu\left(d_{\bar\gamma_{i+1}}\right)+\min\limits_{\bar\gamma_{i+1}}\nu_i\left(\mathbf{Q}_{i+1}^{{\bar\gamma}_{i+1}}\right)-
\frac b{b_i}\left(\beta_i-\nu\left(\partial_{b_i}Q_i\right)\right)=\\
=&\nu_i(h)-\frac b{b_i}\left(\beta_i-\nu\left(\partial_{b_i}Q_i\right)\right),
\end{split}
\end{equation}
as desired (here the last equality holds because, by Definition \ref{de9.11}, a $Q_i$-free $i$-standard expansion is a finite sum of $Q_i$-free $i$-standard monomials whose $\nu$-value equals the minimum of the $\nu$-values of the monomials).

Let $\mathbf{Q}_{i+1}^{{\bar\gamma}_{i+1}}$ be such an $i$-standard monomial. Let $\bar\gamma_{i+1}=\left(\left.\gamma_j\ \right|\ j\le i\right)$ and write
$$
\mathbf{Q}_{i+1}^{{\bar\gamma}_{i+1}}=\mathbf{Q}_i^{\bar\gamma_i}Q_i^{\gamma_i}.
$$
We want to expand $\partial_b\mathbf{Q}_{i+1}^{{\bar\gamma}_{i+1}}$ in terms of products of the form
$Q_i^{\gamma_i-q}\left(\partial_{j_0}\mathbf{Q}_i^{{\bar\gamma}_i}\right)\prod\limits_{t=1}^q\left(\partial_{j_t}Q_i\right)$, where $q\le\gamma_i$ and $j_0+j_1+\dots+j_q=b$. Each such product appears in the sum with a certain positive integer coefficient that we will now compute explicitly.

To do that, we first prove some general formulae about Hasse derivatives of products and powers of polynomials. We start with the following generalized Leibniz rule.
 
\begin{lemma}\label{Lebnitzrule} For any two polynomials $A(X),B(X)\in K[X]$ and any positive integer $b$, we have
\begin{equation}
\partial_b(A(X)B(X))=\sum\limits_{j=0}^{b}(\partial_j A(X))(\partial_{b-j}B(X)).\label{Leibniz}
\end{equation}
\end{lemma}
\begin{proof} Let $m=\deg_XA(X)$ and $n=\deg_XB(X)$. By definition, Hasse derivatives are the coefficients in Taylor expansions:
$$
A(X+Y)=\sum\limits_{i=0}^{m}\partial_iA(X)Y^i\ \text{ and }\ B(X+Y)=\sum\limits_{i=0}^{n}\partial_iB(X)Y^i.
$$
We obtain
$$
AB(X+Y)=A(X+Y)B(X+Y)=\sum\limits_{b=0}^{m+n}\left(\sum\limits_{j=0}^b\partial_jA(X)\partial_{b-j}B(X)\right)Y^b.
$$
Since the coefficients in the Taylor expansion are uniquely determined, this proves (\ref{Leibniz}).
\end{proof}
For positive integers $a_1\le\dots\le a_q$, let $k$ denote the number of distinct elements in the set $\{a_1,\dots,a_q\}$. We define the multiplicities $n_1,\dots,n_k$ of $a_1,\dots,a_q$ as follows. Let $n_1$ be the number of appearances of $a_1$ in the sequence $a_1,\dots,a_q$. Let $n_2$ the number of appearances of the second smallest element of $\{a_1,\dots,a_q\}$, and so on until $n_k$, which is, by definition, the number of appearances of $a_q$.
\smallskip

\noi\textbf{Notation.} Take an integer $n\ge q$. Let $C_n(a_1,\dots,a_q)=\frac{n!}{(n-q)!\cdot n_1!\cdot\ldots\cdot n_k!}$.
\medskip

\begin{lemma}\label{Bn} For every polynomial $B(X)$ and all positive integers $b$ and $n$ we have
\begin{equation}
\partial_b (B(X)^{n})=\sum\limits_{\substack{1\leq q\le n\\j_1+\dots+j_q=b\\0<j_1\le\dots\le j_q}}C_n(j_1,\dots,j_q)B(X)^{n-q}
\prod\limits_{t=1}^q\left(\partial_{j_t}B(X)\right).\label{tag2}
\end{equation}
\end{lemma}
\begin{proof} Let $m=\deg\ B(X)$. We have
\begin{equation}
B(X+Y)^n=\left(\sum\limits_{i=0}^{m}\partial_iB(X)Y^i\right)^n=B(X)^n+
\sum\limits_{b=1}^{mn}\left(\sum\limits_{j_1+\dots+j_n=b}\prod\limits_{t=1}^{n}\partial_{j_t}B(X)\right)Y^{b},\label{eq:nonnegativeexponents}
\end{equation}
where the $j_t$ run over non-negative integers; the $j_t$ are not necessarily ordered by size.
\medskip

Next, we rewrite (\ref{eq:nonnegativeexponents}) in the following way. In each of the products appearing in parentheses on the right hand side of (\ref{eq:nonnegativeexponents}), we separate the terms with $j_t=0$ from those with $j_t>0$. Precisely, for each product appearing in parentheses on the right hand side of (\ref{eq:nonnegativeexponents}), let
$$
q=\#\left\{t\in\{1,\dots n\}\ \left|\ j_t\ne0\right.\right\}.
$$
Then we can rewrite (\ref{eq:nonnegativeexponents}) as
\begin{equation}
B(X+Y)^n=B(X)^n+
\sum\limits_{b=1}^{mn}\left(\sum\limits_{q=1}^n\sum\limits_{j_1+\dots+j_q=b}B^{n-q}(X)\prod\limits_{t=1}^q\partial_{j_t}B(X)\right)Y^b,\label{eq:positiveexponents}
\end{equation}
where the third sum is taken over all the distinct ways in which $b$ can be written as the sum of $q$ {\it strictly positive} integers $j_t$.  Again, the $j_t$ are not necessarily ordered by size. Now we have to count how often the same product
$B^{n-q}(X)\prod\limits_{t=1}^q\partial_{j_t}B(X)$ appears in the third sum on the right hand side. Fix $n$ non-negative integers $j_1\do j_n$ such that $\sum\limits_{t=1}^nj_t=b$. Renumber them so that $j_1\le\dots\le j_n$. How many distinct $n$-tuples can we obtain in (\ref{eq:nonnegativeexponents}) from the numbers $j_1,\dots,j_n$? If all the $j_t$ are distinct then the number is $n!$. But if some of the $j_t$'s are equal, then permuting them only among themselves does not produce new tuples. Similarly, if $n-q\ge2$, permuting the $(n-q)$ factors in $B^{n-q}(X)$ among themselves does not produce new tuples. Let the numbers $n_1,\dots,n_k$ be the multiplicities of the numbers $j_1,\dots,j_n$, defined above. Then the number of appearances of $B^{n-q}(X)\prod\limits_{t=1}^q\partial_{j_t}B(X)$ is
\begin{equation}
\frac{n!}{(n-q)!\cdot n_1!\cdot\ldots\cdot n_k!}=C_{n}(j_1,\dots,j_q),\label{tag10.5}
\end{equation}
where $0<j_1\le\dots\le j_q$. This proves that the coefficient of $Y^b$ in (\ref{eq:nonnegativeexponents}) equals
$$
\sum\limits_{\substack{1\leq q\le n\\j_1+\dots+j_q=b\\0<j_1\le\dots\le j_q}}C_n(j_1,\dots,j_q)B^{n-q}(X)
\prod\limits_{t=1}^q\left(\partial_{j_t}B(X)\right).
$$
By definition of Hasse derivative, this coeffecient is nothing but $\partial_b\left(B(X)^n\right)$. The proof is complete.
\end{proof}

\begin{lemma}\label{ABn}
For every two polynomials $A(X)$ and $B(X)$ and all positive integers $b$ and $n$, we have
\begin{equation}
\partial_b (A(X)B(X)^{n})=\sum\limits_{\substack{0\leq q\le n\\j_0+\dots+j_q=b\\0<j_1\le\dots\le j_q}} C_{n}(j_1,\dots,
j_q)B(X)^{n-q}\left(\partial_{j_0}A(X)\right)\prod\limits_{t=1}^q\left(\partial_{j_t}B(X)\right).\label{tag1}
\end{equation}
\end{lemma}
\begin{proof} By Lemma \ref{Lebnitzrule}, we have $\partial_b(A(X)B(X)^{n})=\sum\limits_{j_0=0}^{b}(\partial_{j_0}A(X))(\partial_{b-{j_0}}B(X)^{n})$.
Now Lemma \ref{ABn} follows from Lemma \ref{Bn}.
\end{proof}

Coming back to the proof of Proposition \ref{Proposition10.1} (1), we have
\begin{equation}
\partial_b\mathbf{Q}_{i+1}^{{\bar\gamma}_{i+1}}=
\sum\limits_{\substack{0\leq q\le\gamma_i\\j_0+j_1+\dots+j_q=b\\0<j_1\le\dots\le j_q}}C_{\gamma_i}(j_1,\dots j_q)Q_i^{\gamma_i-q}\left(\partial_{j_0}\mathbf{Q}_i^{{\bar\gamma}_i}\right)\prod\limits_{t=1}^q\left(\partial_{j_t}Q_i\right),\label{tag10.4}
\end{equation}
by Lemma \ref{ABn}. By Propositions \ref{Proposition9.26b} and \ref{nuiisavaluation} (2), we have
$$
\nu_i\left(\partial_{j_t}Q_i\right)=\nu\left(\partial_{j_t}Q_i\right).
$$
We have
\begin{equation}
\beta_i-\nu_i\left(\partial_{j_t}Q_i\right)=\beta_i-\nu\left(\partial_{j_t}Q_i\right)\le
\frac{j_t}{b_i}\left(\beta_i-\nu\left(\partial_{b_i}Q_i\right)\right)\label{tag10.6}
\end{equation}
by definition of $b_i$. Further,
\begin{equation}
\nu_i\left(\mathbf{Q}_i^{{\bar\gamma}_i}\right)-\nu_i\left(\partial_{j_0}\mathbf{Q}_i^{{\bar\gamma}_i}\right)=
\nu_{i_0}\left(\mathbf{Q}_i^{{\bar\gamma}_i}\right)-\nu_{i_0}\left(\partial_{j_0}\mathbf{Q}_i^{{\bar\gamma}_i}\right)
\le\frac{j_0}{b_i}(\beta_i-\nu(\partial_{b_i}Q_i)),\label{tag10.7}
\end{equation}
where $i_0$ is sufficiently large with $i_0+=i$, the equality holds because $Q_i$ does not appear in $\mathbf{Q}_i^{{\bar\gamma}_i}$ and by Proposition \ref{Proposition9.26b}, and the inequality by the induction assumption and in view of Lemma \ref{el10.4}. Note that by Lemma \ref{el10.4} the inequality in (\ref{tag10.7}) is strict whenever $j_0>0$. Adding the inequalities (\ref{tag10.6}) for $1\le t\le q$ and (\ref{tag10.7}), we obtain:
\begin{equation}
\begin{aligned}
&\nu_i\left(\mathbf{Q}_{i+1}^{{\bar\gamma}_{i+1}}\right)-
\nu_i\left(Q_i^{\gamma_i-q}\left(\partial_{j_0}\mathbf{Q}_i^{{\bar\gamma}_i}\right)\prod\limits_{t=1}^q\left(\partial_{j_t}Q_i\right)\right)\le\\
&\le\frac{j_0+j_1+\dots+j_q}{b_i}(\beta_i-\nu(\partial_{b_i}Q_i))=\frac b{b_i}(\beta_i-\nu(\partial_{b_i}Q_i)).
\end{aligned}\label{tag10.8}
\end{equation}
By (\ref{tag10.4}), (\ref{tag10.8}) and since $\nu$ is non-negative on $\mathbb{N}$ (in particular,
$\nu(C_{\gamma_i}(j_1,\dots,j_q))\ge0$), we have
\begin{equation}
\begin{aligned}
&\nu_i\left(\mathbf{Q}_{i+1}^{{\bar\gamma}_{i+1}}\right)-\nu_i\left(\partial_b\mathbf{Q}_{i+1}^{{\bar\gamma}_{i+1}}\right)\le\\
&\le\nu_i\left(\mathbf{Q}_{i+1}^{{\bar\gamma}_{i+1}}\right)-\min\limits_{(j_0\do j_q)}\left\{\nu_i\left(Q_i^{\gamma_i-q}
\left(\partial_{j_0}\mathbf{Q}_i^{{\bar\gamma}_i}\right)\prod\limits_{t=1}^q\left(\partial_{j_t}Q_i\right)\right)\right\}\le
\frac b{b_i}(\beta_i-\nu(\partial_{b_i}Q_i)),
\end{aligned}\label{tag10.9}
\end{equation}
as desired. Proposition \ref{Proposition10.1} (1) is proved.
\medskip

Now let the notation be as in Proposition \ref{Proposition10.1} (2); in particular, $b$ stands for $b(i,h)$. To prove this part, we first show that replacing $\nu_i(\partial_bh)$ by $\nu\left(\left(\partial_{b_i}Q_i\right)^{p^e}d_{j,i}Q_{i}^{j-p^{e}}\right)$ gives equality in (\ref{tag10.1}). Indeed, we have
\begin{equation}
\begin{aligned}
&\nu_i(h)-\nu\left(\left(\partial_{b_i}Q_i\right)^{p^e}d_{j,i}Q_{i}^{j-p^{e}}\right)=
\nu_i(h)-\left(p^e\nu_i(\partial_{b_i}Q_i)+\nu(d_{j,i})+(j-p^{e})\nu(Q_i)\right)=\\
=&p^{e}\nu(Q_i)-p^{e}\nu(\partial_{b_i}Q_i)=\frac{b_ip^{e}}{b_i}(\beta_i-\nu_i(\partial_{b_i}Q_i))=
\frac{b}{b_i}(\beta_i-\nu_i(\partial_{b_i}Q_i)).
\end{aligned}
\end{equation}
Therefore by part (1) of the Proposition, $\nu_i(\partial_bh)$ must be greater than or equal to
$$
\nu\left(\left(\partial_{b_i}Q_i\right)^{p^e}d_{j,i}Q_{i}^{j-p^{e}}\right)
$$
and it is sufficient to prove that the $i$-standard expansion of $\partial_{b}h$ contains a term of the form $dQ_{i}^{j-p^{e}}$ with
$\nu(d)=\nu\left(\left(\partial_{b_i}Q_i\right)^{p^e}d_{ji}\right)$ and all the other terms $d'Q_{i}^{j'}$ satisfy either $j'\neq j-p^{e}$ or $\nu_i\left(d'Q_i^{j'}\right)>\nu_i\left(dQ_i^{j-p^e}\right)$.
\medskip

We proceed by considering all the terms of the form $\mathbf{Q}_{i+1}^{{\bar\gamma}_{i+1}}$ that appear in the $i$-standard expansion of $h$.
\medskip

First, consider such a term appearing in $d_{j,i}Q_{i}^{j}$, of minimal value. Write $b=\sum\limits_{t=1}^q j_t $, where $q=p^e$ and $j_t=b_i$ for all $t$. For each $Q_i$-free standard monomial $\mathbf{Q}_i^{{\bar\gamma}_i}$, appearing in $d_{j,i}$, the corresponding term in (\ref{tag10.4}) is $\binom j{p^e}Q_i^{j-p^e}\mathbf{Q}_i^{{\bar\gamma}_i}\left(\partial_{b_i}Q_i\right)^{p^e}$ by (\ref{tag10.5}). Put $d=\binom j{p^e}\mathbf{Q}_i^{{\bar\gamma}_i}\left(\partial_{b_i}Q_i\right)^{p^e}$. The binomial coefficient $\binom j{p^e}$ is not divisible by $p$ by definition of $e$ and $j$, so $\nu\left(\binom j{p^e}\right)=0$. Using Proposition \ref{Proposition9.26b} and \ref{nuiisavaluation} (2) and Lemma \ref{qQ+r}, we obtain
$$
\nu(d)=\nu\left(\left(\partial_{b_i}Q_i\right)^{p^e}d_{j,i}\right).
$$
Now for any other choice of $j_0,j_1,\dots,j_q$ such that $q=p^e$ and $\sum\limits_{t=0}^qj_t=b$ we would have at least one
$t\in\{1\do q\}$ such that $j_t<b_i$. Therefore such terms satisfy strict inequality in (1) since they satisfy strict inequality in (\ref{tag10.7}) or in (\ref{tag10.6}) and hence their value is strictly greater than $\nu_i\left(dQ_i^{j-p^e}\right)$. We obtain
\begin{equation}
\begin{aligned}
\nu_{i}\left(d_{j,i}Q_{i}^{j}\right)-\nu_{i}\left(\partial_{b}\left(d_{j,i}Q_{i}^{j}\right)\right)=
\frac{b}{b_i}\left(\beta_i-\nu(\partial_{b_i}Q_i)\right).
\end{aligned}\label{partial_prop1}
\end{equation}

Now assume $\mathbf{Q}_{i+1}^{{\bar\gamma}_{i+1}}$ is such that
$\nu\left(\mathbf{Q}_{i+1}^{{\bar\gamma}_{i+1}}\right)>\nu_{i}\left(d_{j,i}Q_{i}^{j}\right)$. By (\ref{tag10.1}) we have
$$
\nu\left(\mathbf{Q}_{i+1}^{{\bar\gamma}_{i+1}}\right)-
\nu\left(\partial_{b}\mathbf{Q}_{i+1}^{{\bar\gamma}_{i+1}}\right)\leq\frac{b}{b_i}\left(\beta_i-\nu(\partial_{b_i}Q_i)\right)
$$
and using (\ref{partial_prop1}) we find
$\nu\left(\partial_{b}\mathbf{Q}_{i+1}^{{\bar\gamma}_{i+1}}\right)>
\nu_{i}\left(\partial_{b}\left(d_{j,i}Q_{i}^{j}\right)\right)=\nu_i\left(dQ_i^{j-p^e}\right)$.
\medskip

Now consider terms $\mathbf{Q}_{i+1}^{{\bar\gamma}_{i+1}}$ that appear in an expression $d_{m,i}Q_{i}^{m}$, $m\ne j$,  such that
$$
\nu\left(\mathbf{Q}_{i+1}^{{\bar\gamma}_{i+1}}\right)=\nu_{i}\left(d_{j,i}Q_{i}^{j}\right).
$$
It is sufficient to show that for $j'=m-q=j-p^{e}$ such terms satisfy the strict inequality in (\ref{tag10.1}), so in view of (\ref{partial_prop1}) their valuation is strictly greater than $\nu_i\left(dQ_i^{j-p^e}\right)$.
\medskip

Take one such term. We have two cases. If $m>j$ then, by definition of $j$, either
$\nu\left(d_{m,i}Q_{i}^{m}\right)>\nu\left(d_{j,i}Q_{i}^j\right)$ or $\nu\left(d_{m,i}Q_{i}^{m}\right)=\nu\left(d_{j,i}Q_{i}^j\right)$ and $\nu^{(p)}(m)>\nu^{(p)}(j)$. By Proposition \ref{Proposition10.1} (1), in either case such a term satisfies strict inequality in (\ref{tag10.1}).

If $m<j$ and $m-q=j-p^{e}$ then $q<p^{e}$, so for any choice of $j_0,j_1,\dots,j_t$ with $\sum\limits_{t=1}^qj_t=b$ we must have at least one $t\in\{1\do q\}$ such that $j_t<b_i$. Therefore such terms satisfy strict inequality in Proposition 7.1 (1) since they satisfy strict inequality in (\ref{tag10.7}) or (\ref{tag10.6}).
\end{proof}

\begin{remark}{\label{Remark10.5}} Let
\begin{equation}
I_{i,max}=\left\{\tilde b_i\in\mathbb{N}_0\ \left|\
\frac{\beta_i-\nu(\partial_{b_i}Q_i)}{b_i}=\frac{\beta_i-\nu(\partial_{\tilde b_i}Q_i)}{\tilde b_i}\right.\right\}.
\label{tag10.10}
\end{equation}
By definition, we have $b_i=\min\ I_{i,max}$. Proposition \ref{Proposition10.1} (1) holds equally well with $b_i$ replaced by any
$\tilde b_i\in I_{i,max}$. Similarly, Lemma \ref{el10.4} holds if the pair $(b_{i'},b_{i''})$ is replaced by $\left(\tilde b_{i'},\tilde
b_{i''}\right)$ with $\tilde b_{i'}\in\mathbb{N}$, $\tilde b_{i''}\in I_{i'',max}$.
\end{remark}

\begin{corollary}{\label{Corollary10.6}} For each ordinal $i\le\ell $, each $\tilde b_i\in I_{i,max}$ is of the form $\tilde b_i=p^{\tilde e_i}$ for some $\tilde e_i\in\mathbb{N}_0$. In particular, $b_i=p^{e_i}$ for some $e_i\in\mathbb{N}_0$. In the special case when $\ch\ k_\nu=0$ we have $p=1$ and so $I_{i,max}=\{b_i\}=\{1\}$.
\end{corollary}

\begin{proof} Write $\tilde b_i=p^{\tilde e_i}u$, where $p\ \not|\ u$ if $\ch\ k_\nu=p>0$, and $p^{\tilde e_i}=1$ if
$\ch\ k_\nu=0$. We want to prove that $u=1$. We argue by contradiction. Assume that $u>1$. We claim that we can write
\begin{equation}
\binom{\tilde b_i}{b'}\partial_{\tilde b_i}=\partial_{b'}\circ\partial_{b''},\label{tag10.11}
\end{equation}
where $b',b''$ are strictly positive integers such that
\begin{equation}
b'+b''=\tilde b_i\label{tag10.12}
\end{equation}
and
\begin{equation}
\nu\left(\binom{\tilde b_i}{b'}\right)=0\label{tag10.13}
\end{equation}
Indeed, we can take $b'=p^{\tilde e_i}$ and $b''=p^{\tilde e_i}(u-1)$. Now, by Remark \ref{Remark10.9} below, $p$ does not divide
$\binom{\tilde b_i}{b'}$ and therefore its natural image in $K$ is non-zero and its value is $0$ (as usual, we view $\binom{\tilde b_i}{b'}$ as an element of $K$ via the natural map $\mathbb{N}\rightarrow K$).
\medskip

Take $i_0$ as in (\ref{tag10.2}), with $b$ replaced by $b''$. We have
\begin{equation}
\beta_i-\nu\left(\partial_{\tilde b_i}Q_i\right)=
\left(\beta_i-\nu\left(\partial_{b''}Q_i\right)\right)+\left(\nu_{i_0}\left(\partial_{b''}Q_i\right)-
\nu_{i_0}\left(\partial_{\tilde b_i}Q_i\right)\right)\label{tag10.14}
\end{equation}
by (\ref{tag10.2}). By (\ref{tag10.11}), we have $\partial_{\tilde b_i}Q_i=\binom{\tilde b_i}{b'}\partial_{b'}\left(\partial_{b''}Q_i\right)$. Hence
\begin{equation}
\nu_{i_0}\left(\partial_{b''}Q_i\right)-\nu_{i_0}\left(\partial_{\tilde
b_i}Q_i\right)\le\frac{b'}{b_{i_0}}\left(\beta_{i_0}-\nu\left(\partial_{\tilde
b_{i_0}}Q_{i_0}\right)\right)<\frac{b'}{\tilde b_i}\left(\beta_i-\nu\left(\partial_{\tilde
b_i}Q_i\right)\right)\label{tag10.15}
\end{equation}
by (\ref{tag10.13}), Proposition \ref{Proposition10.1} (1) and Lemma \ref{el10.4}. From (\ref{tag10.14})--(\ref{tag10.15}) we obtain
$$
\beta_i-\nu(\partial_{b''}Q_i)>\left(1-\frac{b'}{\tilde b_i}\right)\left(\beta_i-\nu(\partial_{\tilde b_i}Q_i)\right)=
\frac{b''}{\tilde b_i}\left(\beta_i-\nu(\partial_{\tilde b_i}Q_i)\right)
$$
which contradicts the fact that $\tilde b_i\in I_{i,max}$. Corollary \ref{Corollary10.6} is proved.
\end{proof}

Next, we investigate further the case of equality in (\ref{tag10.1}). We give a necessary condition on $h$ and $b$ for the equality to hold in (\ref{tag10.1}) and prove that this condition is sufficient under some additional assumptions. Finally, we derive a formula for $\ini_ih$ in the case when this criterion for equality in (\ref{tag10.1}) holds. We start with the case when $h$ is a single $i$-standard monomial.

\begin{proposition}{\label{Proposition10.7}} Consider an $i$-standard monomial $h=\mathbf{Q}_{i+1}^{{\bar\gamma}_{i+1}}$. Write
\begin{align}
&b_i=p^{e_i}\qquad\text{ and}\label{tag10.16}\\
&\gamma_i=p^eu,\quad\ \text{ where }p\ \not|\ u\text{ if }p>1.\label{tag10.17}
\end{align}
\begin{enumerate}
\item[(1)] If equality holds in (\ref{tag10.1}) then
\begin{equation}
\left.p^{e+e_i}\ \right|\ b.\label{tag10.18}
\end{equation}
\item[(2)] We have the following partial converse to (1). Assume that (\ref{tag10.18}) holds and that
\begin{equation}
\text{either}\quad b=p^{e+e_i}\quad\text{or}\quad I_{i,max}=\{b_i\}.\label{tag10.19}
\end{equation}
Then equality holds in (\ref{tag10.1}) if and only if
\begin{equation}
\nu\left(\binom u{b/p^{e+e_i}}\right)=0.\label{tag10.20}
\end{equation}
\item[(3)] Assume that (\ref{tag10.18})--(\ref{tag10.20}) hold. Then
\begin{equation}
\ini_{\nu_i}\partial_b\mathbf{Q}_{i+1}^{{\bar\gamma}_{i+1}}=
\binom u{b/p^{e+e_i}}\,\ini_{\nu_i}\left(\mathbf{Q}_i^{{\bar\gamma}_i}Q_i^{\gamma_i-\frac b{b_i}}\left(\partial_{b_i}Q_i
\right)^{\frac b{b_i}}\right);\label{tag10.21}
\end{equation}
in particular, $\partial_b\mathbf{Q}_{i+1}^{{\bar\gamma}_{i+1}}\not\equiv0$.
\end{enumerate}
\end{proposition}

\begin{remark}{\label{Remark10.8}} If $b=p^{e+e_i}$ holds in Proposition \ref{Proposition10.7} (2) then $\frac b{p^{e+e_i}}=1$ and $\binom u{b/p^{e+e_i}}=u$, so (\ref{tag10.20}) holds automatically in this case.
\end{remark}

\begin{proof}[Proof of Proposition \ref{Proposition10.7}] We go through the proof of Proposition \ref{Proposition10.1} and analyze the case of equality at each step. We start with a general remark about binomial coefficients in positive and mixed characteristic.

\begin{remark}{\label{Remark10.9}} If $\ch\ k_\nu=0$, we have
\begin{equation}
\nu\left(\binom\gamma j\right)=0\label{tag10.22}
\end{equation}
for all non-negative integers $j\le\gamma$; similarly, $\nu(C_{\gamma_i}(j_1\do j_q))=0$ for every $\gamma_i$ and every
$q$-tuple $(j_1\do j_q)$ as in (\ref{tag10.4}). If $\ch\ k_\nu=p>0$, the following is a well known characterization of the equality (\ref{tag10.22}). Let $\gamma=k_0+pk_1+\dots+p^sk_s$ and $j=t_0+pt_1+\dots+p^st_s$, with $k_0,\dots,k_s,t_0\do t_s\in\{0,1,\dots,p-1\}$, denote the respective $p$-adic expansions of $\gamma$ and $j$ (where we allow one of the $(s+1)$-tuples $(k_0\do k_s)$ and $(t_0\do t_s)$ to end in zeroes). Then (\ref{tag10.22}) holds if and only if
\begin{equation}
k_j\ge t_j\quad\text{ for all }j\in\{0\do s\}.\label{tag10.23}
\end{equation}

We recall the proof for the reader's convenience. For a positive integer $n$, recall that $\nu^{(p)}(n!)$ denotes the $p$-adic value of $n!$. If $n=n_0+pn_1+\dots+p^sn_s$ is the $p$-adic expansion of $n$, we have
$$
\nu^{(p)}(n!)=n_1+\frac{p^2-1}{p-1}n_2+\dots+\frac{p^s-1}{p-1}n_s.
$$
Let $\gamma-j=l_0+pl_1+\dots+p^sl_s$ be the $p$-adic expansion of $\gamma-j$.

First, suppose that (\ref{tag10.23}) holds. Then $k_j=t_j+l_j$ for all $j$. We have
\begin{eqnarray}
\nu^{(p)}(\gamma!)&=k_1+\frac{p^2-1}{p-1}k_2+\dots+\frac{p^s-1}{p-1}k_s,\\
\nu^{(p)}(j!)&=t_1+\frac{p^2-1}{p-1}t_2+\dots+\frac{p^s-1}{p-1}t_s,\\
\nu^{(p)}((\gamma-j)!)&=l_1+\frac{p^2-1}{p-1}l_2+\dots+\frac{p^s-1}{p-1}l_s
\end{eqnarray}
Thus $\nu^{(p)}(\gamma!)=\nu^{(p)}(j!)+\nu^{(p)}((\gamma-j)!)$ and (\ref{tag10.22}) holds.

Conversely, assume that (\ref{tag10.23}) is not true. Let
\begin{equation}
(j_0,j_0+1,\dots,j_1-1,j_1)\label{eq:consecutive}
\end{equation}
be a maximal subsequence of $(1,\dots,s)$ consisting of consecutive integers such that $k_j\ne t_j+l_j$ for $j_0\le j\le j_1$. Then $k_{j_0}=t_{j_0}+l_{j_0}-p$, $k_j=t_j+l_j-p+1$ for $j_0<j<j_1$ and $k_{j_1}=t_{j_1}+l_{j_1}+1$. Thus the total contribution of (\ref{eq:consecutive}) to $\nu^{(p)}(\gamma!)-\nu^{(p)}(j!)-\nu^{(p)}((\gamma-j)!)$ is
$$
\frac{p^{j_1}-1}{p-1}-\sum\limits_{j=j_0+1}^{j_1-1}(p^j-1)-p
\frac{p^{j_0}-1}{p-1}=j_1-j_0\ge1.
$$
The quantity $\nu^{(p)}(\gamma!)-\nu^{(p)}(j!)-\nu^{(p)}((\gamma-j)!)$ is obtained by summing the contributions of all the subsequences of the form (\ref{eq:consecutive}), hence it is strictly positive, as desired.

More generally, consider a finite set of positive integers $d_1\do d_a$ such that $\sum\limits_{\ell=1}^ad_\ell=\gamma$. For
$\ell\in\{1\do a\}$, let $d_\ell=d_{\ell,0}+pd_{\ell,1}+\dots+p^sd_{\ell,s}$, with $d_{\ell,j}\in\{0,1,\dots,p-1\}$, denote the $p$-adic expansions of $d_\ell$ (again, we allow the $s$-tuple $(d_{\ell,0},d_{\ell,1}\do d_{\ell,s})$ to end in zeroes in order to be able to compare tuples of different lengths). Then
\begin{equation}
\nu\left(\frac{\gamma!}{d_1!\cdot\ldots\cdot d_a!}\right)=0\quad\text{ if and only if }\quad\sum\limits_{\ell=1}^ad_{\ell,j}<p\ \text{ for all }j\in\{0\do s\}.\label{eq:Cvalue0}
\end{equation} 
This is proved by induction on $a$, the case $a=2$ having been treated above. This gives a helpful necessary and sufficient condition for the equality $\nu(C_{\gamma_i}(j_1\do j_q))=0$ which will be used in the sequel. One consequence of the equivalent conditions stated in (\ref{eq:Cvalue0}) is
\begin{equation}
d_{\ell,j}\le k_j\quad\text{for all }\ell\in\{1\do a\}\text{ and }j\in\{0\do s\}.\label{eq:Cvalue0componentwise}
\end{equation}
Below, we will be particularly interested in the following special cases of (\ref{tag10.23}):
\begin{enumerate}
\item[(1)] If
\begin{equation}
\gamma=p^eu\quad\text{ with }p\not|\ u\label{tag10.24}
\end{equation}
then (\ref{tag10.22}) implies that $p^e\ |\ j$.
\item[(2)] We have the following partial converse to (1): if $\gamma$ is as in (\ref{tag10.24}) and $j=p^e$ then (\ref{tag10.23}) holds. Indeed, we have $t_e=1$, $t_j=0$ for $j\ne e$ and $k_e\ge1$. In this
case
$$
\binom\gamma j=\binom{p^eu}{p^e}=\frac{p^eu(p^eu-1)\cdot\ldots\cdot(p^eu-p^e+1)}{p^e!}\equiv u\ \md\ m_\nu
$$
since $\frac{p^eu-j}{p^e-j}\equiv1\ \md\ m_\nu$ for all $j\in\{1\do p^e-1\}$.
\end{enumerate}
This is the main situation in which Proposition \ref{Proposition10.7} will be applied in this paper.
\end{remark}
In the next Lemma, let $j_0\do j_t$ be as in the proof of Proposition \ref{Proposition10.1}.
\begin{lemma}{\label{el10.10}}
\begin{enumerate}
\item[(1)] The inequality in (\ref{tag10.6}) is strict unless $j_t\in I_{i,max}$.
\item[(2)] Let $\gamma_i$ and $b_i$ be as in (\ref{tag10.16})--(\ref{tag10.17}). Assume that $j_0=0$, and
\begin{equation}
j_t\in I_{i,max}\quad\text{ for }1\le t\le q.\label{tag10.25}
\end{equation}
If
\begin{equation}
\nu(C_{\gamma_i}(j_1\do j_q))=0\label{tag10.26}
\end{equation}
then
\begin{equation}
\left.p^{e+e_i}\ \right|\ b.\label{tag10.27}
\end{equation}
\item[(3)] Let the assumptions be as in (2) and assume, in addition, that $b=p^{e+e_i}$. Then (\ref{tag10.26}) holds if and only if $q=p^e$ and $j_1=\dots=j_q=b_i$.
\end{enumerate}
\end{lemma}

\begin{proof}
\begin{enumerate}
\item[(1)] is immediate from definitions.
\item[(2)] Let $\{p^{c_1}\do p^{c_\ell }\}\subset I_{i,max}$ with
\begin{equation}
e_i\le c_1<c_2<\dots<c_\ell \label{tag10.28}
\end{equation}
denote the set of {\it distinct} natural numbers appearing among $\{j_1\do j_q\}$ (cf. (\ref{tag10.25}) and Corollary \ref{Corollary10.6}). For $1\le j\le\ell $, let $a_j=\#\{t\in\{1\do q\}\ |\ j_t\le p^{c_j}\}$; let $a_0=0$.
Then
\begin{equation}
b=\sum\limits_{j=1}^\ell(a_j-a_{j-1})p^{c_j}.\label{tag10.29}
\end{equation}
Assume that (\ref{tag10.26}) holds. Apply Remark \ref{Remark10.9} (specifically, (\ref{eq:Cvalue0componentwise})), with
$a_j-a_{j-1}$ playing the role of $d_j$. By definition of $e$, we have $\left.p^e\ \right|\ \gamma_i$, so the first $e$ entries in the
$p$-adic expansion of $\gamma_i$ are 0. By (\ref{eq:Cvalue0componentwise}), the same must be true of each of $a_j-a_{j-1}$. In other words,
\begin{equation}
\left.p^e\ \right|\ a_j\quad\text{ for }1\le j\le\ell. \label{tag10.30}
\end{equation}
(\ref{tag10.28})--(\ref{tag10.30}) imply (\ref{tag10.27}), as desired.
\item[(3)] Assume, in addition, that $b=p^{e+e_i}$.

``Only if''. From (\ref{tag10.28})--(\ref{tag10.30}), we see that $\ell=1$ and $a_1=p^e$; the result follows immediately.

``If''. By assumptions, we have $\ell=1=q$ and $a_1=p^e$. By (\ref{tag10.5}) and Remark  \ref{Remark10.9} (2), we have
$$
C_{\gamma_i}(j_1\do j_q)=C_{\gamma_i}(\underset{p^e}{\underbrace{b_i\do b_i}})=\binom{\gamma_i}{p^e}=u\ \mod\ m_\nu
$$
and the result follows.
\end{enumerate}
\end{proof}

We can now finish the proof of Proposition \ref{Proposition10.7}.

By (\ref{tag10.7}) and Lemma \ref{el10.10} (1), the inequality in (\ref{tag10.8}) is strict unless $j_0=0$, and
\begin{equation}
j_t\in I_{i,max}.\label{tag10.31}
\end{equation}
By Lemma \ref{el10.10} (3), equality (\ref{tag10.26}) holds, so we may apply Lemma \ref{el10.10} (2). By Lemma \ref{el10.10} (2), the first inequality in (\ref{tag10.9}) is strict unless $j_0=0$ and $\left.p^{e+e_i}\ \right|\ b$. This proves (1) of the Proposition.

(2) Assume that (\ref{tag10.18}) holds. If $b=p^{e+e_i}$, by Lemma \ref{el10.10} (3) there is exactly one term on the right hand side of (\ref{tag10.4}) for which equality holds in (\ref{tag10.8}), namely, the term with $q=p^e$ and $j_1=\dots=j_q=b_i$. If $I_{i,max}=\{b_i\}$, then by Lemma \ref{el10.10} (1) there is at most one term on the right hand side of (\ref{tag10.4}) for which equality holds in (\ref{tag10.8}); if such a term exists, it is the term with $q=\frac b{b_i}$ and $j_1=\dots=j_q=b_i$. Moreover, this term satisfies equality in (\ref{tag10.8}) if and only if $\nu(C(\underset{b/b_i}{\underbrace{b_i\do
b_i}}))=\nu\left(\binom{\gamma_i}{b/b_i}\right)=\nu\left(\binom u{b/p^{e+e_i}}\right)=0$, where the second equality follows from Remark \ref{Remark10.9}. In either case, there is at most one term on the right hand side of (\ref{tag10.4}) for which equality holds in (\ref{tag10.8}), and there is exactly one such term if and only if $\nu\left(\binom u{b/p^{e+e_i}}\right)=0$. This proves (2).

(3) of the Proposition follows from (2) and (\ref{tag10.4}).
\end{proof}

Let
$$
\gamma_i=k_0+pk_1+\dots+p^sk_s,
$$
with $k_0,\dots,k_s\in\{0,1,\dots,p-1\}$, denote the $p$-adic expansion of $\gamma_i$. Take integers $s'\in\{0\do s\}$, $k'_{s'}\in\{0\do k_{s'}\}$. Let $b=(k_0+pk_1+\dots+p^{s'-1}k_{s'-1}+p^{s'}k'_{s'})b_{i}$.

\begin{corollary}{\label{Corollary10.11}} For this $b$, equality holds in (\ref{tag10.1}) for the monomial
$h=\mathbf{Q}_{i+1}^{{\bar\gamma}_{i+1}}$. The element $\ini_i\partial_b\mathbf{Q}_{i+1}^{{\bar\gamma}_{i+1}}$ is given by the formula (\ref{tag10.21}).
\end{corollary}

\begin{proof} Repeated application of Proposition \ref{Proposition10.7} (2) and (3), first $k_0$ times with $b$ replaced by 1, then $k_1$ times with $b$ replaced by $p$, and so on.
\end{proof}

Let $h=\sum\limits_{j=0}^sd_{j,i}Q_i^j$ be an $i$-standard expansion. Let $S_i=S_i(h)$, where the notation is as in (\ref{tag9.24}). Write $\ini_ih=\sum\limits_{j\in S_i}\ini_i\left(d_{j,i}Q_i^j\right)$. Write $b_i=p^{e_i}$, as above. Let $e$ be the greatest non-negative integer such that for all $j\in S_i$ we have $\left.p^e\ \right|\ j$.

\begin{proposition}{\label{Proposition10.12}}
\begin{enumerate}
\item[(1)] If equality holds in (\ref{tag10.1}) then
\begin{equation}
\left.p^{e+e_i}\ \right|\ b.\label{tag10.32}
\end{equation}
\item[(2)] Assume that
\begin{equation}
b=p^{e+e_i}.\label{tag10.33}
\end{equation}
Then equality holds in (\ref{tag10.1}). In particular, we have $\partial_bh\not\equiv0$.
\item[(3)] Assume that (\ref{tag10.33}) holds. Let $S_{i}^b=\left\{j\in S_i\ \left|\ p^{e+1}\text{ does not divide }j\right.\right\}$. Then
$$
\ini_{\nu_i}\partial_bh=
\sum\limits_{j\in S_{i}^b}\ini_{\nu_i}\left(\binom j{p^e}d_{j,i}Q_i^{j-p^e}\left(\partial_{b_i}Q_i\right)^{p^e}\right).
$$
\end{enumerate}
\end{proposition}

\begin{proof} (1), (2) and (3) of Proposition \ref{Proposition10.12} follow, respectively, from (1), (2) and (3) of Proposition \ref{Proposition10.7}.
\end{proof}

\begin{corollary}{\label{Corollary10.13}} In the notation of Proposition \ref{Proposition10.12}, assume that (\ref{tag10.33}) holds. We have
\begin{equation}
h\notin K\left[x^{p^{e+e_i+1}}\right].\label{tag10.34}
\end{equation}
\end{corollary}

\begin{proof} Take $b$ as in (\ref{tag10.33}). Now the result follows from Proposition \ref{Proposition10.12} (2).
\end{proof}

Let the notation be as in Proposition \ref{Proposition10.12}.

\begin{proposition}{\label{Proposition10.14}} Take an element $j\in S_i$. Write $j=p^eu$, where
$$
 p\ \not|\ u\quad\text{if }\quad\ch\ k_\nu=p>0.
$$
Assume that
\begin{equation}
\left.p^{e+1}\ \right|\ j'\quad\text{ for all }j'\in S_i,j'<j.\label{tag10.35}
\end{equation}
Let $u=t_0+t_1p+\dots+t_sp^s$ be the $p$-adic expansion of $u$. Then
\begin{equation}
\nu_i(\partial_{jb_i}h)=\nu_i(h)-j(\beta_i-\nu(\partial_{b_i}Q_i)),\label{tag10.36}
\end{equation}
\begin{equation}
\ini_{\nu_i}\partial_{jb_i}h=\left(\prod_{q=1}^st_q!\right)\ini_{\nu_i}d_{j,i}\left(\ini_{\nu_i}\partial_{b_i}Q_i\right)^j+\text{ terms involving }\ini_{\nu_i}Q_i.\label{tag10.37}
\end{equation}
For every $j'\ne j$ we have
\begin{equation}
\frac{\nu_i(h)-\nu_i(\partial_{j'b_i}h)}{j'}\le\frac{\nu_i(h)-\nu_i(\partial_{jb_i}h)}j,\label{tag10.38}
\end{equation}
and the inequality is strict whenever $j'\notin S_i$ or $j'<j$.
\end{proposition}

\begin{proof} By (\ref{tag10.35}) and Proposition \ref{Proposition10.7} (1), terms of the form $d_{j',i}Q_i^{j'}$ with
$$
j'\in S_i,\ j'<j
$$
satisfy strict inequality in (\ref{tag10.1}) with $b=jb_i$. Thus replacing $h$ by $\sum\limits_{j'=j}^sd_{j',i}Q_i^{j'}$ does not affect any of the statements of the Proposition. Apply Proposition \ref{Proposition10.12} repeatedly $t_0+t_1+\dots+t_s$ times. By (2) of Proposition \ref{Proposition10.12},
\begin{equation}
\nu_i\left(\partial_{jb_i}\left(d_{j,i}Q_i^j\right)\right)=\nu_i\left(d_{j,i}Q_i^j\right)-j(\beta_i-\nu(\partial_{b_i}Q_i))\label{tag10.39}
\end{equation}
and
\begin{equation}
\nu_i(\partial_{jb_i}h)=\nu_i(h)-j(\beta_i-\nu(\partial_{b_i}Q_i));\label{tag10.40}
\end{equation}
this proves (\ref{tag10.36}). (\ref{tag10.37}) follows from Proposition \ref{Proposition10.12} (3), by induction on $u$. Finally, the
last statement of the Proposition follows from Proposition \ref{Proposition10.12} (1) and (3), by induction on $u$.
\end{proof}

\begin{remark}{\label{Remark10.15}} Here is an alternative, more explicit explanation of (\ref{tag10.37}). Take $j'\in\{j\do s\}$ and apply (\ref{tag10.4}) to one of the monomials appearing in $d_{j',i}Q_i^{j'}$ (we take $\gamma_i=j'$ and $b=jb_i$ in (\ref{tag10.4})), in order to decide which values of $j'$ and which decompositions $j_0+\dots+j_q=b$ contribute to $\ini_i\partial_bh$.

If either $j'>j$, $q\ne j$, $j_0\ne0$ or $j_t\ne b_i$ for some $t\in\{1\do j\}$ then, by definition of $b_i$, the corresponding term in (\ref{tag10.4}) is either divisible by $Q_i$ or has $\nu_i$-value strictly greater than
$\nu_i\left(d_{j,i}Q_i^j\right)-j(\beta_i-\nu(\partial_{b_i}Q_i))$. This proves (\ref{tag10.37}).

Let $\sum\limits_qc_{q,i}Q_i^q$ denote the $i$-standard expansion of $\partial_{jb_i}h$. The above considerations prove that $c_{0,i}$ coincides with the coefficient of $Q_i^0$ in the $i$-standard expansion of $d_{j,i}\left(\partial_{b_i}Q_i\right)^j$ modulo an element of higher $\nu_i$-value. In particular, $c_{0,i}\ne0$ and
\begin{equation}
\nu(c_{0,i})=\nu_i(c_{0,i}).\label{tag10.41}
\end{equation}
We have
\begin{equation}
\nu(c_{0,i})=\nu_i(\partial_{jb_i}h)=\nu_i\left(\partial_{jb_i}\left(d_{j,i}Q_i^j\right)\right)=
\nu_i\left(d_{j,i}Q_i^j\right)-j(\beta_i-\nu(\partial_{b_i}Q_i)).\label{tag10.42}
\end{equation}
\end{remark}

\begin{corollary}{\label{Corollary10.16}} We have
\begin{equation}
\nu_i(h)=\min\limits_{0\le j\le s}\{\nu_i(\partial_{jb_i}h)+j(\beta_i-\nu(\partial_{b_i}Q_i))\}\label{tag10.43}
\end{equation}
and the minimum in (\ref{tag10.43}) is attained for all $j\in S_i$, satisfying (\ref{tag10.35}).
\end{corollary}

\section{Infinite sequences of key polynomials}\label{Infsequenceskeypolynom}

In this section, we take an ordinal $\ell$ such that $\ell+t\in\bar\Lambda$ and $\alpha_{\ell+t}=1$ for all $t\in\mathbb N$. Take an element $h\in K[x]$. Proposition \ref{Proposition11.2} (1) implies that $\delta_{\ell+t}(h)$ stabilizes for $t$ sufficiently large. Let $\delta(h)$ denote this stable value of $\delta_{\ell+t}(h)$. For a positive integer $t$ we have
\[
\delta_{\ell+t}(h)=0\implies\nu(h)=\nu_{\ell+t}(h)\implies\delta_{\ell+t+1}(h)=0.
\]
Thus saying that $\nu(h)=\nu_{\ell+t}(h)$ for all $t$ sufficiently large is equivalent to saying that $\delta_{\ell+t}(h)=0$ for all $t$ sufficiently large.

Assume that there exists $h\in K[x]$ such that
\begin{equation}
\nu(h)>\nu_{\ell+t}(h)\quad\text{ for all }t\in\mathbb{N}\label{tag12.28}
\end{equation}
(in other words, $\delta_{\ell+t}(h)>0$ for all $t\in\mathbb{N}$ and $\ell+\omega\in\bar\Lambda$); put
$h=Q_{\ell+\omega}$. One of the three main results of this section says that $\delta(h)$ has the form $p^e$ for some $e\in\mathbb{N}_0$ (in particular, $\delta(h)=1$ if $\ch\ k_\nu=0$). To prove this, we use differential operators and their properties derived in \S\ref{Keypolanddiffop}.

The second main result of this section is the statement that if either $\ch\ k_\nu=0$ or $p\ \not|\ \delta(h)$ then the sequences
$\left(\beta_{\ell+t}\right)_{t\in\mathbb{N}}$ and $\left(\nu_{\ell+t}(h)\right)_{t\in\mathbb{N}}$ are unbounded in $\tilde\g_0$ (this is precisely the situation of Proposition \ref{Proposition9.30} of subsection \ref{setswomaxelts}); in particular, the set
$\mathbf{Q}_{\ell+\omega}$ of key polynomials is $\tilde{\g}_0$-complete by Proposition \ref{Proposition9.30}. Finally, in Remark \ref{Remark12.6} we take $\ell=1$ and assume that $\alpha_t=1$ for all $t\in\mathbb N$ and that the sequence
$\{\beta_t\}_{t\in\mathbb{N}}$ is unbounded in $\tilde\g_0$. We show that $h\in K\left[x^{\delta(h)}\right]$.

Replacing $\ell$ by $\ell+t$ for a sufficiently large $t$, we may assume that $\delta_{\ell+t}(h)=\delta(h)$ for all positive integers $t$. Below the ordinal $i$ will run over the set $\{\ell+t\ |\ t\in\mathbb{N}_0\}$. By definition, for all such $i$ we have
\begin{equation}
Q_{i+1}=Q_i+z_i,\label{tag12.1}
\end{equation}
where $z_i$ is a homogeneous $Q_\ell$-free standard expansion of value $\beta_i$ (cf. Proposition \ref{Proposition9.12}). By Proposition \ref{Proposition9.17} (2), we have
\begin{equation}
\deg_xz_i<\deg_xQ_i.\label{tag12.2}
\end{equation}
Finally,
\begin{equation}
\ini_{\nu}Q_i=-\ini_{\nu}z_i\label{tag12.3}
\end{equation}
by (\ref{eq:valuestrictlygreater}). As before, let
$$
h=\sum\limits_{j=0}^{s_i}d_{j,i}Q_i^j
$$
be an $i$-standard expansion of $h$ for $i\ge\ell$, where each $d_{j,i}$ is a $Q_\ell$-free $i$-standard expansion. Note that since
$\alpha_{\ell+t}=1$ for $t\in\mathbb{N}_0$, we have
$$
\deg_xQ_i=\prod\limits_{j=2}^i\alpha_j=\prod\limits_{j=2}^\ell\alpha_j=\deg_xQ_\ell
$$
and so
\begin{equation}
s_i=\left[\frac{\deg_xh}{\deg_xQ_i}\right]=\left[\frac{\deg_xh}{\deg_xQ_\ell }\right]=s_\ell .\label{tag12.4}
\end{equation}

\begin{proposition}\label{Proposition12.1} For each $i$ of the form $i=\ell+t$, $t\in\mathbb{N}$, we have $b_{i+1}\le b_i$.
\end{proposition}

\begin{proof} Write $Q_{i+1}=Q_i+z_i$, as above.

\begin{lemma}{\label{el12.2}} For each $b\in\mathbb{N}$ we have
\begin{equation}
\frac{\beta_i-\nu(\partial_bz_i)}b<\frac{\beta_i-\nu(\partial_{b_i}Q_i)}{b_i}.
\label{tag12.5}
\end{equation}
If, in addition, $b\ge b_i$, then
\begin{equation}
\frac{\beta_{i+1}-\nu(\partial_bz_i)}b<\frac{\beta_{i+1}-\nu(\partial_{b_i}Q_i)}{b_i}.\label{tag12.6}
\end{equation}
\end{lemma}

\begin{proof} Let $i'$ denote the smallest ordinal such that
\begin{equation}
\nu_{i'}(z_i)=\nu(z_i)\quad\text{and}\quad\nu_{i'}(\partial_bz_i)=\nu(\partial_bz_i);\label{tag12.7}
\end{equation}
by Proposition \ref{Proposition9.26b} and (\ref{tag12.2}), $i'<i$. Let $z_i=\sum\limits_{j=0}^{s_{i'}}c_{j,i'}Q_{i'}^j$ be an
$i'$-standard expansion of $z_i$. By Lemma \ref{el10.4} we have
\begin{equation}
\frac{\beta_i-\nu(\partial_{b_i}Q_i)}{b_i}>\frac{\beta_{i'}-\nu(\partial_{b_{i'}}Q_{i'})}{b_{i'}}.\label{tag12.8}
\end{equation}
Combining (\ref{tag12.8}) with (\ref{tag10.1}), we obtain
\begin{equation}
\frac{\nu_{i'}(z_i)-\nu_{i'}(\partial_bz_i)}b\le\frac{\beta_{i'}-\nu(\partial_{b_{i'}}Q_{i'})}{b_{i'}}<
\frac{\beta_i-\nu(\partial_{b_i}Q_i)}{b_i},\label{tag12.10}
\end{equation}
which gives the inequality (\ref{tag12.5}). If $b\ge b_i$, (\ref{tag12.6}) follows immediately by adding the inequality
$\frac{\beta_{i+1}-\beta_i}b\le\frac{\beta_{i+1}-\beta_i}{b_i}$ to (\ref{tag12.5}).
\end{proof}

\begin{corollary}{\label{Corollary12.3}} We have
\begin{equation}
\nu(\partial_{b_i}z_i)>\nu(\partial_{b_i}Q_i)=\nu(\partial_{b_i}Q_{i+1}).
\label{tag12.11}
\end{equation}
\end{corollary}

\begin{proof} The inequality in (\ref{tag12.11}) is a special case of (\ref{tag12.5}) when $b=b_i$. The equality in (\ref{tag12.11}) follows immediately from the inequality by the ultrametric triangle law.
\end{proof}

To prove Proposition \ref{Proposition12.1}, we argue by contradiction. Suppose that
\begin{equation}
b_{i+1}>b_i.\label{tag12.12}
\end{equation}
Letting $b=b_{i+1}$ in Lemma \ref{el12.2}, we obtain
\begin{equation}
\frac{\beta_{i+1}-\nu(\partial_{b_{i+1}}z_i)}{b_{i+1}}<\frac{\beta_{i+1}-\nu(\partial_{b_i}Q_i)}{b_i}.\label{tag12.13}
\end{equation}
We have
\begin{equation}
\frac{\beta_i-\nu(\partial_{b_{i+1}}Q_i)}{b_{i+1}}\le\frac{\beta_i-\nu(\partial_{b_i}Q_i)}{b_i}\label{tag12.14}
\end{equation}
by definition of $b_i$. Combining (\ref{tag12.14}) with (\ref{tag12.12}), we obtain
\begin{equation}
\frac{\beta_{i+1}-\nu(\partial_{b_{i+1}}Q_i)}{b_{i+1}}<\frac{\beta_{i+1}-\nu(\partial_{b_i}Q_i)}{b_i}.\label{tag12.15}
\end{equation}
We can rewrite (\ref{tag12.13}) and (\ref{tag12.15}) as
\begin{equation}
\min\{\nu(\partial_{b_{i+1}}Q_i),\nu(\partial_{b_{i+1}}z_i)\}>
\beta_{i+1}-\frac{b_{i+1}}{b_i}\left(\beta_{i+1}-\nu(\partial_{b_i}Q_i)\right).
\label{tag12.16}
\end{equation}
Since $\partial_{b_{i+1}}Q_{i+1}=\partial_{b_{i+1}}Q_i+\partial_{b_{i+1}}z_i$, (\ref{tag12.16}) shows that
$$
\nu(\partial_{b_{i+1}}Q_{i+1})>\beta_{i+1}-\frac{b_{i+1}}{b_i}\left(\beta_{i+1}-\nu(\partial_{b_i}Q_i)\right),
$$
which contradicts the definition of $b_{i+1}$. This completes the proof of Proposition \ref{Proposition12.1}.
\end{proof}
\begin{corollary}\label{onlyone} Keep the above notation. Assume that $b_{i+1}=b_i$. Then $I_{i+1,max}=\{b_{i+1}\}$.
\end{corollary}
\begin{proof} Take an integer
\begin{equation}
b>b_{i+1}=b_i.\label{eq:b>bi=bi+1}
\end{equation}
Then
\begin{equation}
\frac{\beta_{i+1}-\beta_i}b<\frac{\beta_{i+1}-\beta_i}{b_i}.\label{eq:increment}
\end{equation}
By definition of $b_i$, we have
\begin{equation}
\frac{\beta_i-\nu(\partial_bQ_i)}b\le\frac{\beta_i-\nu(\partial_{b_i}Q_i)}{b_i}.\label{eq:definitionofbi}
\end{equation}
Adding up (\ref{eq:increment}) and (\ref{eq:definitionofbi}) and using Corollary \ref{Corollary12.3}, we obtain
$$
\frac{\beta_{i+1}-\nu(\partial_bQ_{i+1})}b<\frac{\beta_{i+1}-\nu(\partial_{b_{i+1}}Q_{i+1})}{b_{i+1}},
$$
so $b\notin I_{i+1,max}$. This proves the Corollary.
\end{proof}
Recall that $\delta(h)$ denotes the stable value of $\delta_{\ell+t}(h)$ for all sufficiently large integers $t$. Set $\delta:=\delta(h)$. Replacing $\ell$ by $\ell+t$ for a suitable integer $t$, we may assume that $\delta_{\ell+t}(h)=\delta$ for all $t\in\mathbb N$. Write
$\delta=p^eu$, where $p\ \not|\ u$ if $p>1$.

If $\ch\ k_\nu=0$, equations (\ref{tag9.27bis}) and (\ref{tag11.5}) imply that $d_{\delta-1,\ell}\ne0$ and
\begin{equation}
g_{1,\ell}=\bar Q_\ell +\ini_{\nu}\frac{d_{\delta-1,\ell}}{\delta\ d_{\delta,\ell}}.\label{tag12.17}
\end{equation}
If $\ch\ k_\nu=p>0$ then, according to Proposition \ref{Proposition11.2} (2) and using the notation of (\ref{tag9.24}), we see that for $i=\ell+t$, $t\in\mathbb N_0$,
\begin{equation}
\delta-p^e\in S_i(h)\label{tag12.18}
\end{equation}
(in particular, $d_{\delta-p^e,i}\ne0$) and
\begin{equation}
\ini_iz_i=\left(\frac{\ini_id_{\delta-p^e,i}}{u\ \ini_id_{\delta,i}}\right)^{\frac1{p^e}}.\label{tag12.19}
\end{equation}
The equation (\ref{tag11.5}) rewrites as
\begin{equation}
\ini_\ell h=\ini_{\nu}d_{\delta,\ell}g_{1,\ell}^\delta.\label{tag12.20}
\end{equation}
Take ordinals $i$ and $\ell_1$ such that $\ell<\ell_1<i<\ell+\omega$. Next, we prove a comparison result showing that the coefficient $d_{\delta,i}$ agrees with $d_{\delta,\ell_1}$ modulo terms of sufficiently high value.

\begin{proposition}{\label{Proposition12.4}} Assume that
\begin{equation}
\delta_{i+1}(h)=\delta_\ell (h)=\delta.\label{tag12.21}
\end{equation}
We have
\begin{equation}
d_{\delta,i}\equiv d_{\delta,\ell_1}\ \mod\ \mathbf
P_{\nu\left(d_{\delta,\ell_1}\right)+\min\left\{\nu_{\ell_1}^+(h)-\nu_{\ell_1}(h),\beta_{\ell_1}-\beta_\ell\right\}}.\label{tag12.24}
\end{equation}
\end{proposition}

\begin{proof} By definitions, we have $Q_i=Q_{\ell_1}+z_{\ell_1}+\dots+z_{i-1}$, where the $z_i$ are homogeneous $Q_i$-free
$i$-standard expansions. For simplicity, write $z:=z_{\ell_1}+\dots+z_{i-1}$. We will compare the $\ell_1$-standard expansion of $h$ with the $i$-standard one. To this end, we substitute $Q_i=Q_{\ell_1}+z$ into the $i$-standard expansion of $h$. We obtain
\begin{equation}
h=\sum\limits_{j=0}^{s_i}d_{j,i}Q_i^j=\sum\limits_{j=0}^{s_i}d_{j,i}(Q_{\ell_1}+z)^j.\label{tag12.26}
\end{equation}
First note that $\deg_x\sum\limits_{j=0}^{\delta-1}d_{j,i}\left(Q_{\ell_1}+z\right)^j<\delta \deg_xQ_i$. Hence $d_{\delta,\ell_1}$ is completely determined by $d_{\delta,i},d_{\delta+1,i},\dots,d_{s_i,i}$. Next, by (\ref{eq:j'ki}) in the proof of Proposition \ref{Proposition11.2} (3), for
\[
0<j\le s_i-\delta
\]
the coefficient $d_{\delta+j,i}$ is a sum of terms of the form $d_{j',k,\ell_1}$ with $j'\ge j+\delta$ satisfying
\[
\nu\left(d_{j',k,\ell_1}\right)=\nu_{\ell_1}\left(d_{j',k,\ell_1}\right)\ge(j'-j-\delta)\beta_{\ell_1}+\nu\left(d_{j',\ell_1}\right).
\]
Hence
\begin{equation}
\nu_{\ell_1}\left(d_{\delta+j,i}Q_i^{\delta+j}\right)\ge\nu_{\ell_1}^+(h),\label{tag12.27}
\end{equation}
so for $0<j\le s_i-\delta$ the terms $d_{\delta+j,i}Q_i^{\delta+j}$ in (\ref{tag12.26}) contribute nothing to
\begin{equation}
d_{\delta,\ell_1}\mod\ \mathbf P_{(\nu_{\ell_1}(h)-\delta\beta_{\ell_1})+\min\{\nu_{\ell_1}^+(h)-
\nu_{\ell_1}(h),\beta_{\ell_1}-\beta_\ell\}}.\label{dvimod}
\end{equation}
Therefore, the only term on the left hand side of (\ref{tag12.26}) that affects the element (\ref{dvimod}) is\linebreak
$d_{\delta,i}(Q_{\ell_1}+z)^{\delta}$.

We have
\[
d_{\delta,i}Q_i^\delta=d_{\delta,i}\sum\limits_{j=0}^\delta\binom\delta jQ_{\ell_1}^{\delta-j}z^j.
\]
For $j<\delta$, the coefficient of $Q_{\ell_1}^j$ in the $\ell_1$-standard expansion of $d_{\delta,i}z^j$ contributes to
$d_{\delta,\ell_1}$. Let us denote this coefficient by $d'_j$. We have
\[
\nu_\ell(z)=\nu(z)
\]
and
\[
\nu_\ell(d_{\delta,i})=\nu(d_{\delta,i}).
\]
By Lemma \ref{el11.3} (1), the quantity $\nu_\ell(d_{\delta,i}z^{j})$ is the minimum of the $\nu_\ell$-values of the summands appearing in the $\ell_1$-standard expansion of $d_{\delta,i}z^{j}$.  Thus
\[
\nu_\ell\left(d_{\delta,i}z^j\right)=\nu\left(d_{\delta,i}z^j\right)\le\nu_\ell\left(d'_jQ_{\ell_1}^j\right).
\]
Hence
\[
\nu\left(d'_jQ_{\ell_1}^j\right)=\nu_\ell\left(d'_jQ_{\ell_1}^j\right)+j\left(\beta_{\ell_1}-\beta_\ell\right)\geq
\nu\left(d_{\delta, i}z^j\right)+j\left(\beta_{\ell_1}-\beta_\ell\right)=\nu\left(d_{\delta, i}\right)+2j\beta_{\ell_1}-j\beta_\ell.
\]
This shows that $\nu(d'_j)\ge\nu(d_{\delta, i})+j\left(\beta_{\ell_1}-\beta_\ell\right)\ge\nu(d_{\delta, i})+\left(\beta_{\ell_1}-\beta_\ell\right)$, so for $j>0$ the term $d'_jQ_{\ell_1}^j$ does not affect the element (\ref{dvimod}). This completes the proof.
\end{proof}

\begin{proposition}{\label{Proposition12.5}} The integer $\delta$ is of the form $\delta=p^e$ for some $e\in\mathbb{N}_0$ (in particular, $\delta=1$ whenever $\ch\ k_\nu=0$).
\end{proposition}

\begin{proof} We give a proof by contradiction. Write
\begin{equation}
\delta=p^ev,\quad\text{ where }p\ \not|\ v\quad\text{ if }\ch\ k_\nu=p>0.\label{tag12.29}
\end{equation}
Suppose that $v>1$. By Proposition \ref{Proposition12.1}, the sequence $\{b_i\}$ is non-increasing with $t$ and hence stabilizes for $t$ sufficiently large. Let $b_\infty$ denote the stable value of $b_i$. Write $b_\infty=p^{e_\infty}$. Let $b=p^{e+e_\infty}$ and $g=\partial_bh$. By Proposition \ref{Proposition11.2} (2), $\ini_ih$ has the form (\ref{tag11.5}) for $i=\ell+t$, as $t$ runs over
$\mathbb{N}_0$, in particular, $S_i(h)=\left\{0,p^e,2p^e\do vp^e\right\}$ and is independent of $i$. Thus $p^e$ is the same as in Proposition \ref{Proposition10.12}. Hence $h$ and $b$ satisfy the hypotheses of Proposition \ref{Proposition10.12}. By Proposition \ref{Proposition10.12} (3) and (\ref{tag11.5}), $g\not\equiv0$ and, for $t$ sufficiently large, we have
$$
\ini_{\nu_i}g=v\ \ini_{\nu_i}\left(d_{\delta,i}\left(\partial_{b_{i}}Q_{i}\right)^{p^{e}}\left(Q_i+z_i\right)^{\delta-p^e}\right).
$$
In particular,
$$
\nu(g)>\nu\left(d_{\delta, i}\left(\partial_{b_{i}}Q_{i}\right)^{p^{e}}\right)+\delta\beta_{\ell+t}-p^e\beta_{\ell+t}=
\nu\left(d_{\delta, i}\left(\partial_{b_{i}}Q_{i}\right)^{p^{e}}\right)+p^e(v-1)\beta_{\ell+t}=\nu_i(g)
$$
(here is where we are using $v>1$). Now, $h$ was defined as $h=Q_{\ell+\omega}$, in other words, $h$ is assumed to have minimal degree among all the polynomials satisfying $\nu(h)>\nu_i(h)$. This is a contradiction. The Proposition is proved. 
\end{proof}

\begin{remark}{\label{Remark12.6}} Let the notation be as in Proposition \ref{Proposition12.5}. Assume, in addition, that the sequence
$\left\{\beta_{\ell+t}\right\}$ is unbounded in $\tilde\g_0$. Then $\nu(h)\notin\tilde\g_0$.

Assume that $char\ k_\nu=char\ K$,
\begin{equation}
\ell=1\text{ and }\alpha_t=1\quad\text{for all }t\in\mathbb{N}.\label{tag12.30}
\end{equation}
In particular, $b_t=1$ for all $t\in\mathbb{N}$ and $b_\infty=1$.

Let $e_\omega$ be the integer $e$ of Proposition \ref{Proposition12.5}.

We have $p^{e_\omega}=\delta$. We claim that
\begin{equation}
\label{tag12.31}
h\in K\left[x^\delta\right];
\end{equation}
in particular, for all strictly positive integers $b'<\delta$ we have
\begin{equation}
\partial_{b'}h=0.\label{tag12.32}
\end{equation}
by Remark \ref{diffkills}. We prove (\ref{tag12.31}) by contradiction. Assume the contrary. Let $e'$ denote the greatest non-negative integer such that $h\in K\left[x^{p^{e'}}\right]$; by assumption, $e'<e_\omega$. Then $h$ involves at least one monomial $x^j$ such that $j$ is of the form $j=up^{e'}$ where $u$ is not divisible by $p$. We have $\deg_x\partial_{p^{e'}}h<\deg_xh$, so there exists $t_0\in\mathbb N$ such that
\begin{equation}
\nu_{t_0}\left(\partial_{p^{e'}}h\right)=\nu\left(\partial_{p^{e'}}h\right).\label{eq:stablet}
\end{equation}
Take an integer $t>t_0$. Let $\sum\limits_{p^{e'+1}\ \doesnotdivide \ j}c_{j,t}Q_t^j$ denote the sum of all those monomials appearing in the $t$-standard expansion of $h$ whose exponent $j$ is not divisible by $p^{e'+1}$. Note that by (\ref{tag12.30}) we have $c_{j,t}\in K$ for all pairs $(j,t)$. By Remark \ref{diffkills}, the operator $\partial_{p^{e'}}$ annihilates all the monomials whose exponents are divisible by $p^{e'+1}$. Thus
\begin{equation}\label{eq:derivativepeprime}
\partial_{p^{e'}}h=\partial_{p^{e'}}\left(\sum\limits_{p^{e'+1}\ \doesnotdivide\ j}c_{j,t}Q_t^j\right)=
\sum\limits_{p^{e'+1}\ \doesnotdivide\ j}c_{j,t}\binom{j}{p^{e'}}Q_t^{j-p^{e'}}.
\end{equation}
Formulas (\ref{eq:stablet}) and (\ref{eq:derivativepeprime}) imply that the $t$-standard expansion of $h$ contains a monomial of the form $c_{p^{e'},t}Q_t^{p^{e'}}$ and that for each $j$ with $p^{e'+1}\ \not|\ j$ we have
$$
\nu_{t_0}\left(c_{j,t}Q_t^j\right)\ge\nu\left(c_{p^{e'},t}Q_t^{p^{e'}}\right).
$$
Hence
\begin{equation}
\nu\left(c_{j,t}Q_t^j\right)>\nu\left(c_{p^{e'},t}Q_t^{p^{e'}}\right)\text{ for all }j\text{ with }p^{e'+1}\ \not|\ j\text{ and }j>p^{e'}.\label{eq:greatervalue}
\end{equation}
We obtain that for all $t$ sufficiently large the $t$-standard expansion of $h$ contains a monomial of the form
$c_{p^{e'},t}Q_t^{p^{e'}}$ and all the other monomials not divisible by $Q_t^{p^{e'+1}}$ have values strictly greater than
$\nu\left(c_{p^{e'},t}Q_t^{p^{e'}}\right)$.

Then for all $t'>t$ we have $\nu\left(c_{p^{e'},t}\right)=\nu\left(c_{p^{e'},t'}\right)$. Choosing $t'$ sufficiently large, we obtain
$\nu\left(c_{p^{e'},t'}Q_{t'}^{p^{e'}}\right)<\nu\left(c_{\delta,t'}Q_{t'}^\delta\right)$, which contradicts the definition of $\delta$. This completes the proof of (\ref{tag12.31}) and (\ref{tag12.32}).

In fact, by a similar argument this statement can be proved not only for $h$, but for any polynomial satisfying the strict inequalities (\ref{tag12.28}).
\end{remark}

\begin{remark}{\label{Remark12.7}} Keep the assumption that $\left\{\beta_{\ell+t}\right\}$ is unbounded in $\tilde\g_0$, as well as
(\ref{tag12.30}), but now assume that $char\ K=0$ and $char\ k_\nu=p>0$. By studying the coefficient of $Q_t^{\delta-1}$ in the $t$-standard expansion of $h$ for different $t$, one can prove that $\delta=p^e=1$.
\end{remark}

\begin{remark} It was pointed out to us by the referee that the result of Remark \ref{Remark12.6} was generalized in \cite{Nart0}, Theorem 4.11. For every valued field $(K,\nu)$ and every continuous family of iterated augmentations on $K[x]$ such that $\{\beta_{\ell+t}\}$ is unbounded, all the limit key polynomials of the family belong to $K\left[x^{\delta b_\infty}\right]$.

More precisely, Theorem 4.11 of \cite{Nart0} shows that $m:=\delta b_\infty$ is the smallest positive integer such that
$\partial_mh\ne0$. In particular, if $K$ is of characteristic zero, we have $\delta=1$ \cite{Nart0}, Corollary 4.12. This also generalizes the statement of Remark \ref{Remark12.7}.
\end{remark}

\begin{proposition}{\label{Proposition12.8}} Keep the notation and assumptions stated in the beginning of this section. Assume that
$\delta=p^e=1$ in the notation of (\ref{tag12.29}) (this assumption holds automatically if $\ch\ k_\nu=0$). Then the
sequences
\begin{equation}
\left(\nu_i(h)\right)_i\label{tag12.34}
\end{equation}
and
\begin{equation}
\left(\beta_i\right)_i,\label{tag12.35}
\end{equation}
where $i$ runs over the set $\{\ell+t\ |\  t\in\mathbb{N}\}$, are unbounded in $\tilde\g_0$.
\end{proposition}

\begin{proof} Proposition \ref{Proposition11.2} (2) implies that $\nu_i(h)=\beta_i+\nu(d_{1,i})$ and that $\nu(d_{1,i})$ is independent of $i$. Thus to show that the sequence (\ref{tag12.34}) is unbounded in $\tilde\g_0$ it is sufficient to show that (\ref{tag12.35}) is unbounded in $\tilde\g_0$.

Moreover, to prove that (\ref{tag12.35}) is unbounded, it is sufficient to show that the set $\nu(T)$ itself is unbounded in $\tilde\g_0$ (where $T=\left\{Q_\ell+w\ \left|\ w\in K[x],\deg_x(w)<\bar\alpha_\ell\right.\right\}$).

To prove the unboundedness of $\nu(T)$, we will start with the (not necessarily complete) set of key polynomials $\mathbf Q_{\ell+1}$ and will define a new set of key polynomials $\mathbf Q_{\ell+1}\cup\left\{\left.\tilde Q_{\ell+t}\ \right|\ t\in\mathbf N\right\}$ with
$\tilde Q_{\ell+t}\in T$ and the sequence $\nu\left(\tilde Q_{\ell+t}\right)$  unbounded in $\tilde\g_0$.

First, let $d^*_{1,\ell}\in K[x]$ denote a polynomial such that $\ini_\nu d^*_{1,\ell}\ini_\nu d_{1,\ell}=1$ in $G_\nu$. According to Lemma \ref{el4.4} we may choose $d^*_{1,\ell}$ to be of degree strictly less than $\deg_x Q_\ell =\deg_x Q_{\ell+t}$. We have
$\nu_\ell\left(d^*_{1,\ell}\right)=\nu\left(d^*_{1,\ell}\right)$ by Proposition \ref{Proposition9.26b}, hence $\nu_i\left(d^*_{1,\ell}\right)=\nu\left(d^*_{1,\ell}\right)$ for all $i\ge\ell$ by Proposition \ref{nuiisavaluation} (2).

Indeed, by Lemma \ref{qQ+r} applied with $s=2$ the graded algebra $G_{<\bar\alpha_\ell}$ is closed under multiplication. Hence
$\ini_\ell(d^*_{1,\ell}h)=\ini_\ell d^*_{1,\ell}\ini_\ell h$ and $\delta_i(d^*_{1,\ell}h)=1$ for all $i$ of the form $\ell+t$, $t\in\mathbb N$.

Thus multiplying $h$ by $d^*_{1,\ell}$ does not change the problem in the sense that
$$
\nu\left(d^*_{1,\ell}h\right)>\nu_i\left(d^*_{1,\ell}h\right)\quad\text{ for all }i
$$
(the new polynomial $d^*_{1,\ell}h$ may no longer have minimal degree with respect to this property, but we will not use the minimality of degree in the rest of the proof). Therefore we may assume that
\begin{equation}
\ini_{\nu}d_{1,i}=\ini_id_{1,i}=1\text{ for all }i\text{ of the form }\ell+t,\ t\in\mathbb{N}.\label{eq:ddelta=1}
\end{equation}
We will now construct polynomials $\tilde Q_{\ell+t}=Q_\ell+w_t$ such that $\deg_xw_t<\deg_xQ_\ell$ and the sequence
$\left(\nu\left(\tilde Q_{\ell+t}\right)\right)_t$ is strictly increasing and unbounded in $\tilde\g_0$.

We have $\deg_{\bar Q_\ell}\ini_\ell h=\delta=1$, so
\begin{equation}
\ini_\ell h=\ini_\ell\left(Q_\ell+d_{0,\ell}\right).\label{eq:binomial}
\end{equation}
In view of (\ref{tag12.28}), we have $d_{0,\ell}\ne0$ and
\begin{equation}
\nu(Q_\ell)<\nu\left(Q_\ell+d_{0,\ell}\right),\label{tag12.39}
\end{equation}
hence
\begin{equation}
\ini_\nu Q_\ell=-\ini_\nu d_{0,\ell}.\label{eq:iniequal}
\end{equation}
We have $Q_\ell+d_{0,\ell}\in T$. Put
$$
w_1:=d_{0,\ell}
$$
and $\tilde Q_{\ell+1}=Q_\ell+w_1$. From now till the end of the proof, for every object $X$ pertaining to the key polynomials $Q_i$ let us use the notation $\tilde X$ for the analogous object pertaining to the key polynomials $\tilde Q_i$ that we are about to construct.

Let $A$ denote the $\mathbb Z$-subalgebra of $K[x]$ generated by $x$ and the finitely many coefficients of the polynomial $Q_\ell$. The ring $A$ is noetherian. We have $\tilde Q_{\ell+1}\in A$. Let
\begin{equation}
h=\sum\limits_{j=0}^{s_\ell}\tilde d_{j,\ell+1}\tilde Q_{\ell+1}^j\label{tildel+1standard}
\end{equation}
be the $(\ell+1)$-standard expansion with respect to $\tilde Q_{\ell+1}$. Since $\delta=1$, we have $\nu\left(d_{j,\ell}Q_\ell^j\right)>\beta_\ell$ whenever $j>1$. Hence $\nu\left(\tilde d_{j,\ell+1}\tilde Q_{\ell+1}^j\right)>\tilde\beta_{\ell+1}$ whenever $j>1$. The expression (\ref{tildel+1standard}) contains the term $\tilde Q_{\ell+1}$ (with coefficient equal to 1 modulo terms of strictly positive value).

Since the set $\nu(T)$ does not have a maximal element, there exists $i=\ell+t$, $t\in\mathbb N$, such that
\begin{equation}
\nu_i(h)=\beta_i>\tilde\beta_{\ell+1}\ge\tilde\nu_{\ell+1}(h).
\label{eq:betai>betatildel+1}
\end{equation}
Combining (\ref{eq:betai>betatildel+1}) with (\ref{tag12.28}), we obtain
$\nu(h)>\tilde\nu_{\ell+1}(h)$.

We now iterate the procedure with $Q_\ell$ replaced by $\tilde Q_{\ell+1}$. Precisely, assume that $w_1,\dots,w_t$ and $\tilde Q_{\ell+q}=Q_\ell+w_q\in
A$, $q\in\{1\do t\}$ are already constructed,
\begin{equation}
\nu(h)>\tilde\nu_{\ell+1}(h)\label{eq:nu>tildenul+t}
\end{equation}
and $\tilde\delta_{\ell+t}=\tilde d_{1,\ell+t}=1$. By Proposition \ref{Proposition11.2} (2), applied to the newly constructed set
$$
\mathbf Q_{\ell+1}\bigcup\left\{\tilde Q_{\ell+q}\right\}_{q\in\{1\do t\}}
$$
of key polynomials, we have
\begin{equation}
\ini_{\ell+t}h=\ini_{\ell+t}\tilde Q_{\ell+t}+\ini_\nu\tilde d_{0,\ell+t}.\label{eq:zl+t+1}
\end{equation}
Note that (\ref{eq:nu>tildenul+t}) and (\ref{eq:zl+t+1}) imply that $\tilde d_{0,\ell+t}\ne0$. We now define
$$
w_{t+1}:=w_t+\tilde d_{0,\ell+t}
$$
and $\tilde Q_{\ell+t+1}=\tilde Q_\ell+w_{t+1}$. We have $\tilde Q_{\ell+t+1}\in A$.

This completes the recursive construction. Notice that all the elements $w_t$ and $\tilde Q_{\ell+t}$ lie in the noetherian ring $A$. Localizing $A$ at the prime ideal $A\cap M_\nu$, we may further assume that $A$ is a local noetherian ring.

\begin{lemma}\label{noboundedsequences} Let $\mu$ be a rank one valuation, centered in a local noetherian domain $(R,M,k)$ (that is, non-negative on $R$ and strictly positive on $M$). Let
$$
\Phi=\mu(R\setminus\{0\})\subset\tilde\g_0.
$$
Then $\Phi$ contains no infinite strictly increasing bounded sequences.
\end{lemma}
\begin{proof}
An infinite ascending sequence $\alpha_1<\alpha_2<\dots$ in $\Phi$, bounded above by an element $\beta\in\Phi$, would give rise to an infinite descending chain of ideals in $\frac R{I_\beta}$, where $I_\beta$ denotes the $\mu$-ideal of $R$ of value $\beta$. Thus it is sufficient to prove that $\frac R{I_\beta}$ has finite length.

Let $\epsilon:=\mu(M)\equiv\min(\Phi\setminus\{0\})$. Since $\mu$ is of rank one, there exists $n\in\mathbb N$ such that $\beta\le
n\epsilon$. Then $M^n\subset I_\beta$, so that there is a surjective map $\frac R{M^n}\twoheadrightarrow\frac R{I_\beta}$. Thus
$\frac R{I_\beta}$ has finite length, as desired.
\end{proof}
Coming back to the proof of the Proposition, let $H=\{a\in A\ |\ \nu(a)\notin\tilde\Gamma_0\}$ and
$$
M=\{a\in A\ |\ \nu(a)>0\}.
$$
Applying Lemma \ref{noboundedsequences} to the local noetherian ring $\frac{A_M}{HA_M}$ and using the fact that the sequence
$\beta_i$ is strictly increasing with $i$, we obtain that $\{\beta_i\}$ is unbounded in $\tilde\Gamma_0$, as desired.
\end{proof}
\begin{remark} Take a polynomial $g\in K[x]$ such that $\nu_{\ell+t}(g)<\nu(g)$ for all $t\in\mathbb N$, not necessarily of the smallest degree. Let $\delta:=\delta(g)$ denote the stable value of $\delta_{\ell+t}(g)$ for $t\in\mathbb N$ sufficiently large. Assume that
$p^e=1$ in the notation of (\ref{tag12.29}) (in other words, either $\ch\ k_\nu=0$ or $\ch\ k_\nu=p>0$ and $p\ \not|\ \delta$). For $i=\ell+t$ with $t$ sufficiently large we have $\nu_i(g)=\nu(d_{\delta,i})+\delta\beta_i$ with $\nu(d_{\delta,i})$ independent of $i$. Thus $\nu_i(g)$ is unbounded in $\tilde\g_0$.
\end{remark}

\section{Limit key polynomials in the case of fields of positive (equi-) characteristic}\label{Section7}

In this section, we assume that $\ch\ k_\nu=\ch\ K=p>0$. We assume that we have a set of key polynomials  $\{Q_i\}_{i\in\Lambda}$ such that $\Lambda$ contains at least one limit ordinal. Let $\ell+\omega\in \Lambda$ be a limit ordinal. Assume that the sequence
$\{\nu(Q_{\ell+t})\}_{t\in\mathbb N}$ is bounded in $\tilde\g_0$. The main result of this section, Proposition \ref{Proposition13.1}, says that the polynomial $Q_{\ell+\omega}$ can be chosen in such a way that there exist $i_0=\ell+t_0\in \Lambda$, $t_0\in\mathbb N$ (so that $i_0+=\ell+\omega$), such that the $i_0$-standard expansion of $Q_{\ell+\omega}$ is weakly affine.
\medskip

\begin{remark} If $\ell=0$ and $\deg_xQ_t=1$ for all $t\in\mathbb{N}$, this result was proved by I. Kaplansky. In Kaplansky's terminology $x$ is a limit of a pseudo-convergent sequence $\{\rho_j\}_j$ of algebraic type in $K$, and $Q_{\ell+\omega}$ is a monic polynomial of minimal degree, not fixing the values of $\{\rho_j\}_j$ (see \cite{Ka}, Lemma 10, page 311). In the special case $\ell=0$ and $\deg_xQ_t=1$ for all $t\in\mathbb{N}$, the polynomial $Q_{\ell+\omega}$ is an additive polynomial plus a constant. Kaplansky later called such polynomials ``$p$-polynomials''.
\end{remark}
\medskip

Let the notation be as in \S\ref{setswomaxelts}. Write $\delta=p^{e_0}$ with $e_0\in\mathbb N$ (we know that $\delta$ is a power of $p$ by Proposition \ref{Proposition12.5}).

For a technical reason that will become apparent later, we will assume (without loss of generality) that $\ell$ is not a limit ordinal.

Recall the definition of $\bar\beta$:
$$
\bar\beta=\sup\left\{\left.\nu\left(Q_{\ell+t}\right)\ \right|\ t\in\mathbb N\right\}.
$$
\begin{proposition}{\label{Proposition13.1}} The polynomial $Q_{\ell+\omega}$ can be chosen in such a way that there exists
$i\in\{\ell+t\}_{t\in\mathbb{N}_0}$ such that the $i$-standard expansion $Q_{\ell+\omega}=\sum\limits_{j=0}^\delta c_{j,i}Q_i^j$ of $Q_{\ell+\omega}$ is weakly affine and monic of degree $p^{e_0}$ in $Q_i$, with
\begin{equation}
\bar\beta\le\frac1{p^{e_0}}\nu\left(Q_{\ell+\omega}\right)\label{tag13.3}
\end{equation}
and
\begin{equation}
\nu(c_{j,i})=\left(p^{e_0}-j\right)\bar\beta\quad\text{ whenever }j>0\text{ and }c_{j,i}\ne0.\label{eq:onthecriticalline}
\end{equation}
\end{proposition}

\begin{proof} Let $f=Q_{\ell+\omega}$ be a limit key polynomial; we have
\begin{equation}
\nu(f)>\nu_{\ell+t}(f)\quad\text{ for all }t\in\mathbb N.\label{eq:nuifsmallerthannuf}
\end{equation}
The idea is to gradually modify the polynomial $f$ until we arrive at a limit key polynomial $g$ satisfying the conclusion of the Proposition.

For $i=\ell+t$, $t\in\mathbb N$, let
\begin{equation}
f=\sum\limits_{j=0}^\delta a_{j,i}Q_i^j\label{eq:istandardflastsection}
\end{equation}
denote an $i$-standard expansion of $f$. By Proposition \ref{weakly_affine_lim}, the polynomial $f$ is of degree
$\delta\deg_xQ_\ell $. In particular, we have $a_{\delta,i}=1$.

Choose $i_0\ge \ell$ sufficiently large so that
\begin{equation}
\beta_{i_0}-\alpha_\ell \beta_{\ell-1}>2p^{e_0}(\bar\beta-\beta_{i_0}).\label{tag13.10}
\end{equation}
Before plunging into technical details, let us try to informally outline the strategy of the proof. Let $(\gamma,l)$ be the coordinate system on the plane where the Newton polygon $\Delta_i(f)$ lives. Conclusion (\ref{eq:onthecriticalline}) of the Proposition says that for each monomial $a_{j,i}Q_i^j$, $j>0$, appearing in (\ref{eq:istandardflastsection}) the pair $(\nu(c_{j,i}),j)$ lies on the ``critical line'' $\gamma+l\bar\beta=p^{e_0}\bar\beta$. Assume that for a certain $i\ge i_0$ the $i$-standard expansion (\ref{eq:istandardflastsection}) does not satisfy the conclusion of the Proposition. This means that at least one of the monomials $a_{j,i}Q_i^j$, $j>0$, appearing in (\ref{eq:istandardflastsection}), has  one of the following properties:

(a) the point $(\nu(a_{j,i}),j)$ lies strictly above the critical line
$\gamma+l\bar\beta=p^{e_0}\bar\beta$

(b) the point $(\nu(a_{j,i}),j)$ lies strictly below the critical line

(c) the point $(\nu(a_{j,i}),j)$ lies on the critical line and $j$ is not a power of $p$.
\smallskip

In general, momomials $a_{j,i}Q_i^j$ lying strictly above the critical line cannot be immediately discarded, because they could give rise to monomials on or below the critical line in the $(i+1)$-standard expansion of $f$ after the substitution $Q_i=Q_{i+1}-z_i$. However, this problem does not occur if $\nu\left(a_{j,i}Q_i^j\right)$ is sufficiently large, that is, if our monomial lies far enough above the critical line (this fact, as well as the precise meaning of ``sufficiently large'', is explained in the Remark below). Such monomials are called $i$-superfluous. We define a monomial to be bad if it is not $i$-superfluous and satisfies one the conditions (a)-(c) above. We analyze all three types of bad monomials and show, after three lemmas, that all of the bad monomials disappear for some $i<\ell+\omega$ sufficiently large.

In the rest of the proof of the Proposition we will consider ordinals $i$ satisfying
$$
\ell+\omega>i\ge i_0.
$$
\begin{definition} Take an ordinal $i$, $i_0\le i<\ell+\omega$. A polynomial $g\in K[x]$ is said to be $i${\bf-superfluous} if
\begin{equation}
\nu_i(g)\ge p^{e_0}\bar\beta+\left(p^{e_0}-j\right)\left(\bar\beta-\beta_i\right).
\label{eq:verylargevalue}
\end{equation}
Let $i'$ be an ordinal such that $i\le i'<\ell+\omega$. Take an integer $j$, $1\le j<p^{e_0}$. The monomial $a_{j,i'}Q_{i'}^j$ is said to be $i${\bf-superfluous} if it is $i$-superfluous viewed as an element of $K[x]$.
\end{definition}
The set of all $i$-superfluous polynomials wil be denoted by $SUP_i$.
\begin{remark} A monomial $a_{j,i'}Q_{i'}^j$ is $i$-superfluous if and only if
\begin{equation}
\nu(a_{j,i'})+j\bar\beta\ge2p^{e_0}\bar\beta-p^{e_0}\beta_i.
\label{eq:verylargevaluemon}
\end{equation}
In particular,
\begin{equation}
SUP_i\subset SUP_{i+1}.\label{eq:supiinsupi+1}
\end{equation}
\end{remark}
\begin{remark}{\label{Remark13.2}} Let $i,i'$ be as above. Consider an $i$-superfluous monomial $a_{j,i}Q_i^j$ appearing in the $i$-expansion of $f$. By Proposition \ref{nuiisavaluation} (2), we have $\nu_{i'}\left(a_{j,i}Q_i^j\right)=\nu\left(a_{j,i}Q_i^j\right)$. Hence
\begin{equation}
\nu_{i'}\left(a_{j,i}Q_i^j\right)\ge2p^{e_0}\bar\beta-p^{e_0}\beta_i+j\beta_i-j\bar\beta>p^{e_0}\bar\beta>p^{e_0}\beta_{i'}=\nu_{i'}(f).\label{eq:serpfluousdonotcontribute}
\end{equation}
Then $\ini_{i'}\left(f-a_{j,i}Q_i^j\right)=\ini_{i'}f$ and
$$
\nu_{i'}\left(f-a_{j,i}Q_i^j\right)<\nu\left(f-a_{j,i}Q_i^j\right).
$$
Thus replacing $f$ by $f-a_{j,i}Q_i^j$ does not affect the condition (\ref{eq:nuifsmallerthannuf}); $f-a_{j,i}Q_i^j$ is still a limit key polynomial with index $\ell+\omega$.

For most of the proof of the Proposition we will work with non-superfluous monomials. At the end of the proof we will modify the polynomial $f$ by subtracting all the superfluous monomials from it.

Next, we will compare the $Q_i$-expansions of $f$ for different $i$.
\end{remark}
\begin{definition}{\label{de13.3}} Take an $i\ge i_0$. Let
\begin{equation}
f=\sum\limits_{j=0}^\delta a_{j,i}Q_i^j\label{eq:istandardexpansionfi}
\end{equation}
be the $Q_i$-expansion of $f$ and let $a_{j,i}Q_i^j$ be a monomial appearing in this expansion. We say that $a_{j,i}Q_i^j$ is {\bf bad} if it is not $i$-superfluous and at least one of the following three conditions holds:
\begin{enumerate}
\item[(1)]
\begin{equation}
\nu(a_{j,i})<(p^{e_0}-j)\bar\beta\label{tag13.12}
\end{equation}
\item[(2)] $j$ is not a power of $p$
\item[(3)]
\begin{equation}
\nu(a_{j,i})>(p^{e_0}-j)\bar\beta.\label{tag13.13}
\end{equation}
\end{enumerate}
\end{definition}
In view of Remark \ref{Remark13.2}, to arrive at an $i$-standard expansion (\ref{eq:istandardexpansionfi}) satisfying the conclusion of Proposition \ref{Proposition13.1} it is sufficient to show that it contains no bad monomials, in which case, after subtracting all the
$i$-superfluous monomials from $f$, there will be nothing more to do.

Take $i\ge i_0$. Assume that the $Q_i$-expansion of $f$ contains at least one bad monomial. Let $j(i)$ denote the greatest $j\in\{1,...,p^{e_0}-1\}$ such that the monomial $a_{j,i}Q_i^j$ is bad. Let $j^\bullet(i)$ denote the element $j\in\{1,...,p^{e_0}-1\}$ which minimizes the pair $(\nu(a_{j,i})+j\beta_i,-j)$ in the lexicographical ordering among all the elements of $\{1,...,p^{e_0}-1\}$ such that the monomial $a_{j,i}Q_i^j$ is bad.

To finish the proof of Proposition \ref{Proposition13.1}, we will first prove the following three Lemmas.

\begin{lemma}{\label{el13.4}} Assume that the $Q_{i+1}$-expansion of $f$ contains at least one bad monomial. We have
\begin{equation}
j(i+1)\le j(i)\label{tag13.14}
\end{equation}
and
\begin{equation}
j^\bullet(i+1)\le j^\bullet(i).\label{tag13.15}
\end{equation}
If $j\in\{j(i),j^\bullet(i)\}$ then
\begin{equation}
\ini_{\nu}a_{j,i+1}=\ini_{\nu}a_{j,i}.\label{tag13.16}
\end{equation}
\end{lemma}

\begin{lemma}{\label{el13.5}} If $j\in\{j(i),j^\bullet(i)\}$ then (\ref{tag13.12}) does not hold.
\end{lemma}

\begin{lemma}{\label{el13.6}} If $j=j(i)$ then (\ref{tag13.13}) holds.
\end{lemma}

After proving the three lemmas, we will show that increasing $i$ either eliminates the last bad monomial or strictly decreases $j(i)$. At that point the proof of Proposition \ref{Proposition13.1} will be finished by induction on $j(i)$.
\medskip

Before embarking on the proofs, we make a general remark on comparing the $Q_i$- and the $Q_{i+1}$-expansions of $f$ that will be used in these proofs.
\begin{remark}\label{Euclideansuperfluous} We want to compare the $Q_i$- and the $Q_{i+1}$-expansions of $f$. To do that, write
\begin{equation}
f=\sum\limits_{j'=0}^{\delta}a_{j',i}(Q_{i+1}-z_i)^{j'}.\label{tag13.18}
\end{equation}
Consider an integer $j'$ of the form $j'=p^{e'}$, $e'\in\{1\do e_0-1\}$. Then
\begin{equation}
a_{j',i}(Q_{i+1}-z_i)^{j'}=a_{j',i}Q_{i+1}^{p^{e'}}-a_{j',i}z_i^{p^{e'}}.\label{tag13.19}
\end{equation}
In general, the right hand side of (\ref{tag13.19}) need not be a $Q_{i+1}$-expansion, since $\deg_xa_{j',i}z_i^{p^{e'}}$ may be quite large, even as large or larger than $\deg_xQ_{i+1}^{p^{e'}}$. However, the $Q_{i+1}$-expansion 
of\linebreak$a_{j',i}(Q_{i+1}-z_i)^{p^{e'}}$ is obtained from it by Euclidean division by $Q_{i+1}$, as follows. Write
$$
a_{j',i}z_i^{p^{e'}}=Q_{i+1}g+r\left(a_{j',i}z_i^{j'}\right)
$$
with $\deg_xr\left(a_{j',i}z_i^{j'}\right)<\deg_xQ_{i+1}$ (here $r\left(a_{j',i}z_i^{j'}\right)$ is the remainder of the Euclidean division of $a_{j',i}z_i^{j'}$ by $Q_{i+1}$). By Lemma \ref{qQ+r} and Corollary \ref{nui0Qi+t} we have
$$
\nu_{\ell-1}\left(a_{j',i}z_i^{p^{e'}}\right)=\nu\left(a_{j',i}z_i^{p^{e'}}\right)=\nu_{i+1}\left(a_{j',i}z_i^{p^{e'}}\right),
$$
and
\begin{equation}
\nu_{i+1}(Q_{i+1}g)-\nu\left(a_{j',i}z_i^{p^{e'}}\right)\ge\beta_{i+1}-\alpha_\ell \beta_{\ell-1}.\label{eq:newmonomialssuperfluous}
\end{equation}
Hence $\nu_{i+1}(Q_{i+1}g)-(\beta_{i+1}-\alpha_\ell\beta_{\ell-1})\ge\nu\left(a_{j',i}z_i^{p^{e'}}\right)\ge p^{e_0}\beta_i$. Thus
$Q_{i+1}^{j'}-r\left(a_{j',i}z_i^{j'}\right)$ is a $Q_{i+1}$-expansion and
$\nu_{i+1}\left(a_{j',i}z_i^{j'}-r\left(a_{j',i}z_i^{j'}\right)\right)\ge p^{e_0}\beta_i+(\beta_{i+1}-\alpha_\ell \beta_{\ell-1})$, and hence, using (\ref{tag13.10}), we have
\begin{equation}
a_{j',i}z_i^{j'}-r\left(a_{j',i}z_i^{j'}\right)\in SUP_{i+1}.\label{eq:differencesuperfluous}
\end{equation}
Exactly the same analysis can be carried out for every $j'\in\{1\do p^{e_0-1}\}$, regardless of whether $j'$ is a power of $p$, but the notation is simpler in the $p$-power case which is enough for our purposes.
\end{remark}
\medskip

\noi{\bf Notation.} For $j\in\{0\do\delta-1\}$ write
$$
Sup_{j,i}(f)=\left\{s\in\{j+1,\dots,\delta\}\}\ \left|\ a_{s,i}Q_i^s\text{ is $i$-superfluous}\right.\right\}
$$
and
$$
NS_{j,i}(f)=\{j+1,\dots,\delta\}\setminus Sup_i(f).
$$
\begin{proof}[Proof of Lemma \ref{el13.4}] Let $j=j(i)$. Since $j$ is the greatest element of $NS_{0,i}$ satisfying one of the conditions (1)--(3) of Definition \ref{de13.3}, every $j'\in NS_{j,i}$ is a power of $p$ and satisfies
\begin{equation}
\nu(a_{j',i})=(p^{e_0}-j')\bar\beta.\label{tag13.17}
\end{equation}
We want to analyze the monomial $a_{j,i+1}Q_{i+1}^j$ in the $Q_{i+1}$-expansion of $f$.

Now, terms in (\ref{tag13.18}) with $j'<j$ or $j'\in Sup_{j,i}$ do not affect either $\ini_{\nu}a_{j,i+1}$ or $\ini_{\nu}a_{j,i}$. We claim that the same is true of the terms with $j'\in NS_{j,i}$. Indeed, take a $j'\in NS_{j,i}$. As explained in the beginning of the proof of this Lemma, $j'$ is a power of $p$. Write $j'=p^{e'}$.

By Remark \ref{Euclideansuperfluous} (specifically, (\ref{eq:newmonomialssuperfluous})), (\ref{tag13.10}) and the fact that
$a_{j,i}Q_i^j$ is not $i$-superfluous, all the monomials $dQ_{i+1}^s$, $s>0$, appearing in the $Q_{i+1}$-expansion of
$a_{j',i}z_i^{p^{e'}}$ satisfy
$$
\nu\left(dQ_{i+1}^s\right)>\nu(a_{j,i})+j\beta_{i+1}.
$$
In particular, if such a monomial is of the form $dQ_{i+1}^j$, with $\deg_x d<\bar\alpha_{i+1}$, we have $\nu(d)>\nu(a_{j,i})$. This proves that (\ref{tag13.16}) holds for $j=j(i)$.

The fact that all the new monomials with strictly positive exponents arising from the Euclidean division of
\[
a_{j',i}z_i^{p^{e'}},\qquad j'>j,
\]
are $(i+1)$-superfluous also shows that no bad monomials $a_{j'',i+1}Q_{i+1}^{j''}$ with $j''>j$ appear in the $Q_{i+1}$-expansion of $f$. This proves (\ref{tag13.14}).

The proof of the Lemma in the case $j=j^\bullet(i)$ is very similar to that of $j=j(i)$, except for the following minor change. We can no longer assert that $j'$ is a power of $p$. On the other hand, $j'$ satisfies $\nu(a_{j',i})+j'\beta_i>\nu_i(f)$, which allows us to use similar arguments as in the $j=j(i)$ case. This completes the proof of Lemma \ref{el13.4}.
\end{proof}

\begin{proof}[Proof of Lemma \ref{el13.5}] We argue by contradiction. Suppose that $j=j(i)$ and that (\ref{tag13.12}) holds for this $j$. (\ref{tag13.12}) can be rewritten as $\nu(a_{j,i})+j\bar\beta<p^{e_0}\bar\beta$. Combining this with (\ref{tag13.16}) we obtain that (\ref{tag13.12}) holds with $i$ replaced by $i+1$ and
\begin{equation}
\nu(a_{j,i+1})+j\bar\beta<2p^{e_0}\bar\beta-p^{e_0}\beta_{i+1},\label{tag13.20}
\end{equation}
so the monomial $a_{j,i+1}Q_{i+1}^j$ is also bad and $j(i+1)=j(i)$ in view of (\ref{tag13.14}). By induction on $i'\ge i$ we see that $j(i')$ is independent of $i'$, so the $Q_{i'}$-expansion of $f$ contains a monomial
$a_{j,i'}Q_{i'}^j$ with
\[
\nu\left(a_{j,i'}Q_{i'}^j\right)=\nu(a_{j,i'})+j\beta_{i'}=
\nu(a_{j,i})+j\beta_{i'}<\nu(a_{j,i})+j\bar\beta.
\]
Then $p^{e_0}\beta_{i'}=\nu_{i'}(f_{i'})<\nu(a_{j,i})+j\bar\beta<p^{e_0}\bar\beta$ for all $i'\ge i$, where the equality holds because
$$
p^{e_0}=\delta\in S_{i'}(f).
$$
Therefore the least upper bound of $\{\beta_{i'}\}_{i\le i'<\ell+\omega}$ is bounded above by $\frac1{p^{e_0}}(\nu(a_{j,i})+j\bar\beta)$ and hence is strictly less than $\bar\beta$. This contradicts the definition of
$\bar\beta$.

The proof in the case $j=j^\bullet(i)$ is similar to that with $j=j(i)$ and we omit it.
\end{proof}

\begin{proof}[Proof of Lemma \ref{el13.6}] We argue by contradiction. Assume that $j=j(i)$ and that (\ref{tag13.13}) does not hold. In view of Lemma \ref{el13.5} this implies that
\begin{equation}
\nu(a_{j,i})+j\bar\beta=p^{e_0}\bar\beta.\label{tag13.21}
\end{equation}
Then, by definition of $j(i)$, $j$ is not a power of $p$. Write $j=p^eu$, $u\ge2$ and $p\ \not|\ u$, and
$$
Q_{i+1}=Q_i+z_i.
$$
Lemma \ref{el13.5}, applied to $j^\bullet(i)$, implies that
\begin{equation}
\nu(a_{j',i})+j'\bar\beta\ge p^{e_0}\bar\beta\qquad\text{for all }j'\in\{1,\dots,p^{e_0}\}.\label{tag13.21bis}
\end{equation}
Let $b=p^eb_\infty$. Arguing as in the proof of Lemma \ref{el13.4}, we can show that $j(i)$ and $\nu(a_{j,i})$ remain unchanged as $i$ increases. Moreover, after suitably increasing $\ell$ and $i$, we may assume that $b_\ell=b_i=b_\infty$ (by Proposition \ref{Proposition12.1}) and $\deg_xQ_\ell=\deg_xQ_i$. Let us denote the common value of $\deg_xQ_\ell$ and $\deg_xQ_i$ by
$\bar\alpha_\infty$. By Corollary \ref{onlyone} there is at most one value of $i=\ell+t$, $t\in\mathbb N$, for which $\#I_{i,max}>1$. Hence, increasing $\ell$ and $i$, we may assume, in addition, that
\begin{equation}
I_{i',max}=\{b_\infty\}\quad\text{ for all }i'\ge\ell.\label{eq:Imaxsingleton}
\end{equation}
Finally, increasing $\ell$ and $i$ again, if necessary, and using (\ref{eq:Imaxsingleton}), we may assume that
\begin{equation}
\beta_i-\beta_\ell>\delta b_\infty(\bar\beta-\beta_i).\label{eq:ilarge}
\end{equation}
We continue to use the notation of (\ref{eq:istandardexpansionfi}) for the $i$-expansion of $f_i$.

For each polynomial $a$ with
\begin{equation}
\deg_x a<\deg_{x}Q_i\label{eq:dega<}
\end{equation}
by Proposition \ref{Proposition9.26b} (applied separately to $i$ and then to $i$ replaced by $\ell$ and using the fact that neither $\ell$ nor $i$ are limit ordinals), we have $\nu_\ell(a)=\nu_i(a)=\nu(a)$. For each $b'\in\mathbb N$ and each $i'$, $\ell<i'<\ell+\omega$, we have
\begin{equation}
\begin{split}
\nu(\partial_{b'}a)&=\nu_\ell(\partial_{b'}a)\ge\nu(a)-\frac{b '}{b_\infty}(\beta_\ell-\nu(\partial_{b_\infty}Q_\ell))=\\
=&\nu(a)-\frac{b'}{b_\infty}(\beta_{i'}-\nu(\partial_{b_\infty}Q_{i'}))+\frac{b'}{b_\infty}(\beta_{i'}-\beta_\ell)>\\
>&\nu(a)-\frac{b'}{b_\infty}(\beta_{i'}-\nu(\partial_{b_\infty}Q_{i'}))+\frac{2p^{e_0}b'}{b_\infty}(\bar\beta-\beta_{i'})\label{eq:smalldegree}
\end{split}
\end{equation}
by Proposition \ref{Proposition10.1}.
\begin{remark}\label{derivativeindepedentbinfty} For $i'=\ell+t'$, $t'\in\mathbb N$ (the integer $t'$ is strictly positive here), we trivially have
\begin{equation}
\nu\left(\partial_{b_\infty}Q_{i'}\right)=
\nu\left(Q_{i'}\right)-\left(\beta_{i'}-\nu\left(\partial_{b_\infty}Q_{i'}\right)\right)\label{eq:powerofQis=1}.
\end{equation}
Combining this with (\ref{eq:smalldegree}), applied to $a=z_\ell+\dots+z_{i'-1}$ and $b'=b_\infty$, we obtain
\begin{equation}
\nu(\partial_{b_\infty}Q_{i'})=\nu(\partial_{b_\infty}Q_\ell).\label{eq:partialbiindependent}
\end{equation}
In particular, the quantity $\nu\left(\partial_{b_\infty}Q_{\ell +t'}\right)$ is independent of $t'\in\mathbb N$.
\end{remark}
\begin{remark}\label{derivativeindepedentbinftystrong} By (\ref{eq:Imaxsingleton}), for all $i'>\ell$ and all
$b'\in\{1\do b\}\setminus\{b_\infty\}$ we have
\begin{equation}
\frac{\beta_{i'}-\nu(\partial_{b_\infty}Q_{i'})}{b_\infty}>\frac{\beta_{i'}-\nu(\partial_{b'}Q_{i'})}{b '}.\label{eq:Imaxsinglebis}
\end{equation}
Assume that
\begin{equation}
\frac{\beta_{i'}-\nu(\partial_{b_\infty}Q_{i'})}{b_\infty}-
\frac{\beta_{i'}-\nu(\partial_{b'}Q_{i'})}{b '}<\frac1{b_\infty}(\beta_{i'}-\beta_\ell).\label{eq:Imaxsinglebisquantitative}
\end{equation}
Combining this  with (\ref{eq:smalldegree}), applied to $a=z_\ell+\dots+z_{i'-1}$, we obtain
\begin{equation}
\nu(\partial_{b'}Q_{i'})=\nu(\partial_{b'}Q_\ell).\label{eq:partialbiindependentb'}
\end{equation}
In particular, the quantity $\nu\left(\partial_{b'}Q_{i'}\right)$ is the same for all those $i'$ that satisfy (\ref{eq:Imaxsinglebisquantitative}). In this way, we have generalized Remark \ref{derivativeindepedentbinfty} from the case of the pair $(b_\infty,i')$ to the case of every pair $(b',i')$ satisfying (\ref{eq:Imaxsinglebisquantitative}).
\end{remark}
By Lemma \ref{ABn}, for all $s,t'\in\mathbb N$,  letting $i'=\ell+t'$, we have
\begin{equation}
\partial_b\left(a_{s,i'}Q_{i'}^s\right)=\sum\limits_{\substack{0\le q\le s\\j_0+\dots+j_q=b\\0<j_1\le\dots\le
j_q}}C_s(j_1,\dots, j_q)Q_{i'}^{s-q}\left(\partial_{j_0}a_{s,i'}\right)\prod\limits_{t=1}^q\left(\partial_{j_t}Q_{i'}\right),\label{eq:Leibnizrulefors}
\end{equation}
so
\begin{equation}
\partial_bf=\sum\limits_{s=0}^\delta\sum\limits_{\substack{0\le q\le s\\j_0+\dots+j_q=b\\0<j_1\le\dots\le
j_q}}\left(C_s(j_1,\dots, j_q)Q_{i'}^{s-q}\left(\partial_{j_0}a_{s,i'}\right)\prod\limits_{t=1}^q\left(\partial_{j_t}Q_{i'}\right)\right).\label{eq:doublesum}
\end{equation}
\medskip

\begin{remark}\label{biggerterms} (1) Consider an integer $s\in\{1\do j-1\}\cup Sup_{j,i'}$. We have
\begin{equation}
\nu\left(a_{j,i'}Q_{i'}^j\right)<\nu\left(a_{s,i'}Q_{i'}^s\right).\label{eq:minimalityofj}
\end{equation}
Indeed, (\ref{eq:minimalityofj}) holds for $s\in\{1\do j-1\}$ by (\ref{tag13.21})--(\ref{tag13.21bis}), with $i$ replaced by $i'$, since
$\beta_{i'}<\bar\beta$. It holds for $s\in Sup_{j,i'}$ since, by (\ref{eq:serpfluousdonotcontribute}) (with $i$ replaced by $i'$), each $i'$-superfluous monomial has value strictly greater than
$$
p^{e_0}\bar\beta>\nu\left(a_{j,i'}Q_{i'}^j\right).
$$
(2) By definition of $j=j(i)=j(i')$, every $s\in NS_{j,i'}(f)$ is of the form $s=p^{e'}$. Since $s>j$, we have
$$
e'>e.
$$
By (\ref{Remark10.9}), we have $C_s(j_1,\dots,j_q)=0$ unless $j_0=0$, $q=s=p^{e'}$ (in particular, $s\le b$) and $j_1=\dots=j_q=\frac b{p^{e'}}$.

\noi(3) By (1) and (2), for every {\it non-zero} term
$C_s(j_1,\dots, j_q)Q_{i'}^{s-q}\left(\partial_{j_0}a_{s,{i'}}\right)\prod\limits_{t=1}^q\left(\partial_{j_t}Q_{i'}\right)$ on the right hand side of (\ref{eq:doublesum}), at least one of the following three conditions holds:

\noi(a)$_{i'}$ (\ref{eq:minimalityofj})

\noi(b)$_{i'}$ $s=p^{e'}$, where $p^e<p^{e'}\le b$ and
$$
\{j_0,j_1\do j_q\}=\big(0,\underset{p^{e'}}{\underbrace{\frac b{p^{e'}},\dots,\frac b{p^{e'}}}}\big)
$$
(c)$_{i'}$ $j_0>0$ (this condition holds whenever $s=0$)

\noi(d)$_{i'}$ $s=up^e=j$ and
$$
\{j_0,j_1\do j_q\}=\bigg(0,\underset{p^{e}}{\underbrace{b_\infty,\dots,b_\infty}}\bigg).
$$
\end{remark}
\noi\textbf{Claim.} There exists a strictly increasing infinite sequence $(t_r)_{r\in\mathbb N}\subset\mathbb N$ and $s\in\bigcap\limits_{r=1}^\infty S_{j,i_r}$ (where $i_r=i+t_r$) such that, for all $r\in\mathbb N$, the term
$C_j(\underset{p^e}{\underbrace{b_\infty\do b_\infty}})a_{j,i_r}Q_{i_r}^{j-p^e}\left(\partial_{b_\infty}Q_{i_r}\right)^{p^e}$ has strictly smaller value than all the other terms in the sum on the right hand side of (\ref{eq:doublesum}).
\medskip

\noi\textit{Proof of Claim.}  We have
$C_j(\underset{p^e}{\underbrace{b_\infty\do b_\infty}})=\frac{j!}{p^e!}=\frac{(up^e)!}{p^e!}=u$, so
$\nu(C_j(\underset{p^e}{\underbrace{b_\infty\do b_\infty}}))=0$. Hence, for all $i'$ of the form $i'=i+t'$, $t'\in\mathbb N$, we have
\begin{equation}
\nu\left(C_j(\underset{p^e}{\underbrace{b_\infty\do
b_\infty}})a_{j,i'}Q_{i'}^{j-p^e}\left(\partial_{b_\infty}Q_{i'}\right)^{p^e}\right)
=\nu\left(a_{j,i'}Q_{i'}^j\right)-p^e\left(\beta_{i'}-\nu\left(\partial_{b_\infty}Q_{i'}\right)\right)\label{eq:jpowerofQi}
\end{equation}
(recall that $\nu\left(a_{j,i'}\right)$ and $\nu\left(\partial_{b_\infty}Q_{i'}\right)$ are independent of $i'$).

First, let us prove that all the terms on the right hand side of (\ref{eq:doublesum}) satisfying (a)$_{i'}$ or (c)$_{i'}$ have values strictly greater than $\nu\left(a_{j,i'}Q_{i'}^j\right)-p^e\left(\beta_{i'}-\nu\left(\partial_{b_\infty}Q_{i'}\right)\right)$.

Consider a non-zero term
$C_s(j_1,\dots, j_q)Q_{i'}^{s-q}\left(\partial_{j_0}a_{s,i'}\right)\prod\limits_{t=1}^q\left(\partial_{j_t}Q_{i'}\right)$ appearing on the right hand side of (\ref{eq:doublesum}). For all positive integers $s$ and $b'$, we have
\begin{equation}
\nu\left(\partial_{b'}Q_{i'}^s\right)\ge\nu\left(Q_{i'}^s\right)-\frac{b'}{b_\infty}\left(\beta_{i'}-\nu\left(\partial_{b_\infty}Q_{i'}\right)\right)\label{eq:powerofQi}.
\end{equation}
by Proposition \ref{Proposition10.1} (1). Combining this with (\ref{eq:smalldegree}), we obtain
\begin{equation}
\begin{aligned}
&\nu\left(Q_{i'}^{s-q}\left(\partial_{j_0}a_{s,i'}\right)\prod\limits_{t=1}^q\left(\partial_{j_t}Q_{i'}\right)\right)=
(s-q)\beta_{i'}+\nu\left(\partial_{j_0}a_{s,i'}\right)+\sum\limits_{t=1}^q\nu\left(\partial_{j_t}Q_{i'}\right)\ge\\
&\ge\nu\left(a_{s,i'}Q_{i'}^s\right)-\frac{j_0}{b_\infty}(\beta_{i'}-\nu(\partial_{b_\infty}Q_{i'}))+
\frac{j_0}{b_\infty}(\beta_{i'}-\beta_\ell)-\sum\limits_{t=1}^q\frac{j_t}{b_\infty}(\beta_{i'}-\nu(\partial_{b_\infty}Q_{i'}))=\\
&=\nu\left(a_{s,i'}Q_{i'}^s\right)-p^e(\beta_{i'}-\nu(\partial_{b_\infty}Q_i))+\frac{j_0}{b_\infty}(\beta_{i'}-\beta_\ell).\label{eq:lowerboundl>0}
\end{aligned}
\end{equation}
If (\ref{eq:minimalityofj}) holds, we obtain
\begin{equation}
\nu\left(Q_{i'}^{s-q}\left(\partial_{j_0}a_{s,i'}\right)\prod\limits_{t=1}^q\left(\partial_{j_t}Q_{i'}\right)\right)>
\nu_i\left(a_{j,'i}Q_{i'}^j\right)-p^e\left(\beta_{i'}-\nu\left(\partial_{b_\infty}Q_{i'}\right)\right).
\label{eq:strictapowerofQi+1}
\end{equation}
If (c)$_{i'}$ is satisfied, that is, $j_0>0$, (\ref{eq:strictapowerofQi+1}) holds again. This shows that in order to prove the Claim it is sufficient to restrict attention to the terms on the right hand side of (\ref{eq:doublesum}) satisfying (b)$_{i'}$.

Next, choose a strictly increasing sequence $(t_r)_{r\in\mathbb N}\subset\mathbb N$ such that for each
$s\in\{j+1\do\delta\}$ one of the following conditions holds for all $r\in\mathbb N$:

(1) $s\in S_{j,i_r}$ and (\ref{eq:Imaxsinglebisquantitative}) holds with $i'=i_r$

(2) either $s\in Sup_{j,i_r}$ or (\ref{eq:Imaxsinglebisquantitative}) does not hold with $i'=i_r$
\medskip

We obtain a decomposition $\{j+1\do\delta\}=J_1\coprod J_2$, where (1) holds for $s\in J_1$ and (2) holds for $s\in J_2$. We have $j\in J_1$; in particular, $J_1\ne\emptyset$. The set $J_2$ is non-empty since it contains all the non-powers of $p$ in the set
$\{j+1\do\delta\}$.

We will now show that for $s\in J_2$ and $r\gg0$ the strict inequality (\ref{eq:strictapowerofQi+1}) holds with $i'=i_r$. For $s\in Sup_{j,i_r}$ this has already been shown (and does not require $r$ to be large). Assume that $s\in S_{j,i_r}$ and (\ref{eq:Imaxsinglebisquantitative}) does not hold with $i'=i_r$. Consider a non-zero term
$C_s(j_1,\dots, j_q)Q_{i_r}^{s-q}\left(\partial_{j_0}a_{s,i_r}\right)\prod\limits_{t=1}^q\left(\partial_{j_t}Q_{i_r}\right)$ appearing on the right hand side of (\ref{eq:doublesum}). This term satisfies condition (b)$_{i_r}$ (in particular, it is the uniqe term with the given $s$). Using the notation of (b)$_{i_r}$, we have
\begin{equation}
\begin{aligned}
&\nu\left(a_{s,i_r}\left(\partial_{\frac b{p^{e'}}}Q_{i_r}\right)^{p^{e'}}\right)=
\nu\left(a_{s,i_r}\right)+p^{e'}\nu\left(\partial_{\frac b{p^{e'}}}Q_{i_r}\right)=\\
&=\nu\left(a_{s,i_r}Q_{i_r}^s\right)-p^{e'}\left(\beta_{i_r}-\nu\left(\partial_{\frac b{p^{e'}}}Q_{i_r}\right)\right).\label{eq:lowerboundl>0Claim}
\end{aligned}
\end{equation}
Further, by (\ref{tag13.21}) and (\ref{tag13.21bis}), applied to $i_r$ instead of $i$, we have
\begin{equation}
\nu\left(a_{j,i_r}Q_{i_r}^j\right)-\nu\left(a_{s,i_r}Q_{i_r}^s\right)\le(s-j)(\bar\beta-\beta_{i_r})<\delta(\bar\beta-\beta_{i_r}).\label{eq:errorterms>j}
\end{equation}

Combining (\ref{eq:ilarge}), the negation of (\ref{eq:Imaxsinglebisquantitative}) and
(\ref{eq:lowerboundl>0Claim})--(\ref{eq:errorterms>j}), we obtain
\begin{equation}
\begin{aligned}
&\nu\left(a_{s,i_r}\left(\partial_{\frac b{p^{e'}}}Q_{i_r}\right)^{p^{e'}}\right)\ge
\nu\left(a_{s,i_r}Q_{i_r}^s\right)-p^e(\beta_{i_r}-\nu(\partial_{b_\infty}Q_i))+p^e\left(\beta_{i_r}-\beta_\ell\right)\ge\\
&\ge\nu\left(a_{j,i_r}Q_{i_r}^j\right)-p^e(\beta_{i_r}-\nu(\partial_{b_\infty}Q_i))+p^e\left(\beta_{i_r}-\beta_\ell\right)-\delta b\left(\bar\beta-\beta_{i_r}\right)>\\
&>\nu\left(a_{j,i_r}Q_{i_r}^j\right)-p^e(\beta_{i_r}-\nu(\partial_{b_\infty}Q_{i_r}))=
\nu\left(Q_{i_r}^{j-p^{e}}a_{j,i_r}\left(\partial_{b_\infty}Q_{i_r}\right)^{p^{e}}\right).\label{eq:lowerboundl>0Claimbis}
\end{aligned}
\end{equation}
This shows that to prove the Claim it is sufficient to restrict attention to the terms\linebreak
$C_s(j_1,\dots, j_q)Q_{i_r}^{s-q}\left(\partial_{j_0}a_{s,i_r}\right)\prod\limits_{t=1}^q\left(\partial_{j_t}Q_{i_r}\right)$ appearing on the right hand side of (\ref{eq:doublesum}) such that $s\in J_1$.

Consider one such term. It must satisfy condition (b)$_{i_r}$, so it is completely determined by $s$ and we can write it as
$a_{s,i_r}\left(\partial_{\frac b{p^{e'}}}Q_{i_r}\right)^{p^{e'}}$. By (\ref{eq:Imaxsinglebisquantitative}), Remark \ref{derivativeindepedentbinftystrong} applies to $\partial_{\frac b{p^{e'}}}Q_{i_r}$, in other words,
$\nu\left(\partial_{\frac b{p^{e'}}}Q_{i_r}\right)$ does not depend on $r\in\mathbb N$. Therefore,
$\nu\left(a_{s,i_r}\left(\partial_{\frac b{p^{e'}}}Q_{i_r}\right)^{p^{e'}}\right)$ is independent of $r$.

We will use the formulae (\ref{eq:lowerboundl>0Claim}) and (\ref{eq:errorterms>j}) which are valid for all $s\in\{0\do\delta\}$. Let us take the limit as $r$ tends to $\infty$ separately of each summand on the right hand side of (\ref{eq:lowerboundl>0Claim}). By (\ref{eq:errorterms>j}), we have
$\lim\limits_{r\to\infty}\nu\left(a_{j,i_r}Q_{i_r}^j\right)\le\lim\limits_{r\to\infty}\nu\left(a_{s,i_r}Q_{i_r}^s\right)$. Applying (\ref{eq:Imaxsinglebis}) with $i'=i_r$ for each $r\in\mathbb N$ and passing to the limit as $r$ goes to infinity, we obtain
$$
\lim\limits_{r\to\infty}p^{e}\left(\beta_{i_r}-\nu\left(\partial_{b_\infty}Q_{i_r}\right)\right)\ge
\lim\limits_{r\to\infty}p^{e'}\left(\beta_{i_r}-\nu\left(\partial_{\frac b{p^{e'}}}Q_{i_r}\right)\right).
$$
Thus
\begin{equation}\label{eq:takinglimits}
\begin{aligned}
&\nu\left(a_{s,i_r}\left(\partial_{\frac b{p^{e'}}}Q_{i_r}\right)^{p^{e'}}\right)=
\lim\limits_{r\to\infty}\nu\left(a_{s,i_r}\left(\partial_{\frac b{p^{e'}}}Q_{i_r}\right)^{p^{e'}}\right)=\\
&=\lim\limits_{r\to\infty}\left(\nu\left(a_{s,i_r}Q_{i_r}^s\right)-p^{e'}\left(\beta_{i_r}-\nu\left(\partial_{\frac b{p^{e'}}}Q_{i_r}\right)\right)\right)\ge\\
&\ge\lim\limits_{r\to\infty}\left(\nu\left(a_{j,i_r}Q_{i_r}^j\right)-p^e(\beta_{i_r}-\nu(\partial_{b_\infty}Q_{i_r}))\right)=
\lim\limits_{r\to\infty}\nu\left(Q_{i_r}^{j-p^{e}}a_{j,i_r}\left(\partial_{b_\infty}Q_{i_r}\right)^{p^{e}}\right).
\end{aligned}
\end{equation}
Since $\nu\left(a_{s,i_r}\left(\partial_{\frac b{p^{e'}}}Q_{i_r}\right)^{p^{e'}}\right)$ is independent of $r$ and
$\nu\left(Q_{i_r}^{j-p^{e}}a_{j,i_r}\left(\partial_{b_\infty}Q_{i_r}\right)^{p^{e}}\right)$ is a strictly increasing linear function of $r$, we have $\nu\left(a_{s,i_r}\left(\partial_{\frac b{p^{e'}}}Q_{i_r}\right)^{p^{e'}}\right)>
\nu\left(Q_{i_r}^{j-p^{e}}a_{j,i_r}\left(\partial_{b_\infty}Q_{i_r}\right)^{p^{e}}\right)$ for all $r\in\mathbb N$.

This completes the proof of the Claim.
\medskip

By the Claim, for all $r\in\mathbb N$ we have
$\nu(\partial_bf)=\nu\left(Q_{i_r}^{j-p^{e}}a_{j,i_r}\left(\partial_{b_\infty}Q_{i_r}\right)^{p^{e}}\right)$ with the left hand side independent of $r$ and the right hand side --- a strictly increasing linear function of $r$, which gives the desired contradiction. This completes the proof of Lemma \ref{el13.6}.
\end{proof}

Let $j=j(i)$. By Lemma \ref{el13.6} the inequality (\ref{tag13.13}) holds for $j$.

By Lemma \ref{el13.4} (cf. (\ref{tag13.16})), $\ini_{\nu}a_{j,i}$ is independent of $i\geq i_0$.

Since $\bar\beta-\beta_i$ can be made arbitrarily small as $i\to i_0+=\ell+\omega$, by (\ref{tag13.13}), taking $i_1$ sufficiently large, we can ensure that
\begin{equation}
\nu\left(a_{j,i}\right)+j\bar\beta\ge2p^{e_0}\bar\beta-p^{e_0}\beta_{i_1}.\label{tag13.27}
\end{equation}
Take the smallest $i_1$ satisfying (\ref{tag13.27}). By the minimality of $i_1$, Lemma \ref{el13.4} (specifically, (\ref{tag13.16})) and induction on $i'$, $i\le i'<i_1$, we see that the monomial $a_{j,i'}Q_{i'}^j$ remains bad for $i\le i'<i_1$ and that
\begin{equation}
\nu\left(a_{j,i}\right)=\nu\left(a_{j,i_1}\right).\label{tag13.28}
\end{equation}
From (\ref{tag13.27})--(\ref{tag13.28}) we obtain
\begin{equation}
\nu\left(a_{j,i_1}\right)+j\bar\beta\ge2p^{e_0}\bar\beta-p^{e_0}\beta_{i_1}.\label{tag13.29}
\end{equation}
Thus the monomial $a_{j,i_1}Q_{i_1}^j$ is $i_1$-superfluous. Therefore
\begin{equation}
j(i_1)<j(i')\quad\text{ for all }\quad i'<i_1.\label{eq:ji1<ji}
\end{equation}
Since the strict inequality (\ref{eq:ji1<ji}) can occur for at most finitely many values of $i_1$, there exists $i_2<\ell+\omega$ such that $f$ containing no bad monomials. Let $Q_{\ell+\omega}$ be equal to $f$ minus the sum of all the $i_2$-superfluous monomials of $f$. Now, $Q_{\ell+\omega}$ is monic of degree $p^{e_0}\bar\alpha_i$ and satisfies (\ref{tag13.3}). It contains no $i_2$-superfluous or bad monomials, hence it satisfies (\ref{eq:onthecriticalline}); in particular, it is weakly affine. The polynomial $Q_{\ell+\omega}$ satisfies (\ref{eq:nuifsmallerthannuf}) and so is a limit key polynomial by Remark \ref{Remark13.2}. This completes the proof of Proposition \ref{Proposition13.1}.
\end{proof}

\begin{remark}{\label{Remark13.8}} The property that the $i_2$-standard expansion of $Q_{\ell+\omega}$ is weakly affine is not preserved when we pass from $i_2$ to some other ordinal $i_2+t$, $t\in\mathbb N$. However, the above results show that for all $i'$ of the form $i'=i+t$, $t\in\mathbb{N}_0$, $Q_{\ell+\omega}$ is a sum of a weakly affine expansion in $Q_{i'}$ all of whose monomials $a_{j,i'}Q_{i'}^j$ for $j>0$ lie on the critical line $\nu(a_{j,i'})=(p^{e_0}-j)\bar\beta$ and another $i'$-standard expansion of degree strictly less than $p^{e_0}\bar\alpha_\ell$, all of whose monomials are $i'$-superfluous.
\end{remark}

\end{document}